\documentclass[12pt]{amsart}
\usepackage[margin=1in]{geometry}
\usepackage{bm}
\usepackage[dvipsnames,svgnames,x11names]{xcolor}
\usepackage[hyphens]{url}
\usepackage[colorlinks=true,
            linkcolor=Maroon,
            citecolor=blue,
            urlcolor=blue,
            hypertexnames=false,
            linktocpage]{hyperref}
\usepackage{bookmark}
\usepackage{amsmath,thmtools,mathtools,amssymb}
\mathtoolsset{showonlyrefs=true}
\usepackage{graphicx}
\usepackage{tikz}
\usepackage{fancyhdr}
\usepackage{esint}
\usepackage{enumerate}
\usepackage{caption}
\allowdisplaybreaks

\newcommand{\R}{\mathbb{R}}
\newcommand{\Sn}{S^{n-1}}
\newcommand{\ip}[2]{\langle #1,#2\rangle}
\newcommand{\abs}[1]{\left|#1\right|}
\newcommand{\tr}{\operatorname{tr}}
\newcommand{\divnabla}{\operatorname{div}_{\nabla}}
\newcommand{\divPhi}{\operatorname{div}_{\Phi}}
\newcommand{\Ric}{\operatorname{Ric}}
\newcommand{\Rm}{\operatorname{Rm}}
\newcommand{\LPhi}{\mathcal{L}_{\Phi}}

\theoremstyle{plain}
\newtheorem{theorem}{Theorem}[section]
\newtheorem*{theorem*}{Theorem}
\newtheorem{claim}[theorem]{Claim}
\newtheorem{proposition}[theorem]{Proposition}
\newtheorem{lemma}[theorem]{Lemma}
\newtheorem{corollary}[theorem]{Corollary}

\theoremstyle{definition}
\newtheorem{definition}[theorem]{Definition}

\theoremstyle{remark}
\newtheorem{remark}[theorem]{Remark}
\numberwithin{equation}{section}

\newcommand{\eq}[1]{\begin{equation}\begin{alignedat}{2}#1\end{alignedat}\end{equation}}

\begin{document}

\title{Weighted centro-affine Poincar\'e inequalities}
\author[Y. Hu, M. N. Ivaki]{Yingxiang Hu, Mohammad N. Ivaki}

\begin{abstract}
We obtain weighted centro-affine Bochner formulas on spherical caps associated with smooth strictly convex hypersurfaces. As a consequence, we prove weighted Poincar\'e inequalities on caps and on intersections of caps for a class of weights depending on the position vector $X$ of the hypersurface. In the unconditional case, we obtain a centro-affine Poincar\'e inequality with weight $|X|^2$, which is used to prove a Brunn--Minkowski inequality for the $(n+2)$-th dual quermassintegral. 

We also establish an $L_0$-Brunn--Minkowski inequality for the $q$-th dual quermassintegral for $q\in(0,n)$, with equality only for dilates, and an $L_p$-Brunn--Minkowski inequality for $q=n+\alpha$ whenever
\[
0<\alpha\le \frac{2p(1-p)}{2-p},
\]
which in particular covers the range $q\in(n,n+6-4\sqrt{2}]$ for suitable $p\in(0,1)$. These Brunn--Minkowski inequalities imply weighted centro-affine Poincar\'e inequalities and uniqueness results for the $L_{p,q}$-Minkowski problem in the unconditional class. 

Our main contribution is the introduction of a flat logarithmic centro-affine geometry on the positive orthant $(0,\infty)^n$, adapted to the multiplicative structure of the $L_0$-sum. In this geometry, a Bochner formula yields a sharp Poincar\'e inequality, as well as a new proof of the centro-affine Poincar\'e inequality with constant $n$ due to Kolesnikov--Milman, for unconditional bodies and unconditional functions.
\end{abstract}

\maketitle

\tableofcontents

\section{Introduction}

Throughout the paper we assume $n\ge 2$. A convex body in $\R^n$ is a compact convex set with nonempty interior. Let $k\ge 2$. A convex body $K\subset \R^n$ is said to be $C^k_+$ if its boundary is $C^k$-smooth and it has positive Gauss curvature. We always assume that $K$ contains the origin in its interior. Let $h=h_K$ denote the support function of $K$, and define the inverse Gauss map of $K$ by
\eq{
X(x)=Dh(x)=\bar{\nabla} h(x)+h(x)x, \quad x\in \Sn,
}
where $D$ denotes the Euclidean gradient in $\R^n$, and $\Sn$ is the unit sphere equipped with its standard induced round metric $\bar{g}$ and Levi-Civita connection $\bar{\nabla}$.

For a smooth function $f$ on $\Sn$, we write $\bar{\nabla}f$ and $\bar{\nabla}^2f$ for the corresponding gradient and Hessian with respect to $\bar{g}$. The rescaled cone-volume measure of $K$ is defined as
\eq{
dV_K=h\det(\bar{\nabla}^2h+hI)\, dx,
}
where $dx$ denotes the standard measure on $\Sn$. The centro-affine metric of $K$ on $\Sn$ is defined by 
\eq{
g=\frac{1}{h}(\bar{\nabla}^2h+hI).
}
We write $\nabla$ for the torsion-free centro-affine connection.

For $w\in\R^n\setminus\{0\}$, define $\ell_w=\ip{X}{w}$. One of our main results is the following weighted centro-affine Poincar\'e inequality on intersections of caps.

\begin{theorem}\label{thm:intersection-caps-general-weight}
Let $K$ be a $C^2_+$ convex body. Assume that $\omega\in C([0,\infty)^m)\cap C^2((0,\infty)^m)$ and $\omega|_{(0,\infty)^m}>0$. Let
\eq{
\Phi(x)=-\log\omega(\ell_{w_1}(x),\dots,\ell_{w_m}(x)), \quad d\mu_{\Phi}=e^{-\Phi}\, dV_K=\omega(\ell_{w_1},\dots,\ell_{w_m})\, dV_K.
}
Suppose for some constants $\alpha>0$ and $\beta\ge 0$, and for all $t=(t_1,\dots,t_m)\in(0,\infty)^m$
\eq{
\label{eq:matrix-condition-general-weight} 
D^{2}\omega^{\frac{1}{\alpha}}(t)\le 0
}
as a symmetric $m\times m$ matrix, and
\eq{
\label{eq:radial-condition-general-weight} 
\sum_{i=1}^{m} t_i\partial_i\log\omega(t)\ge \beta.
}
Then for every $F\in C^{1}(\Sn)$,
\eq{
\lambda_{\alpha,\beta}\int_{\Omega}(F-\bar{F}_{\Phi})^{2}\,d\mu_{\Phi} \le \int_{\Omega}|\nabla F|_{g}^{2}\,d\mu_{\Phi},
}
where
\eq{
\Omega=\bigcap_{i=1}^m\{\ell_{w_i}>0\}\neq\emptyset, \quad \bar{F}_{\Phi}=\frac{\int_{\Omega}F\,d\mu_{\Phi}}{\int_{\Omega}d\mu_{\Phi}}, \quad \lambda_{\alpha,\beta}=\frac{n-1+\alpha}{n-2+\alpha}(n-2+\beta).
}
\end{theorem}

In the dual Brunn--Minkowski theory, the $q$-th dual quermassintegral is defined by
\eq{
\widetilde{V}_q(K)=\frac{q}{n}\int_K |x|^{q-n}\, dx.
}

As a corollary of recent analytic developments, Sadovsky and Zhang \cite{SZ25}, building on results of Kolesnikov--Milman \cite{KM22}, Kolesnikov--Livshyts \cite{KL21}, and a key ingredient due to Cordero-Erausquin and Rotem \cite{CER23}, established a Brunn--Minkowski inequality for $\widetilde{V}_q$ in the range $1<q<n$ for origin-symmetric convex bodies. The case $q=n$ reduces to the classical Brunn--Minkowski inequality. For related results in the range $q\le 1$, see \cite{XZ22}.

\begin{theorem*}
Let $q\in(1,n)$, and let $K,L\subset \R^n$ be origin-symmetric convex bodies. Then
\eq{
\widetilde{V}_q(K+L)^{\frac{1}{q}}\ge \widetilde{V}_q(K)^{\frac{1}{q}}+\widetilde{V}_q(L)^{\frac{1}{q}}.
}
\end{theorem*}

A function $F:\Sn\to \R$ is called unconditional if it is invariant under all coordinate reflections. A convex body $K\subset\R^n$ is called unconditional if its support function $h_K$ is unconditional.

Our cap inequalities have several consequences in symmetric settings. In particular, in the unconditional class we obtain a weighted centro-affine Poincar\'e inequality with the weight $|X|^2$, and use it to prove the following Brunn--Minkowski inequality.

\begin{theorem}\label{thm:main-dualquermass-q}
Let $K,L\subset \R^n$ be unconditional convex bodies. Then
\eq{
\widetilde{V}_{n+2}(K+L)^{\frac{1}{n+2}}\ge \widetilde{V}_{n+2}(K)^{\frac{1}{n+2}}+\widetilde{V}_{n+2}(L)^{\frac{1}{n+2}}.
}
Moreover, equality holds if and only if $K$ and $L$ are dilates of each other.
\end{theorem}

\emph{Note added:} After this paper was posted on arXiv, Shay Sadovsky brought to our attention that the inequality in \autoref{thm:main-dualquermass-q} is a special case of Hadwiger's more general Brunn--Minkowski inequality \cite[Sec. 2.3]{Had56}.

We present two proofs of \autoref{thm:main-dualquermass-q}. The first is based on the weighted centro-affine Poincar\'e inequality in \autoref{cor:main-unconditional-poincare}, which is derived from \autoref{thm:cap-poincare-from-epsilon} with $\alpha=2$ and $w=(1,\dots,1)$, namely a single-cap version of \autoref{thm:intersection-caps-general-weight} with a power weight. This cap-based approach is essential for our argument: within the weighted centro-affine Bochner framework on the whole sphere, we were not able to obtain effective control of the term involving $-2\nabla^2\log |X|$; see \autoref{rem:full-bochner-weight-pinching}. The second proof is shorter, relies on the Borell--Brascamp--Lieb inequality, and yields the equality characterization.

Let $K,L\subset\R^n$ be unconditional convex bodies. For $0<p\le 1$ and $\lambda\in[0,1]$, define
\eq{
(1-\lambda)\cdot K+_p\lambda\cdot L=\left\{x\in\R^n:\langle x,y\rangle\le\left((1-\lambda)h_K(y)^p+\lambda h_L(y)^p\right)^{\frac{1}{p}}\quad \forall y\in\Sn\right\}.
}
For $p=0$, define
\eq{
(1-\lambda)\cdot K+_0\lambda\cdot L=\left\{x\in\R^n:\langle x,y\rangle\le h_K(y)^{1-\lambda}h_L(y)^\lambda\quad\forall y\in\Sn\right\}.
}

We prove:

\begin{theorem}
\label{thm:unconditional-logBM-Vq}
Let $q\in(0,n)$. If $K,L\subset \R^n$ are unconditional convex bodies, then for every $\lambda\in(0,1)$,
\eq{
\label{eq:unconditional-logBM-Vq} 
\widetilde{V}_q\left((1-\lambda)\cdot K+_0 \lambda\cdot L\right) \ge \widetilde{V}_q(K)^{1-\lambda}\widetilde{V}_q(L)^{\lambda}.
}
Moreover, equality holds if and only if $K$ and $L$ are dilates of one another.
\end{theorem}

The proof for $q\le n$ and the second proof for $q=n+2$ both rely on the Pr\'ekopa--Leindler or Borell--Brascamp--Lieb inequality, but in essentially different coordinate systems. For $q<n$, one uses the well-known trick of logarithmic change of variables on the positive orthant, as in \cite{BL95,Sar15}, which is compatible with the $L_0$-sum and turns the radial density $|x|^{q-n}\, dx$ into a log-concave density. For $q=n+2$, one instead works in the original Euclidean coordinates on a half-space, where the weight $\langle x,w\rangle^2$ is directly suited to the Borell--Brascamp--Lieb inequality; when $w=(1,\ldots,1)$, this half-space weight reproduces the radial weight $|x|^2$ after integration over unconditional sets. Although we do not know of an analogous half-space mechanism for the whole range $q\in(n,n+2)$, these two arguments suggest interpolating between the logarithmic and Euclidean coordinate systems. This leads naturally to $L_p$-type changes of variables on the positive orthant $(0,\infty)^n$, and yields an extension to the range $q\in(n,n+6-4\sqrt{2}]$:

\begin{theorem}
\label{thm:Lp-homogeneous-BM}
Let $K,L\subset\R^n$ be unconditional convex bodies, let $0<p<1$, $q=n+\alpha$, and assume $0<\alpha\le\frac{2p(1-p)}{2-p}$.
Then
\eq{
\widetilde{V}_q(K+_pL)^{\frac{p}{q}}\ge \widetilde{V}_q(K)^{\frac{p}{q}}+\widetilde{V}_q(L)^{\frac{p}{q}}.
}
Moreover, equality holds if and only if $K$ and $L$ are dilates.
\end{theorem}

See the discussion in \autoref{section:q-greater-than-n} for alternative perspectives on the distinction between the cases $q<n$ and $q>n$ through the Reilly approach of \cite{KM22}, which was successfully used to treat the case $q=n$.

We also prove the following weighted Poincar\'e inequality as a consequence of the Brunn--Minkowski inequalities for the $q$-th dual quermassintegral for $q\in(0,n)$. The case $q=n$ is the local Brunn--Minkowski inequality, which was proved in \cite{KM22,Mil23}. 

\begin{theorem}\label{thm:SZ-poincare}
Let $q\in(0,n)$. Assume that $K$ is a $C^2_+$ origin-symmetric convex body and $F\in C^1(\Sn)$ is even. Then
\eq{
(q-1)\int_{\Sn}(F-\bar{F})^2|X|^{q-n}\, dV_K \le \int_{\Sn} |\nabla F|_g^2|X|^{q-n}\, dV_K,
}
where
\eq{
\bar{F}=\frac{\int_{\Sn}F|X|^{q-n}\, dV_K}{\int_{\Sn}|X|^{q-n}\, dV_K}.
}
If $K$ is unconditional and $F$ is unconditional, then the constant $q-1$ improves to $q$.
\end{theorem}

In addition, we prove the following uniqueness results for the $L_{p,q}$-Minkowski problem.

\begin{theorem}
\label{thm:uniqueness-Lpq-unconditional}
Let $f\in C(\Sn)$ be positive and unconditional, and assume either
\begin{enumerate}
\item[\rm(i)] $q=n+2$ and $1\le p<q$, or
\item[\rm(ii)] $q\in(0,n)$ and $0\le p<q$.
\end{enumerate}
Then there exists at most one $C^2_+$ unconditional convex body whose support function $h$ solves
\eq{
|Dh|^{q-n}h^{1-p}\det(\bar{\nabla}^2 h+hI)=f.
}
\end{theorem}

The uniqueness statement in $\rm(i)$ appears to be new even in the case $f=1$ and the pair $(p,q)=(1,n+2)$ for $n\ge 3$. For $f=1$, $n=2$, and $(p,q)=(1,4)$, uniqueness without the unconditional symmetry assumption was established in \cite[Thm. 1.2, Case (3), Subcase (2)]{LW26}. We refer the reader to \cite{And99,BLYZ12,BCD17,KM22,Iva23,IM23,HI24,LW24,CSX24,HYZ25,LW26} and the references therein for further related uniqueness results for the $L_{p,q}$-Minkowski problem. For $p\ge q$, uniqueness without any symmetry assumption is already known; see, for instance, \cite[Thm. B]{XZ22}. Alternatively, it also follows from a very simple maximum principle argument.

In the last section, we introduce the logarithmic centro-affine geometry. The proof of \autoref{thm:unconditional-logBM-Vq}, in the range $0<q\le n$, suggests that centro-affine geometry, although well adapted to the $L_1$-structure, does not fully reflect the multiplicative nature of the $L_0$-sum. The relevant infinitesimal geometry for unconditional bodies is instead obtained after passing to logarithmic coordinates on the positive orthant. More precisely, on
\eq{
\Omega_+=\{x\in\Sn:\ X_i(x)>0,\ i=1,\ldots,n\},
}
we consider the immersion $Y=\log X=(\log X_1,\ldots,\log X_n)$ and the constant transversal vector $\bm{\xi}=(1,\ldots,1)$. This induces a flat equiaffine structure $(g_{\log},\nabla^{\log})$ whose affine volume is $dV_{\log}=e^{-\eta}dV_K$, where $\eta:=\sum_iY_i$.

The logarithmic metric is closely related to the centro-affine metric. Indeed,
\eq{
g_{\log}=g+\sum_{i=1}^n\frac{x_iX_i}{h}d\log X_i\otimes d\log X_i.
}
Since $K$ is unconditional, we have $x_iX_i>0$ on $\Omega_+$, and hence $g_{\log}>g$. Consequently, the Poincar\'e inequality in the logarithmic geometry with weight $e^\eta$ and constant $n$ yields the centro-affine Poincar\'e inequality with the same constant $n$. See \autoref{lem:log-affine-metric}, \autoref{prop:g-log-g-polar-log} and \autoref{prop:g-log-plus-g-polar-log-is-round}.

The logarithmic geometry also combines the geometry of $K$ and its polar body. Let
\eq{
\bm{\xi}^*=\left(\frac{x_1X_1}{h},\ldots,\frac{x_nX_n}{h}\right)
}
denote the conormal field. Its components
\eq{
p_i=\frac{x_iX_i}{h},\quad i=1,\ldots,n,
}
are the diagonal entries of the matrix $\mathcal{P}=\frac{x}{h}\otimes X$, coupling the coordinate $X_i$ of the boundary point of $K$ with the corresponding coordinate $x_i/h$ of the polar point. These functions satisfy
\eq{
\mathcal{L}_0^{\log}p_i+np_i=1,\quad \text{where}\quad \mathcal{L}_0^{\log}:=\Delta_{\log}+g_{\log}(\nabla^{\log}\eta,\nabla^{\log}\cdot).
}
Hence the functions $p_i-\frac{1}{n}$ are eigenfunctions of $\mathcal{L}_0^{\log}$ with eigenvalue $-n$.

We note that, in the unconditional setting, \cite[Sec. 5]{KM16} uses a change of metric on the positive orthant in the derivation of functional inequalities. Here the change of metric emerges from the affine geometry induced by a constant transversal vector. Using the Hessian identity $(\nabla^{\log})^2\eta+ng_{\log}=0$ and a logarithmic Bochner formula, our approach yields
\eq{
n\int_{\Omega_+}(F-\bar{F}_+)^2\,dV_K
\le \int_{\Omega_+}\abs{\nabla^{\log}F}_{g_{\log}}^2\,dV_K
\le \int_{\Omega_+}\abs{\nabla F}_{g}^2\,dV_K,
}
where the first inequality is sharp in the logarithmic geometry. In particular, equality in the first inequality is attained by the functions $p_i-\frac{1}{n}$, which are the analogues of the quadratic spherical harmonics on the round sphere. Moreover, equality in the second inequality is attained only by constant functions. See \autoref{thm:log-proof-unconditional-poincare} and \autoref{thm:log-poincare-equality-case}. This should also be compared with \cite[Thm. 8.3 and Rem. 8.4]{KM22}, where the corresponding unconditional infinitesimal statement is obtained by a Reilly-type method, but without passing through the intermediate inequality above or the equality characterization.

\section{Background}

Define the positive definite $(1,1)$-tensor $\tau=\tau[h]=\bar{\nabla}^2 h+hI$. Then
\eq{
\label{eq:dX-tau-relation} 
dX|_x(v)=\tau(x)v,\quad v\in T_x\Sn.
}
We write $\tau^{ij}$ for the entries of the inverse matrix of $\tau$. 

Let $S_K$ denote the surface area measure of $K$:
\eq{
dS_K=\det(\tau)\, dx=\frac{1}{\mathcal{K}}\, dx=\frac{1}{h}\, dV_K,
}
where $\mathcal{K}$ denotes the Gauss curvature of $\partial K$ at the point $X(x)$.

\begin{lemma}[{\cite[Lem. 2.9]{HLYZ16}}]\label{lem:dual-quermass-centro-affine}
Let $K\subset \R^n$ be a $C^2_+$ convex body containing the origin in its interior. Then for every $q>0$, $\widetilde{V}_q(K)=\frac{1}{n}\int_{\Sn}|X|^{q-n}\, dV_K$.
\end{lemma}

\begin{lemma}
\label{lem:unconditional-convexity}
Let $K\subset\R^n$ be unconditional and convex. If $x\in K$ and $|y_i|\le |x_i|$ for all $i=1,\dots,n$, then $y\in K$.
\end{lemma}

\begin{proof}
By unconditionality, all points $(\sigma_1x_1,\dots,\sigma_nx_n)$, $\sigma_i\in\{-1,1\}$ belong to $K$. Since their convex hull is $\prod_{i=1}^n[-|x_i|,|x_i|]$, this box is contained in $K$.
\end{proof}

\begin{lemma}
\label{lem:sign-xi-Xi}
Let $K\subset \R^n$ be a $C^2_+$ unconditional convex body. Then, for every $x\in\Sn$ and every $i=1,\ldots,n$,
\eq{
X_i(x)>0 \quad\Longleftrightarrow\quad x_i>0.
}
In particular,
\eq{
\{x\in\Sn: X_i(x)>0,\ i=1,\ldots,n\}=\{x\in\Sn: x_i>0,\ i=1,\ldots,n\}.
}
\end{lemma}

\begin{proof}
Let $\{E_i\}$ be the standard coordinate basis of $\R^n$. Let $R_i$ denote reflection in the coordinate hyperplane $\{x_i=0\}$. We first prove that $X_i(x)=0$ implies $x_i=0$. Set $y=X(x)$ and suppose that $y_i=0$. Since $R_iy=y$, both $x$ and $R_ix$ are outer unit normals to the boundary of $K$, $\partial K$, at $y$. Therefore, $R_ix=x$.

We now prove that $X_i(x)>0$ implies $x_i>0$. By reflection symmetry, we then also have that $X_i(x)<0$ implies $x_i<0$. By \autoref{lem:unconditional-convexity}, for every sufficiently small $t>0$ with $t<y_i$, the point $y-tE_i$ belongs to $K$. If $x_i<0$, then
\eq{
\langle y-tE_i,x\rangle=\langle y,x\rangle-tx_i>\langle y,x\rangle=h_K(x),
}
which contradicts the definition of the support function. If $x_i=0$, then
\eq{
\langle y-tE_i,x\rangle=\langle y,x\rangle=h_K(x).
}
Thus the line segment $[y-tE_i,y]$ lies on $\partial K$. This contradicts the strict convexity of $K$. Hence $x_i>0$.

Putting everything together, the claim follows.
\end{proof}

\section{Centro-affine geometry}\label{sec:Background-centro-affine-preliminaries}

We now recall some basics from centro-affine geometry; see \cite{NS94,Mil23} for details. The position vector $X$ induces a torsion-free connection $\nabla$ and a metric $g$ on $\Sn$ as follows:
\eq{
D_UdX(V)=dX(\nabla_UV)-g(U,V)X,\quad U,V\in T\Sn.
}
If $\{e_i\}_{i=1}^{n-1}$ is a local $\bar{g}$-orthonormal frame diagonalizing $\tau$, say $\tau e_i=\lambda_i e_i$, then
\eq{
\label{eq:g-frame} 
g_{ij}=g(e_i,e_j)=\frac{\lambda_i}{h}\delta_{ij}.
}

We also use the conjugate torsion-free connection $\nabla^*$ defined via:
\eq{
\label{eq:conjugacy} 
U(g(V,W))=g(\nabla_UV,W)+g(V,\nabla_U^*W).
}
For a smooth function $F$ on $\Sn$, the centro-affine gradient is defined by
\eq{
g(\nabla F,V)=dF(V),\quad V\in T\Sn.
}
We also define
\eq{
\nabla^2F(U,V)&=U(VF)-(\nabla_UV)F, \\ 
(\nabla^*)^2F(U,V)&=U(VF)-(\nabla_U^*V)F, \\ 
\Delta F&=\tr_g\left((\nabla^*)^2F\right).
}
We record two important identities, which will be used repeatedly:
\eq{
\Ric:=\Ric^{\nabla}=(n-2)g,\quad \nabla^2X+gX=0.
}
The latter identity is a special case of the following useful lemma.

\begin{lemma}
\label{lem:homogeneous-extension}
Let $\tilde{f}\in C^{2}(\R^{n}\setminus\{0\})$ be $k$-homogeneous, and define
\eq{
f(x)=\tilde{f}(X(x)),\quad x\in \Sn.
}
Then for every $x\in \Sn$ and every $v_1,v_2\in T_{x}\Sn$,
\eq{
D^{2}\tilde{f}(X(x))[dX(v_1),dX(v_2)]=\nabla^{2} f(v_1,v_2)+kfg(v_1,v_2).
}
In particular, if $\tilde{f}$ is concave on $\R^{n}$, then $\nabla^2f+kf g\le 0$.
\end{lemma}

\begin{proof}
We first take the directional derivative in the $v_2$-direction:
\eq{
v_2(f)=D\tilde{f}(X)[dX(v_2)].
}
Differentiating once more in the direction $v_1$:
\eq{
v_1(v_2 f)=D^{2}\tilde{f}(X)[dX(v_1),dX(v_2)]+D\tilde{f}(X)[D_{v_1} dX(v_2)].
}
Now using $D_{v_1} dX(v_2)=dX(\nabla_{v_1}v_2)-g(v_1,v_2)X$, and Euler's identity, we obtain
\eq{
D\tilde{f}(X)[D_{v_1} dX(v_2)] &=D\tilde{f}(X)[dX(\nabla_{v_1}v_2)] -g(v_1,v_2)\,D\tilde{f}(X)[X]\\ &=(\nabla_{v_1}v_2)f-kfg(v_1,v_2).
}
Rearranging the terms yields
\eq{
D^{2}\tilde{f}(X)[dX(v_1),dX(v_2)] &=v_1(v_2 f)-(\nabla_{v_1}v_2)f+kfg(v_1,v_2)\\ 
&=\nabla^{2} f(v_1,v_2)+kfg(v_1,v_2).
}
\end{proof}

\begin{lemma}
\label{lem:Delta-bar-nabla}
For every $f\in C^2(\Sn)$, we have
\eq{
\Delta f=\frac{1}{h\det(\tau)}\bar{\nabla}_i\left(h^2\tau^{ij}\det(\tau)\bar{\nabla}_j f\right).
}
\end{lemma}

\begin{proof}
We recall the following identity; see \cite{Mil23}:
\eq{
\label{eq:centro-affine-laplacian} 
\Delta f+(n-1)f
=\tau^{ij}\left(\bar{\nabla}_{ij}^2(hf)+\bar{g}_{ij}hf\right).
}
Therefore
\eq{
\Delta f=h\tau^{ij}\bar{\nabla}_{ij}^2 f+2\tau^{ij}\bar{\nabla}_i h\bar{\nabla}_j f.
}
Moreover, due to the identity $\bar{\nabla}_i (\tau^{ij}\det(\tau))=0$,
\eq{
\bar{\nabla}_i\left(h^2\tau^{ij}\det(\tau)\bar{\nabla}_j f\right)=2h\tau^{ij}\det(\tau)\bar{\nabla}_i h\bar{\nabla}_j f+h^2\tau^{ij}\det(\tau)\bar{\nabla}_{ij}^2 f.
}
Dividing this last equation by $h\det(\tau)$ proves the claim.
\end{proof}

\begin{lemma}
\label{lem:divnabla-divbarnabla}
Let $Z=Z^i\partial_i$ be a smooth vector field on $\Sn$, written in local coordinates with respect to the round metric $\bar{g}$. Then
\eq{
\divnabla Z=\frac{1}{h\det(\tau)}\bar{\nabla}_i\left(h\det(\tau)Z^i\right).
}
\end{lemma}

\begin{proof}
Let $\phi\in C^\infty(\Sn)$. By integration-by-parts with respect to $(g,\nabla,V_K)$,
\eq{
\int_{\Sn}\phi\divnabla Z\, dV_K=-\int_{\Sn}g(\nabla\phi,Z)\, dV_K=-\int_{\Sn}(Z^i\bar{\nabla}_i\phi)h\det(\tau)\, dx.
}
Moreover, applying the integration-by-parts with respect to $(\bar{g}, \bar{\nabla},dx)$, we obtain
\eq{
-\int_{\Sn}(Z^i\bar{\nabla}_i\phi)h\det(\tau)\, dx=\int_{\Sn}\phi\bar{\nabla}_i\left(h\det(\tau)Z^i\right)\, dx.
}
Hence
\eq{
\int_{\Sn}\phi\divnabla Z\, dV_K=\int_{\Sn}\phi\bar{\nabla}_i\left(h\det(\tau)Z^i\right)\, dx.
}
Since this holds for all smooth test functions $\phi$, the claim follows.
\end{proof}

Let $\Phi, F\in C^{\infty}(\Sn)$, and set $d\mu_{\Phi}=e^{-\Phi}\, dV_K$. We define the weighted divergence of a smooth vector field $Z$ and the weighted centro-affine Laplacian of $F$ by
\eq{
\divPhi Z&=\divnabla Z-g(\nabla\Phi,Z),\\ 
\LPhi F&=\Delta F-g(\nabla\Phi,\nabla F)=\divPhi(\nabla F).
}

\begin{lemma}
For every smooth function $u$ and smooth vector field $Z$ on $\Sn$,
\eq{
\label{eq:weighted-ibp} 
\int_{\Sn}u\divPhi Z\,d\mu_{\Phi}=-\int_{\Sn}g(\nabla u,Z)\,d\mu_{\Phi}.
}
In particular,
\eq{
\label{eq:weighted-div-zero} 
\int_{\Sn}\divPhi Z\,d\mu_{\Phi}=0,
}
and for smooth functions $u,v$,
\eq{
\label{eq:weighted-green} 
\int_{\Sn}u\LPhi v\,d\mu_{\Phi}=-\int_{\Sn}g(\nabla u,\nabla v)\,d\mu_{\Phi}.
}
\end{lemma}

\begin{proof}
Let us recall the following unweighted integration-by-parts formula:
\eq{
\label{eq:unweighted-div-ibp} 
\int_{\Sn}f \divnabla Z\, dV_K=-\int_{\Sn}g(\nabla f,Z)\, dV_K.
}
By the definition of the weighted divergence, the left-hand side of \eqref{eq:weighted-ibp} expands as
\eq{
\label{eq:defn-weighted-divergence-proof} 
\int_{\Sn}u\divPhi Z\,d\mu_{\Phi}=\int_{\Sn}ue^{-\Phi}\divnabla Z\, dV_K-\int_{\Sn}u g(\nabla\Phi,Z)e^{-\Phi}\, dV_K.
}

We now apply \eqref{eq:unweighted-div-ibp} with $f=ue^{-\Phi}$:
\eq{
\int_{\Sn}ue^{-\Phi}\divnabla Z\, dV_K=-\int_{\Sn}g\left(\nabla(ue^{-\Phi}),Z\right)\, dV_K.
}
On the other hand, we have
\eq{
g\left(\nabla(ue^{-\Phi}),Z\right)=e^{-\Phi}g(\nabla u,Z)-ue^{-\Phi}g(\nabla\Phi,Z).
}
Hence
\eq{
\int_{\Sn}ue^{-\Phi}\divnabla Z\, dV_K=-\int_{\Sn}e^{-\Phi}g(\nabla u,Z)\, dV_K+\int_{\Sn}u\,e^{-\Phi}g(\nabla\Phi,Z)\, dV_K.
}
Inserting this identity into \eqref{eq:defn-weighted-divergence-proof} and rearranging the terms establishes \eqref{eq:weighted-ibp}. Taking the special case $u\equiv 1$ in this identity then yields \eqref{eq:weighted-div-zero}.

Finally, to obtain \eqref{eq:weighted-green}, it suffices to substitute the  vector field $Z=\nabla v$ into \eqref{eq:weighted-ibp}.
\end{proof}

\begin{lemma}
\label{lem:grad-hess}
For a smooth function $F$ and vector fields $U,V$ on $\Sn$ we have
\begin{align}
g(\nabla_U\nabla F,V)&=(\nabla^*)^2F(U,V),\label{eq:grad-hess-star}\\
g(V,\nabla_U^*\nabla F)&=\nabla^2F(U,V).\label{eq:grad-hess-plain}
\end{align}
Moreover, we have
\eq{
\label{eq:trace-hess} 
\tr_g\left((\nabla Z)\circ(\nabla Z)\right)=\abs{(\nabla^*)^2F}_g^2\quad \text{for }Z=\nabla F.
}
\end{lemma}

\begin{proof}
The conjugacy identity \eqref{eq:conjugacy} states
\eq{
U(VF)=U\left(g(\nabla F,V)\right)=g(\nabla_U\nabla F,V)+g(\nabla F,\nabla_U^*V).
}
Hence
\eq{
g(\nabla_U\nabla F,V)=U(VF)-(\nabla_U^*V)F=(\nabla^*)^2F(U,V),
}
which is \eqref{eq:grad-hess-star}. The proof of the second identity \eqref{eq:grad-hess-plain} follows similarly, with the roles of the two connections interchanged.

To establish \eqref{eq:trace-hess}, we work in a local $g$-orthonormal frame $\{e_i\}_{i=1}^{n-1}$. By definition of the trace we have
\eq{
\tr_g\left((\nabla Z)\circ(\nabla Z)\right)=\sum_{i,j=1}^{n-1}g(\nabla_{e_i}Z,e_j)g(\nabla_{e_j}Z,e_i).
}
Choosing the vector field $Z=\nabla F$ and applying \eqref{eq:grad-hess-star} yields
\eq{
g(\nabla_{e_i}Z,e_j)=g(\nabla_{e_i}\nabla F,e_j)=(\nabla^*)^2F(e_i,e_j).
}
Since the Hessian $(\nabla^*)^2F$ is symmetric, the double sum on the right-hand side is the squared norm of this Hessian:
\eq{
\sum_{i,j=1}^{n-1}g(\nabla_{e_i}Z,e_j)g(\nabla_{e_j}Z,e_i)=\abs{(\nabla^*)^2F}_g^2.
}
\end{proof}

\begin{theorem}
Let $\Phi,F\in C^{\infty}(\Sn)$, and $Z=\nabla F$. Then
\eq{
\label{eq:centro-affine-bochner-pointwise} 
\divPhi(\nabla_ZZ)=Z(\LPhi F)+\left(\Ric+\nabla^2\Phi\right)(Z,Z)+\abs{(\nabla^*)^2F}_g^2.
}
Moreover, integrating against $d\mu_{\Phi}=e^{-\Phi}\, dV_K$ yields
\eq{
\label{eq:main-centro-affine-bochner-integrated} 
\int_{\Sn}(\LPhi F)^2\,d\mu_{\Phi}=\int_{\Sn}\abs{(\nabla^*)^2F}_g^2\,d\mu_{\Phi}+\int_{\Sn}(\Ric+\nabla^2\Phi)(\nabla F,\nabla F)\,d\mu_{\Phi}.
}
\end{theorem}

\begin{proof}
We start from the asymmetric Bochner identity (cf. \cite[Prop. 5.1]{Mil23}):
\eq{
\divnabla(\nabla_ZZ)=Z(\divnabla Z)+\Ric(Z,Z)+\tr_g\left((\nabla Z)\circ(\nabla Z)\right).
}
For $Z=\nabla F$, we have $\divnabla Z=\Delta F$ and, by \eqref{eq:trace-hess},
\eq{
\tr_g\left((\nabla Z)\circ(\nabla Z)\right)=\abs{(\nabla^*)^2F}_g^2.
}
Hence
\eq{
\divnabla(\nabla_ZZ)=Z(\Delta F)+\Ric(Z,Z)+\abs{(\nabla^*)^2F}_g^2
}
and
\eq{
\label{eq:weighted-step-1} 
\divPhi(\nabla_ZZ) &=\divnabla(\nabla_ZZ)-g(\nabla\Phi,\nabla_ZZ) \\ 
&=Z(\Delta F)-g(\nabla\Phi,\nabla_ZZ)+\Ric(Z,Z)+\abs{(\nabla^*)^2F}_g^2.
}

Since $\LPhi F=\Delta F-g(\nabla\Phi,Z)$, we have
\eq{
\label{eq:XLPhi} 
Z(\LPhi F)=Z(\Delta F)-Z\left(g(\nabla\Phi,Z)\right).
}
Due to \eqref{eq:conjugacy},
\eq{
Z\left(g(\nabla\Phi,Z)\right)=g(\nabla_ZZ,\nabla\Phi)+g(Z,\nabla_Z^*\nabla\Phi).
}
Moreover, by the second identity in \autoref{lem:grad-hess}, $g(Z,\nabla_Z^*\nabla\Phi)=\nabla^2\Phi(Z,Z)$. Thus
\eq{
\label{eq:XgYX-2} 
Z\left(g(\nabla\Phi,Z)\right)=g(\nabla\Phi,\nabla_ZZ)+\nabla^2\Phi(Z,Z).
}
Putting \eqref{eq:XLPhi} and \eqref{eq:XgYX-2} together yields
\eq{
\label{eq:key-rewrite} 
Z(\Delta F)-g(\nabla\Phi,\nabla_ZZ)=Z(\LPhi F)+\nabla^2\Phi(Z,Z).
}
Now substituting \eqref{eq:key-rewrite} into \eqref{eq:weighted-step-1} proves \eqref{eq:centro-affine-bochner-pointwise}.

We integrate \eqref{eq:centro-affine-bochner-pointwise} against $d\mu_{\Phi}$ and then use \eqref{eq:weighted-div-zero} to obtain
\eq{
0=\int_{\Sn}Z(\LPhi F)\,d\mu_{\Phi}+\int_{\Sn}(\Ric+\nabla^2\Phi)(Z,Z)\,d\mu_{\Phi}+\int_{\Sn}\abs{(\nabla^*)^2F}_g^2\,d\mu_{\Phi}.
}
By \eqref{eq:weighted-green} with $u=\LPhi F$ and $v=F$,
\eq{
\int_{\Sn}Z(\LPhi F)\,d\mu_{\Phi}=\int_{\Sn}g\left(\nabla(\LPhi F),\nabla F\right)\,d\mu_{\Phi}=-\int_{\Sn}(\LPhi F)^2\,d\mu_{\Phi}.
}
Rearranging the terms, we obtain \eqref{eq:main-centro-affine-bochner-integrated}.
\end{proof}

\begin{remark}
Let $\sigma_k=\sigma_k(\tau)$ denote the $k$-th elementary symmetric function of eigenvalues of $\tau=\tau[h]$. We write
\eq{
\sigma_{k}^{ij}:=\frac{\partial \sigma_{k}}{\partial \tau_{ij}}, \quad \sigma_{k}^{ij,pq}:=\frac{\partial^2 \sigma_{k}}{\partial \tau_{ij}\partial \tau_{pq}},
}
and define $\mathrm{H}=(\nabla^*)^2 f$.
By \cite[Lem. 2.3]{CHG17}, $\sigma_{k}^{ij,pq}\tau[hf]_{pq}$ is $\bar{\nabla}$-divergence-free. Hence
\eq{
\label{bochner-higher-order} 
\int_{\Sn}h\sigma_k^{ij,pq}\tau[hf]_{pq}\left(\tau[hf]_{ij}-f\tau_{ij}\right)\, dx=\int_{\Sn}\sigma_{k}^{ij,pq}\tau[hf]_{pq}\bar\nabla_i\left(h^2\bar\nabla_jf\right)\, dx=0.
}

Due to \eqref{eq:centro-affine-laplacian}, $\tau[hf]_{ij}=h\mathrm{H}_{ij}+f\tau_{ij}$. Thus the identity \eqref{bochner-higher-order} for $k=n-1$ becomes
\eq{
0=\int_{\Sn}h^2\sigma_{n-1}^{ij,pq}\tau[hf]_{pq}\mathrm{H}_{ij}\, dx.
}
Substituting $\tau[hf]_{pq}=h\mathrm{H}_{pq}+f\tau_{pq}$ into this last identity yields
\eq{
0=\int_{\Sn}\left(h^3\sigma_{n-1}^{ij,pq}\mathrm{H}_{ij}\mathrm{H}_{pq}+h^2 f\sigma_{n-1}^{ij,pq}\tau_{pq}\mathrm{H}_{ij}\right)\, dx.
}

Now using the identity $\sigma_{n-1}^{ij,pq}\tau_{pq}=(n-2)\sigma_{n-1}^{ij}$
we obtain
\eq{
\label{eq:bochner-n-pre} 
0=\int_{\Sn}\left(h^3\sigma_{n-1}^{ij,pq}\mathrm{H}_{ij}\mathrm{H}_{pq}+(n-2)h^2 f\sigma_{n-1}^{ij}\mathrm{H}_{ij}\right)\, dx.
}
The second term in \eqref{eq:bochner-n-pre} is simply $(n-2)\int_{\Sn} f\Delta f\, dV_K$. Also note that
\eq{
\sigma_{n-1}^{ij,pq} =\left(\tau^{ij}\tau^{pq} -\tau^{ip}\tau^{jq}\right)\sigma_{n-1}.
}
Therefore
\eq{
h^3\sigma_{n-1}^{ij,pq}\mathrm{H}_{ij}\mathrm{H}_{pq}\, dx =\left((\Delta f)^2-|(\nabla^*)^2f|_g^2\right)\, dV_K.
}
Substituting this into \eqref{eq:bochner-n-pre}, using $h^2\sigma_{n-1}^{ij}\mathrm{H}_{ij}\, dx=\Delta f\, dV_K$, and integrating by parts yields
\eq{
\int_{\Sn}(\Delta f)^2\, dV_K=\int_{\Sn}|(\nabla^*)^2 f|_g^2+(n-2)|\nabla f|_g^2\, dV_K.
}
This last identity, the integrated centro-affine Bochner formula, was proved in \cite{Opo15,Mil23}. Hence, \eqref{bochner-higher-order} may be regarded as a higher-order integrated Bochner formula.

It is worth noting that the integrated centro-affine Bochner formula may also be recovered variationally from the identity
\eq{
\int_{\Sn}\Delta_t f\,dV_t=0,
}
valid for every $t$ with $|t|$ sufficiently small, where
\eq{
h_t&=e^{tf}h_K,\quad K_t&&=\text{the convex body with support function }h_t,\\ \Delta_t&=\Delta_{K_t},\quad dV_t&&=dV_{K_t}.
}
Indeed, one computes
\eq{
\dot{dV}_t\big|_{t=0}&=(\Delta_K f+nf)\, dV_K,\\ 
\dot{\Delta}_t\big|_{t=0}f &=-\left|(\nabla_K^*)^2f\right|_{g_K}^2+2|\nabla_K f|_{g_K}^2.
}
Then differentiating the identity $\int_{\Sn}\Delta_t f\,dV_t=0$ at $t=0$ and integrating by parts yields
\eq{
\int_{\Sn}(\Delta_K f)^2\, dV_K =\int_{\Sn}\left|(\nabla_K^*)^2f\right|_{g_K}^2\, dV_K+(n-2)\int_{\Sn}|\nabla_K f|_{g_K}^2\, dV_K.
}
We leave the details to the interested reader.
\end{remark}

\subsection{Poincar\'e inequalities on caps}
Let $w\in\R^n\setminus\{0\}$, and recall that $\ell_w=\ip{X}{w}$. We define the open cap
\eq{
\Omega_w=\{x\in\Sn:\ \ell_w(x)>0\}
}
together with its boundary
\eq{
\Sigma_{w}=\partial\Omega_{w}=\{x\in\Sn:\ \ell_w=0\}.
}

For every sufficiently small positive number $0<\varepsilon<\varepsilon_0$, we also consider the corresponding smooth open set and its boundary
\eq{
\Omega_{w,\varepsilon}=\{\ell_w>\varepsilon\}, \quad \Sigma_{w,\varepsilon}=\partial\Omega_{w,\varepsilon}=\{\ell_w=\varepsilon\},
}
where the threshold $\varepsilon_0>0$ is chosen exactly as in \autoref{lem:cap-boundary-smooth-x}.  

Note that $\Omega_{w,\varepsilon}$ is connected: Let $x_0,x_1\in\Omega_{w,\varepsilon}$ and set $p_i=X(x_i)\in\partial K$, $i=0,1$. Then $p_t=(1-t)p_0+tp_1\in K$ and $\langle p_t,w\rangle>\varepsilon$ for all $t\in[0,1]$. Let $q_t=p_t/\|p_t\|_K$, where $\|\cdot\|_K$ denotes the Minkowski functional of $K$. Since $p_t\in K$, we have $\|p_t\|_K\le 1$, and hence $\langle q_t,w\rangle>\varepsilon$. Therefore, the path $t\mapsto X^{-1}(q_t)\subset \Omega_{w,\varepsilon}$ joins $x_0$ to $x_1$.

\begin{lemma}\label{lem:cap-boundary-smooth-x}
Suppose $K$ is $C^{\infty}_+$. There exists $\varepsilon_0>0$ such that, for every $0< \varepsilon<\varepsilon_0$, the level set $\Sigma_{w,\varepsilon}$ is a  smooth hypersurface of $\Sn$.
\end{lemma}

\begin{proof}
We show that $0$ is a regular value of $\ell_w$. Let $x\in \Sigma_w$. If $d\ell_w|_x=0$, then
\eq{
0=d\ell_w|_x(v)=\ip{dX|_x(v)}{w} \quad \text{for all } v\in T_x\Sn.
}
That is, $w$ is parallel to $x$. Therefore $\ell_w(x)=h_K(x)\ip{x}{w}\neq 0$,
which contradicts $\ell_w(x)=0$. By continuity, all sufficiently small $\varepsilon > 0$ are also regular values.
\end{proof}

The $g$-unit and $\bar{g}$-unit outward normals to the hypersurface $\Sigma_{w,\varepsilon}$ are given by
\eq{
\mathbf{N}_{\varepsilon}=-\frac{\nabla \ell_{w}}{|\nabla \ell_{w}|_g}, \quad \mathbf{n}_{\varepsilon}=-\frac{\bar{\nabla} \ell_{w}}{|\bar{\nabla} \ell_{w}|_{\bar{g}}}.
}
Here, ``outward'' means that as $\varepsilon$ decreases, the domains $\Omega_{w,\varepsilon}$ expand in the outward normal direction. For the limiting case $\varepsilon=0$, we simply write $\mathbf{N}=\mathbf{N}_0$ and $\mathbf{n}=\mathbf{n}_0$. 

Note that, in general, the two unit normals $\mathbf{N}_{\varepsilon}$ and $\mathbf{n}_{\varepsilon}$ are not collinear. At each point of $\Sigma_{w,\varepsilon}$ we have the $\bar{g}$-orthogonal decomposition
\eq{
\mathbf{N}_{\varepsilon}=\frac{|\nabla \ell_{w}|_g}{|\bar{\nabla} \ell_{w}|_{\bar{g}}} \mathbf{n}_{\varepsilon}+\mathbf{T}, \quad \mathbf{T}\in T\Sigma_{w,\varepsilon}.
}
Indeed, this follows from the definitions of the following inner products:
\eq{
g(Z,\mathbf{N}_{\varepsilon}) =-\frac{1}{|\nabla\ell_{w}|_g}d\ell_{w}(Z), \quad \bar{g}(Z,\mathbf{n}_{\varepsilon}) =-\frac{1}{|\bar{\nabla}\ell_{w}|_{\bar{g}}}d\ell_{w}(Z).
}
Hence, for every vector field $Z$, we have the following relations along $\Sigma_{w,\varepsilon}$:
\eq{
\label{eq:N-n-relation} 
\bar{g}(Z,\mathbf{n}_{\varepsilon}) =\frac{|\nabla\ell_{w}|_g}{|\bar{\nabla}\ell_{w}|_{\bar{g}}} g(Z,\mathbf{N}_{\varepsilon}), \quad d\ell_{w}\left(\mathbf{N}_{\varepsilon} -\frac{|\nabla \ell_{w}|_g}{|\bar{\nabla} \ell_{w}|_{\bar{g}}} \mathbf{n}_{\varepsilon}\right)=0.
}

Let $\alpha>0$. On $\Omega_w$ we consider the weighted measure associated with the potential
\eq{
\Phi=\Phi_{w,\alpha}=-\alpha\log \ell_w, \quad d\mu_{w,\alpha}=e^{-\Phi_{w,\alpha}}\, dV_K.
}
We also introduce a boundary measure induced by $dV_K$ on the hypersurface $\Sigma_{w,\varepsilon}$:
\eq{
\label{def: sigma-K-eps} 
d\sigma_{K,\varepsilon}=h\det(\tau)\frac{|\nabla \ell_{w}|_g}{|\bar{\nabla} \ell_{w}|_{\bar{g}}}\,d\bar\sigma_{\varepsilon},
}
where $d\bar\sigma_{\varepsilon}$ denotes the hypersurface measure of $\Sigma_{w,\varepsilon}$ with respect to the metric $\bar{g}$. The corresponding weighted boundary measure is then defined by
\eq{
d\sigma_{w,\alpha,\varepsilon}=e^{-\Phi_{w,\alpha}}\,d\sigma_{K,\varepsilon}.
}

\begin{lemma}\label{lem:unweighted-divergence-eps-cap}
For every smooth vector field $Z$ on a neighborhood of $\overline{\Omega_{w,\varepsilon}}$,
\eq{
\int_{\Omega_{w,\varepsilon}} \divnabla Z\, dV_K=\int_{\Sigma_{w,\varepsilon}} g(Z,\mathbf{N}_{\varepsilon})\,d\sigma_{K,\varepsilon}.
}
\end{lemma}

\begin{proof}
By \autoref{lem:divnabla-divbarnabla}, if $Z=Z^i\partial_i$, then $(\divnabla Z)\, dV_K=\bar{\nabla}_i\left(h\det(\tau)Z^i\right)\, dx$. Integrating this over $\Omega_{w,\varepsilon}$ and applying the divergence theorem on $(\Sn,\bar{g},\bar{\nabla})$ yields
\eq{
\label{eq:div-boundary-round} 
\int_{\Omega_{w,\varepsilon}} \divnabla Z\, dV_K = \int_{\Sigma_{w,\varepsilon}} h\det(\tau)\bar{g}(Z,\mathbf{n}_{\varepsilon})\,d\bar\sigma_{\varepsilon}.
}
Due to \eqref{eq:N-n-relation} and \eqref{def: sigma-K-eps}, the claim follows.
\end{proof}

\begin{lemma}\label{lem:unweighted-eps-cap-x}
For all $f\in C^1(\overline{\Omega_{w,\varepsilon}})$ and $\eta\in C^2(\overline{\Omega_{w,\varepsilon}})$ we have
\eq{
\int_{\Omega_{w,\varepsilon}} f\Delta\eta\, dV_K=-\int_{\Omega_{w,\varepsilon}} g(\nabla f,\nabla\eta)\, dV_K+\int_{\Sigma_{w,\varepsilon}} f\eta_{\mathbf{N}_{\varepsilon}}\,d\sigma_{K,\varepsilon},
}
where $\eta_{\mathbf{N}_{\varepsilon}}:=g(\nabla\eta,\mathbf{N}_{\varepsilon})$.
\end{lemma}

\begin{proof}
Apply \autoref{lem:unweighted-divergence-eps-cap} to $Z=f\nabla\eta$.
\end{proof}

\begin{lemma}
\label{lem:weighted-divergence-eps-cap}
For every smooth vector field $Z$ on a neighborhood of $\overline{\Omega_{w,\varepsilon}}$,
\eq{
\int_{\Omega_{w,\varepsilon}} \divPhi Z\,d\mu_{w,\alpha}=\int_{\Sigma_{w,\varepsilon}} g(Z,\mathbf{N}_{\varepsilon})\,d\sigma_{w,\alpha,\varepsilon},
}
where $\divPhi Z:=\divnabla Z-g(\nabla\Phi_{w,\alpha},Z)$.
\end{lemma}

\begin{proof}
We compute
\eq{
\divPhi Z\, d\mu_{w,\alpha} =\left(\divnabla Z-g(\nabla\Phi_{w,\alpha},Z)\right)e^{-\Phi_{w,\alpha}}\, dV_K =\divnabla(e^{-\Phi_{w,\alpha}}Z)\, dV_K.
}
Now we apply \autoref{lem:unweighted-divergence-eps-cap} to the vector field $e^{-\Phi_{w,\alpha}}Z$:
\eq{
\int_{\Omega_{w,\varepsilon}} \divPhi Z\,d\mu_{w,\alpha} =\int_{\Sigma_{w,\varepsilon}} g(e^{-\Phi_{w,\alpha}}Z,\mathbf{N}_{\varepsilon})\,d\sigma_{K,\varepsilon} =\int_{\Sigma_{w,\varepsilon}} g(Z,\mathbf{N}_{\varepsilon})\,d\sigma_{w,\alpha,\varepsilon}.
}
\end{proof}

\begin{lemma}\label{lem:LPhi-eps-cap}
For all $f\in C^1(\overline{\Omega_{w,\varepsilon}})$ and $\eta\in C^2(\overline{\Omega_{w,\varepsilon}})$,
\eq{
\label{eq:green-eps-cap} 
\int_{\Omega_{w,\varepsilon}} f\LPhi\eta\,d\mu_{w,\alpha} =-\int_{\Omega_{w,\varepsilon}} g(\nabla f,\nabla\eta)\,d\mu_{w,\alpha}+\int_{\Sigma_{w,\varepsilon}} f\eta_{\mathbf{N}_{\varepsilon}}\,d\sigma_{w,\alpha,\varepsilon},
}
where $\LPhi\eta:=\Delta\eta-g(\nabla\Phi_{w,\alpha},\nabla\eta)$.
\end{lemma}

\begin{proof}
We apply \autoref{lem:unweighted-eps-cap-x} with $fe^{-\Phi_{w,\alpha}}$ in place of $f$:
\eq{
\int_{\Omega_{w,\varepsilon}} f\Delta\eta\,d\mu_{w,\alpha} &=-\int_{\Omega_{w,\varepsilon}} g(\nabla(fe^{-\Phi_{w,\alpha}}),\nabla\eta)\, dV_K+\int_{\Sigma_{w,\varepsilon}} f\eta_{\mathbf{N}_{\varepsilon}}\,d\sigma_{w,\alpha,\varepsilon} \\ 
&=-\int_{\Omega_{w,\varepsilon}} g(\nabla f,\nabla\eta)\,d\mu_{w,\alpha}+\int_{\Omega_{w,\varepsilon}} fg(\nabla\Phi_{w,\alpha},\nabla\eta)\,d\mu_{w,\alpha} \\ 
&\quad+\int_{\Sigma_{w,\varepsilon}} f\eta_{\mathbf{N}_{\varepsilon}}\,d\sigma_{w,\alpha,\varepsilon}.
}
\end{proof}

Let $u\in C^3(\overline{\Omega_{w,\varepsilon}})$ and $Z=\nabla u$. In view of \autoref{lem:weighted-divergence-eps-cap} and \autoref{lem:LPhi-eps-cap} we have
\eq{
\label{eq:divergence-nablaZZ} 
\int_{\Omega_{w,\varepsilon}}\divPhi(\nabla_Z Z)\,d\mu_{w,\alpha}=\int_{\Sigma_{w,\varepsilon}} g(\nabla_Z Z,\mathbf{N}_{\varepsilon})\,d\sigma_{w,\alpha,\varepsilon},
}
and
\eq{
\label{eq:integration-by-parts-LPhi} 
\int_{\Omega_{w,\varepsilon}} Z(\LPhi u)\,d\mu_{w,\alpha}=-\int_{\Omega_{w,\varepsilon}} (\LPhi u)^2\,d\mu_{w,\alpha}+\int_{\Sigma_{w,\varepsilon}} (\LPhi u)\,u_{\mathbf{N}_{\varepsilon}}\,d\sigma_{w,\alpha,\varepsilon}.
}

\begin{definition}
For tangent vector fields $U,V$ along $\Sigma_{w,\varepsilon}$, define
\eq{
\mathrm{II}^*_{\varepsilon}(U,V)=g(U,\nabla_V^*\mathbf{N}_{\varepsilon}).
}
\end{definition}

\begin{theorem}\label{thm:R-type-Neumann-eps}
Let $u\in C^3(\overline{\Omega_{w,\varepsilon}})$ satisfy $u_{\mathbf{N}_{\varepsilon}}=g(\nabla u,\mathbf{N}_{\varepsilon})=0$ on $\Sigma_{w,\varepsilon}$.
Then
\eq{
&\int_{\Omega_{w,\varepsilon}} (\LPhi u)^2-|\left(\nabla^*\right)^2u|_g^2-(\Ric+\nabla^2\Phi_{w,\alpha})(\nabla u,\nabla u)\,d\mu_{w,\alpha} \\ 
&\quad=\int_{\Sigma_{w,\varepsilon}} \mathrm{II}^*_{\varepsilon}(\nabla u,\nabla u)\,d\sigma_{w,\alpha,\varepsilon}.
}
\end{theorem}

\begin{proof}
Let $Z=\nabla u$. By the weighted Bochner identity \eqref{eq:centro-affine-bochner-pointwise},
\eq{
\divPhi(\nabla_Z Z)=Z(\LPhi u)+(\Ric+\nabla^2\Phi_{w,\alpha})(Z,Z)+|\left(\nabla^*\right)^2u|_g^2.
}
Integrating over $\Omega_{w,\varepsilon}$, using \eqref{eq:divergence-nablaZZ} and \eqref{eq:integration-by-parts-LPhi}, after rearranging we obtain
\eq{
&\int_{\Omega_{w,\varepsilon}} (\LPhi u)^2-|\left(\nabla^*\right)^2u|_g^2-(\Ric+\nabla^2\Phi_{w,\alpha})(\nabla u,\nabla u)\,d\mu_{w,\alpha} \\ 
&\quad=-\int_{\Sigma_{w,\varepsilon}} g(\nabla_Z Z,\mathbf{N}_{\varepsilon})\,d\sigma_{w,\alpha,\varepsilon}.
}

Since $g(Z,\mathbf{N}_{\varepsilon})=0$ on $\Sigma_{w,\varepsilon}$, $Z$ is tangent to $\Sigma_{w,\varepsilon}$. Due to the conjugacy identity,
\eq{
Z\left(g(Z,\mathbf{N}_{\varepsilon})\right)=g(\nabla_Z Z,\mathbf{N}_{\varepsilon})+g(Z,\nabla_Z^*\mathbf{N}_{\varepsilon}).
}
Hence $-g(\nabla_Z Z,\mathbf{N}_{\varepsilon})=g(Z,\nabla_Z^*\mathbf{N}_{\varepsilon})=\mathrm{II}^*_{\varepsilon}(Z,Z)$.
\end{proof}

\begin{lemma}\label{lem:boundary-sign-eps}
For every tangent vector $Z\in T\Sigma_{w,\varepsilon}$, we have $\mathrm{II}^*_{\varepsilon}(Z,Z)=\varepsilon\frac{|Z|_g^2}{|\nabla \ell_w|_g}$.
\end{lemma}

\begin{proof}
We have $Z(\ell_w)=g(\nabla \ell_w,Z)=0$. Hence
\eq{
\mathrm{II}^*_{\varepsilon}(Z,Z)=g(Z,\nabla_Z^*\mathbf{N}_{\varepsilon}) =-\frac{g(Z,\nabla_Z^*\nabla \ell_w)}{|\nabla \ell_w|_g}.
}
By \autoref{lem:grad-hess}, $g(Z,\nabla_Z^*\nabla \ell_w)=\nabla^2 \ell_w(Z,Z)$. Using $\nabla^2 \ell_w=-\ell_w g$ on $\Sigma_{w,\varepsilon}$, where $\ell_w=\varepsilon$, we obtain $g(Z,\nabla_Z^*\nabla \ell_w)=-\varepsilon |Z|_g^2$.
\end{proof}

\begin{lemma}\label{lem:BE-cap}
On $\Omega_w$ we have
\eq{
\label{eq:BE-cap} 
\Ric+\nabla^2\Phi_{w,\alpha}-\frac{1}{\alpha} d\Phi_{w,\alpha}\otimes d\Phi_{w,\alpha}=(n-2+\alpha)g.
}
\end{lemma}

\begin{proof}
Recall $\nabla^2 \ell_w=-\ell_w g$. Therefore
\eq{
\nabla^2(\log \ell_w)=\frac{1}{\ell_w}\nabla^2 \ell_w-\frac{1}{\ell_w^2}d \ell_w\otimes d \ell_w=-g-d\log \ell_w\otimes d\log \ell_w,
}
and $\nabla^2\Phi_{w,\alpha}=\alpha g+\alpha d\log \ell_w\otimes d\log \ell_w$.
We also have
\eq{
\frac{1}{\alpha} d\Phi_{w,\alpha}\otimes d\Phi_{w,\alpha}=\alpha d\log \ell_w\otimes d\log \ell_w.
}
Since $\Ric=(n-2)g$, identity \eqref{eq:BE-cap} follows.
\end{proof}

\begin{lemma}\label{lem:finite-N-BE}
For every smooth function $F$ on $\Omega_w$,
\eq{
\label{eq:finite-N-cap} 
&\abs{(\nabla^*)^2F}_g^2+\nabla^2\Phi_{w,\alpha}(\nabla F,\nabla F)\\ 
& \ge \frac{1}{n+\alpha-1}(\mathcal{L}_{\Phi}F)^2+\left(\nabla^2\Phi_{w,\alpha}-\frac{1}{\alpha}d\Phi_{w,\alpha}\otimes d\Phi_{w,\alpha}\right)(\nabla F,\nabla F).
}
\end{lemma}

\begin{proof}
Due to the Cauchy--Schwarz inequality, $\abs{(\nabla^*)^2F}_g^2 \ge \frac{1}{n-1}(\Delta F)^2$. Now let $a=\mathcal{L}_{\Phi}F$, $b=g(\nabla\Phi_{w,\alpha},\nabla F)$, and $N=n+\alpha-1$. Then $\Delta F=a+b$ and
\eq{
\frac{1}{n-1}(a+b)^2\ge \frac{1}{N}a^2-\frac{1}{N-n+1}b^2=\frac{1}{n+\alpha-1}a^2-\frac{1}{\alpha}b^2.
}
Hence
\eq{
\abs{(\nabla^*)^2F}_g^2 \ge \frac{1}{n+\alpha-1}(\mathcal{L}_{\Phi}F)^2-\frac{1}{\alpha}g(\nabla\Phi_{w,\alpha},\nabla F)^2.
}
\end{proof}

\begin{corollary}
\label{cor:spectral-eps}
Let $u\in C^3(\overline{\Omega_{w,\varepsilon}})$ satisfy $u_{\mathbf{N}_{\varepsilon}}=0$ on $\Sigma_{w,\varepsilon}$. Then
\eq{
\label{eq:spectral-eps} 
\int_{\Omega_{w,\varepsilon}}(\LPhi u)^2\,d\mu_{w,\alpha} \ge (n+\alpha-1)\int_{\Omega_{w,\varepsilon}}|\nabla u|_g^2\,d\mu_{w,\alpha}.
}
\end{corollary}

\begin{proof}
In view of \autoref{thm:R-type-Neumann-eps} and \autoref{lem:boundary-sign-eps},
\eq{
\int_{\Omega_{w,\varepsilon}}(\LPhi u)^2\,d\mu_{w,\alpha} \ge \int_{\Omega_{w,\varepsilon}}|(\nabla^*)^2u|_g^2\,d\mu_{w,\alpha}+\int_{\Omega_{w,\varepsilon}}(\Ric+\nabla^2\Phi_{w,\alpha})(\nabla u,\nabla u)\,d\mu_{w,\alpha}.
}
Now we apply \autoref{lem:finite-N-BE} and \autoref{lem:BE-cap}:
\eq{
\int_{\Omega_{w,\varepsilon}}(\LPhi u)^2\,d\mu_{w,\alpha} \ge \frac{1}{n+\alpha-1}\int_{\Omega_{w,\varepsilon}}(\LPhi u)^2\,d\mu_{w,\alpha}+ (n-2+\alpha)\int_{\Omega_{w,\varepsilon}}|\nabla u|_g^2\,d\mu_{w,\alpha}.
}
\end{proof}

\begin{lemma}
\label{lem:neumann-eps}Suppose $K$ is $C^{\infty}_+$. For every $0<\varepsilon< \varepsilon_0$ and every $f_{\varepsilon}\in C^\infty(\overline{\Omega_{w,\varepsilon}})
$ with $ \int_{\Omega_{w,\varepsilon}}f_{\varepsilon}\,d\mu_{w,\alpha}=0$, there exists $u_{\varepsilon}\in C^\infty(\overline{\Omega_{w,\varepsilon}})$ such that
\eq{
\LPhi u_{\varepsilon}=-f_{\varepsilon}\quad \text{in }\Omega_{w,\varepsilon},\quad (u_{\varepsilon})_{\mathbf{N}_{\varepsilon}}=0\quad \text{on }\Sigma_{w,\varepsilon}.
}
The solution is unique in the class $\int_{\Omega_{w,\varepsilon}}u_{\varepsilon}\,d\mu_{w,\alpha}=0$.
\end{lemma}

\begin{proof}
Due to \autoref{lem:Delta-bar-nabla},
\eq{
\LPhi u&=\Delta u - g(\nabla\Phi_{w,\alpha},\nabla u)\\
&=\frac{e^{\Phi_{w,\alpha}}}{h\det(\tau)}\bar{\nabla}_i\left(e^{-\Phi_{w,\alpha}}h^2\tau^{ij}\det(\tau)\bar{\nabla}_j u\right).
}
Thus $\LPhi$ is uniformly elliptic on $\Omega_{w,\varepsilon}$. Moreover, the boundary condition $(u_{\varepsilon})_{\mathbf{N}_{\varepsilon}}=0$ is uniformly oblique. Indeed, for the outward normals $\mathbf{N}_{\varepsilon}$ and $\mathbf{n}_{\varepsilon}$ we have
\eq{
\bar{g}(\mathbf{N}_{\varepsilon},\mathbf{n}_{\varepsilon}) =\frac{\bar{g}(\nabla\ell_{w},\bar{\nabla}\ell_{w})} {|\nabla\ell_{w}|_g|\bar{\nabla}\ell_{w}|_{\bar{g}}} =\frac{|\nabla\ell_{w}|_g}{|\bar{\nabla}\ell_{w}|_{\bar{g}}}>0 \quad \text{on } \Sigma_{w,\varepsilon}.
}
Equivalently, the inward vector $-\mathbf{N}_{\varepsilon}$ is uniformly oblique with respect to the inward normal $-\mathbf{n}_{\varepsilon}$, which is the convention used in \cite[Thm. 2.30]{Lie13}.

The compatibility condition for the Neumann problem is $\int_{\Omega_{w,\varepsilon}} f_{\varepsilon}\,d\mu_{w,\alpha}=0$, which holds by assumption. For convenience, we reproduce the argument from \cite[Thm. 2.30]{Lie13} using notation adapted to our setting.

Let $L=\mathcal{L}_{\Phi}$ and, in the convention of \cite[Thm. 2.30]{Lie13}, define the inward oblique boundary operator by
\eq{
Mu=-g\left(\nabla u,\mathbf{N}_{\varepsilon}\right)
}
on the smooth bounded domain $\Omega_{w,\varepsilon}$. Both $L$ and $M$ have no zeroth-order term. Hence, for any $\kappa>0$, the shifted operators $L-\kappa$ and $M-\kappa$ have strictly negative zeroth-order coefficients. Therefore, by \cite[Thm. 2.27]{Lie13},
\eq{
(L-\kappa)u=f \quad \text{in }\Omega_{w,\varepsilon}, \quad (M-\kappa)u=\phi \quad \text{on }\Sigma_{w,\varepsilon}
}
has a unique solution $u\in H^{2+\gamma}\left(\overline{\Omega_{w,\varepsilon}}\right)$ for any $(f,\phi)\in \mathcal{B}:=H^{\gamma}\left(\Omega_{w,\varepsilon}\right)\times H^{1+\gamma}\left(\Sigma_{w,\varepsilon}\right)$.

Now we define the compact operator $\mathcal{T}:\mathcal{B}\to \mathcal{B}$ by
\eq{
\mathcal{T}(f,\phi)=\left(u,u|_{\Sigma_{w,\varepsilon}}\right).
}
As in the proof of \cite[Thm. 2.30]{Lie13}, we also define
\eq{
N_{\mu}&=\left\{(f,\phi)\in \mathcal{B}:\ (f,\phi)=\mu \mathcal{T}(f,\phi)\right\},\\ 
\mathfrak{S}&=\left\{\sigma=-\kappa-\mu:\ \mu\neq 0 \text{ and } \dim N_{\mu}>0\right\}.
}

If $c$ is constant and $(f,\phi)=\left(-\kappa c,-\kappa c\right)$, then $\mathcal{T}(f,\phi)=\left(c,c\right)$. Hence $\left(I+\kappa \mathcal{T}\right)(f,\phi)=(0,0)$ and $\dim N_{-\kappa}\ge 1$. Therefore, $0=-\kappa-(-\kappa)\in \mathfrak{S}$.

Consider the linear functional
\eq{
\Lambda(f,\phi)=\int_{\Omega_{w,\varepsilon}} f\, d\mu_{w,\alpha}+\int_{\Sigma_{w,\varepsilon}} \phi\,d\sigma_{w,\alpha,\varepsilon}.
}
If $(f,\phi)$ is solvable, that is, if $Lu=f$ in $\Omega_{w,\varepsilon}$ and $Mu=\phi$ on $\Sigma_{w,\varepsilon}$, then $g\left(\nabla u,\mathbf{N}_{\varepsilon}\right)=-\phi$, and \autoref{lem:LPhi-eps-cap} gives
\eq{
\int_{\Omega_{w,\varepsilon}} f \, d\mu_{w,\alpha}=\int_{\Omega_{w,\varepsilon}} Lu \, d\mu_{w,\alpha}=\int_{\Sigma_{w,\varepsilon}} g\left(\nabla u,\mathbf{N}_{\varepsilon}\right)\,d\sigma_{w,\alpha,\varepsilon}=-\int_{\Sigma_{w,\varepsilon}} \phi\,d\sigma_{w,\alpha,\varepsilon}.
}
Thus every solvable datum lies in $\ker\Lambda$. On the other hand, by \autoref{lem:LPhi-eps-cap}, the only solutions to $Lu=0$ in $\Omega_{w,\varepsilon}$ and $Mu=0$ on $\Sigma_{w,\varepsilon}$ are the constant functions, hence $I_0=1$. Now \cite[Thm. 2.30]{Lie13} yields a one-dimensional subspace $W_0^*\subset \mathcal{B}^*$ such that the problem $Lu=f$ in $\Omega_{w,\varepsilon}$, $Mu=\phi$ on $\Sigma_{w,\varepsilon}$ has a solution if and only if $(f,\phi)\in \left(W_0^*\right)^{\perp}$. We have $W_0^*=\operatorname{span}\{\Xi\}$ for some nonzero functional $\Xi\in \mathcal{B}^*$. Therefore the space of solvable data is exactly $\ker\Xi$ and $\ker\Xi\subset \ker\Lambda$. Since $\Lambda\neq 0$, this implies that $\ker\Xi=\ker\Lambda$. Consequently, the space of solvable data is exactly $\ker\Lambda$.

In our oblique problem we have $\phi=0$. Hence, due to $\int_{\Omega_{w,\varepsilon}} f_{\varepsilon}\,d\mu_{w,\alpha}=0$, we have
\eq{
\Lambda(-f_{\varepsilon},0)=-\int_{\Omega_{w,\varepsilon}} f_{\varepsilon}\,d\mu_{w,\alpha}=0.
}
Therefore there exists $u_{\varepsilon}\in H^{2+\gamma}\left(\overline{\Omega_{w,\varepsilon}}\right)$ such that
\eq{
\LPhi u_{\varepsilon}=-f_{\varepsilon}\quad \text{in }\Omega_{w,\varepsilon},\quad \left(u_{\varepsilon}\right)_{\mathbf{N}_{\varepsilon}}=0\quad \text{on }\Sigma_{w,\varepsilon}.
}
Since the coefficients, the boundary, and $f_{\varepsilon}$ are smooth, we have $u_{\varepsilon}\in C^\infty\left(\overline{\Omega_{w,\varepsilon}}\right)$. The uniqueness of the solution in the class $\int_{\Omega_{w,\varepsilon}}u_{\varepsilon}\,d\mu_{w,\alpha}=0$ follows from integration by parts and connectedness of $\Omega_{w,\varepsilon}$.
\end{proof}

\begin{theorem}\label{thm:cap-poincare-from-epsilon}
Let $w\in \R^n\setminus\{0\}$, $\alpha>0$, $F\in C^1(\overline{\Omega_w})$. Suppose $K$ is $C^2_+$. Then
\eq{
\label{eq:cap-poincare-final} 
(n+\alpha-1)\int_{\Omega_w}(F-\bar{F}_w)^2\langle X,w\rangle^\alpha\, dV_K \le \int_{\Omega_w}|\nabla F|_g^2\langle X,w\rangle^\alpha\, dV_K,
}
where
\eq{
\bar{F}_w= \frac{\int_{\Omega_w}F\langle X,w\rangle^\alpha\, dV_K} {\int_{\Omega_w}\langle X,w\rangle^\alpha\, dV_K}.
}
\end{theorem}

\begin{proof}
By approximation, we may assume $F\in C^{\infty}(\overline{\Omega_w})$ and $K$ is $C^{\infty}_+$. For each $\varepsilon\in (0,\varepsilon_0)$, we define
\eq{
\bar{F}_{w,\varepsilon}=\frac{\int_{\Omega_{w,\varepsilon}}F\,d\mu_{w,\alpha}}{\int_{\Omega_{w,\varepsilon}}d\mu_{w,\alpha}}, \quad f_{\varepsilon}=F-\bar{F}_{w,\varepsilon}.
}
In view of \autoref{lem:neumann-eps}, there exists $u_{\varepsilon}\in C^\infty(\overline{\Omega_{w,\varepsilon}})$ such that
\eq{
\LPhi u_{\varepsilon}=-f_{\varepsilon}\quad \text{in }\Omega_{w,\varepsilon},\quad (u_{\varepsilon})_{\mathbf{N}_{\varepsilon}}=0\quad \text{on }\Sigma_{w,\varepsilon}.
}
Since $\nabla f_{\varepsilon}=\nabla F$ and $(u_{\varepsilon})_{\mathbf{N}_{\varepsilon}}=0$, \autoref{lem:LPhi-eps-cap} implies that
\eq{
\int_{\Omega_{w,\varepsilon}} f_{\varepsilon}^2\,d\mu_{w,\alpha} =\int_{\Omega_{w,\varepsilon}} g(\nabla u_{\varepsilon},\nabla F)\,d\mu_{w,\alpha}.
}
Hence, by the Cauchy--Schwarz inequality,
\eq{
\label{eq:eps-cs} 
\int_{\Omega_{w,\varepsilon}}(F-\bar{F}_{w,\varepsilon})^2\,d\mu_{w,\alpha} \le \left(\int_{\Omega_{w,\varepsilon}}|\nabla u_{\varepsilon}|_g^2\,d\mu_{w,\alpha}\right)^{1/2} \left(\int_{\Omega_{w,\varepsilon}}|\nabla F|_g^2\,d\mu_{w,\alpha}\right)^{1/2}.
}
On the other hand, by \autoref{cor:spectral-eps},
\eq{
\int_{\Omega_{w,\varepsilon}}(F-\bar{F}_{w,\varepsilon})^2\,d\mu_{w,\alpha}=\int_{\Omega_{w,\varepsilon}}(\LPhi u_{\varepsilon})^2\,d\mu_{w,\alpha} \ge (n+\alpha-1)\int_{\Omega_{w,\varepsilon}}|\nabla u_{\varepsilon}|_g^2\,d\mu_{w,\alpha}.
}
Now substituting into \eqref{eq:eps-cs} yields
\eq{
\label{eq:eps-poincare} 
(n+\alpha-1)\int_{\Omega_{w,\varepsilon}} (F-\bar{F}_{w,\varepsilon})^2\,d\mu_{w,\alpha} \le \int_{\Omega_{w,\varepsilon}} |\nabla F|_g^2\,d\mu_{w,\alpha}.
}
Since $\mathbf{1}_{\Omega_{w,\varepsilon}}\to \mathbf{1}_{\Omega_w}$ as $\varepsilon\to 0$, and $d\mu_{w,\alpha}=\ell_w^\alpha\, dV_K$ is finite on $\Omega_w$, by the dominated convergence theorem we may pass to the limit in \eqref{eq:eps-poincare}.
\end{proof}

\begin{corollary}
\label{cor:w-even}
Assume that $K$ is a $C^2_+$ origin-symmetric convex body and $F\in C^1(\Sn)$ is even. Then
\eq{
(n+\alpha-1)\int_{\Sn}(F-\bar{F}_w)^2|\langle X,w\rangle|^\alpha\, dV_K \le \int_{\Sn} |\nabla F|_g^2|\langle X,w\rangle|^\alpha\, dV_K,
}
where $\bar{F}_w=\frac{\int_{\Sn}F|\langle X,w\rangle|^\alpha\, dV_K}{\int_{\Sn}|\langle X,w\rangle|^\alpha\, dV_K}$.
\end{corollary}

\begin{proof}
In view of \autoref{thm:cap-poincare-from-epsilon},
\eq{
(n+\alpha-1)\left[\int_{\Omega_w}(F-\bar{F}_w)^2\ell_w^\alpha\, dV_K+ \int_{\Omega_{-w}}(F-\bar{F}_{-w})^2\ell_{-w}^\alpha\, dV_K \right] \le \int_{\Sn} |\nabla F|_g^2|\ell_w|^\alpha\, dV_K.
}
Since $K$ is origin-symmetric and $F$ is even, $\bar{F}_w=\bar{F}_{-w}$
and
\eq{
\int_{\Omega_w}(F-\bar{F}_w)^2\ell_w^\alpha\, dV_K+\int_{\Omega_{-w}}(F-\bar{F}_{-w})^2\ell_{-w}^\alpha\, dV_K=\int_{\Sn}(F-\bar{F}_w)^2|\ell_w|^\alpha\, dV_K.
}
\end{proof}

\begin{corollary}\label{cor:main-unconditional-poincare}
Assume that $K$ is a $C^{2}_{+}$ unconditional convex body and $F\in C^1(\Sn)$ is unconditional. Then
\eq{
(n+1)\int_{\Sn}(F-\bar{F})^2|X|^2\, dV_K \le \int_{\Sn} |\nabla F|_g^2|X|^2\, dV_K,\quad \bar{F}=\frac{\int_{\Sn}F|X|^2\, dV_K}{\int_{\Sn}|X|^2\, dV_K}.
}
\end{corollary}

\begin{proof}
Let $w=(1,\dots,1)$. By \autoref{cor:w-even},
\eq{
(n+1)\int_{\Sn}(F-\bar{F}_w)^2\ell_w^2\, dV_K \le \int_{\Sn}|\nabla F|_g^2\ell_w^2\, dV_K.
}

We claim that $\bar{F}_w=\bar{F}$. Indeed,
\eq{
\ell_w^2=\left(\sum_{i=1}^n X_i\right)^2=|X|^2+2\sum_{1\le i<j\le n}X_iX_j.
}
Since both $F$ and $K$ are unconditional, for every $i\neq j$,
\eq{
\int_{\Sn}F X_iX_j\, dV_K=0,\quad
\int_{\Sn}X_iX_j\, dV_K=0.
}
Hence
\eq{
\int_{\Sn}F\ell_w^2\, dV_K=\int_{\Sn}F|X|^2\, dV_K, \quad \int_{\Sn}\ell_w^2\, dV_K=\int_{\Sn}|X|^2\, dV_K,
}
and
\eq{
\bar{F}_w=\frac{\int_{\Sn}F\ell_w^2\, dV_K}{\int_{\Sn}\ell_w^2\, dV_K}=\frac{\int_{\Sn}F|X|^2\, dV_K}{\int_{\Sn}|X|^2\, dV_K}=\bar{F}.
}

Moreover, both $(F-\bar{F})^2$ and $|\nabla F|_g^2$ are unconditional, so the same cancellation of mixed terms yields
\eq{
\int_{\Sn}(F-\bar{F})^2\ell_w^2\, dV_K&=\int_{\Sn}(F-\bar{F})^2|X|^2\, dV_K,\\ 
\int_{\Sn}|\nabla F|_g^2\ell_w^2\, dV_K&=\int_{\Sn}|\nabla F|_g^2|X|^2\, dV_K.
}
\end{proof}

\begin{remark}
\label{rem:full-bochner-weight-pinching}
Let $\Phi=-\alpha\log|X|$ with $\alpha>0$. Using the identities $\nabla^2X+gX=0$ and $dX=\tau$, we compute
\eq{
\nabla^2\Phi=\alpha g-\frac{\alpha}{|X|^2}\bar{g}(\tau(\cdot),\tau(\cdot))+2\alpha\,d\log|X|\otimes d\log|X|.
}
Hence, for any $N>n-1$,
\eq{
\nabla^2\Phi-\frac{1}{N-(n-1)}d\Phi\otimes d\Phi=~&\alpha g-\frac{\alpha}{|X|^2}\bar{g}(\tau(\cdot),\tau(\cdot))\\
&+\left(2\alpha-\frac{\alpha^2}{N-(n-1)}\right)d\log|X|\otimes d\log|X|.
}
Choosing $N=n-1+\frac{\alpha}{2}$ cancels the last term and we find
\eq{
\label{eq:BE-logX-remark}
\operatorname{Ric}_{\Phi,N}:=\operatorname{Ric}+\nabla^2\Phi-\frac{1}{N-(n-1)}d\Phi\otimes d\Phi=(n-2+\alpha)g-\frac{\alpha}{|X|^2}\bar{g}(\tau(\cdot),\tau(\cdot)).
}
The final term has the wrong sign and, in general, cannot be controlled from below by a positive multiple of $g$ without an additional pinching assumption. Thus the global weighted Bochner formula does not seem well suited for proving \autoref{thm:main-dualquermass-q}. This motivates the cap-based argument developed in this section, which yields \autoref{cor:main-unconditional-poincare}.
\end{remark}

\subsection{Poincar\'e inequalities on intersections of caps}

We extend the cap-based Bochner method used above to prove \autoref{thm:cap-poincare-from-epsilon} to intersections of caps and power weights, and then prove \autoref{thm:intersection-caps-general-weight}. As we will see, this extension plays a crucial role in deriving Poincar\'e-type inequalities for unconditional convex bodies here and in \autoref{sec:log-centro-affine-geometry}.

Let
\eq{
\ell_{w_i}=\ip{X}{w_i}, \quad \Omega=\bigcap_{i=1}^m\{\ell_{w_i}>0\}, \quad \Phi=-\sum_{i=1}^m \alpha_i\log \ell_{w_i},\quad a=\sum_{i=1}^m \alpha_i,
}
where $w_1,\dots,w_m\in\R^n\setminus\{0\}$ and $\alpha_1,\dots,\alpha_m>0$.
Since $\nabla^2\ell_{w_i}=-\ell_{w_i} g$,
\eq{
\nabla^2\Phi-\frac{1}{a}\,d\Phi\otimes d\Phi &=ag+\sum_{i=1}^m \alpha_i\, d\log \ell_{w_i}\otimes d\log \ell_{w_i} \\ 
&\quad -\frac{1}{a}\left(\sum_{i=1}^m \alpha_i d\log \ell_{w_i}\right)\otimes\left(\sum_{i=1}^m \alpha_i d\log \ell_{w_i}\right).
}
By the Cauchy--Schwarz inequality, $\nabla^2\Phi-\frac{1}{a}\,d\Phi\otimes d\Phi\ge ag$.
Therefore
\eq{
\Ric+\nabla^2\Phi-\frac{1}{a}\,d\Phi\otimes d\Phi\ge (n-2+a)g.
}
This suggests a weighted Poincar\'e inequality on $\Omega$ with spectral constant $n+a-1$.

The main additional difficulty, compared to the single-cap case, is that $\Omega$ has only piecewise smooth boundary. We therefore approximate $\Omega$ from inside by a family of smooth domains $\Omega_{\varepsilon}$. The key point is that this approximation must be chosen so that, in the weighted Bochner argument on $\Omega_{\varepsilon}$, the $\nabla^*$-second fundamental form of $\partial\Omega_{\varepsilon}$ is non-negative definite. In the single-cap case, this sign ultimately comes from the concavity of the defining function $\ell_w$, namely $\nabla^2\ell_w=-\ell_w g\le 0$.

For intersections of several caps, one therefore needs a smooth defining function for the approximation that enjoys the same concavity property. 

\begin{theorem}\label{thm:intersection-caps-poincare-ineq}
Suppose $K$ is $C^2_+$. With the notation above, assume that $\Omega\neq\emptyset$, and set
\eq{
d\mu_{\Phi}=e^{-\Phi}\, dV_K=\prod_{i=1}^m \ell_{w_i}^{\alpha_i}\, dV_K.
}
Then for every $F\in C^1(\Sn)$,
\eq{
(n+a-1)\int_\Omega (F-\bar{F}_{\Phi})^2\,d\mu_{\Phi} \le \int_\Omega |\nabla F|_g^2\,d\mu_{\Phi},\quad \bar{F}_{\Phi}=\frac{\int_\Omega F\,d\mu_{\Phi}}{\int_\Omega d\mu_{\Phi}}.
}
\end{theorem}

\begin{proof} We may assume $F\in C^{\infty}(\Sn)$ and $K$ is $C^{\infty}_+$. Let $q(x)=\min_{1\le i\le m}\ell_{w_i}(x)$. For $\varepsilon>0$ sufficiently small, on $S^{n-1}$ we define
\eq{
q_{\varepsilon}=-\varepsilon\log\left(\sum_{i=1}^m e^{-\ell_{w_i}/\varepsilon}\right).
}
Note that we have
\eq{
\label{eq:lower-upper-bounds-Q} 
q_{\varepsilon}(x)\le q(x)\le q_{\varepsilon}(x)+\varepsilon\log m,\quad \text{for all }x\in \Sn.
}

Due to Sard's theorem, for each $\varepsilon>0$ there exists a regular value $c(\varepsilon)\in(0,\varepsilon)$. We now define smooth open sets $\Omega_{\varepsilon}=\{q_{\varepsilon}>c(\varepsilon)\}$. We next show that they are connected.

\begin{claim}\label{claim:Omega-eps-connected}
$\Omega_{\varepsilon}$ is connected.
\end{claim}

\begin{proof}
Let $x_0,x_1\in \Omega_{\varepsilon}$ be arbitrary. We will construct a continuous curve in $\Omega_{\varepsilon}$ connecting $x_0$ to $x_1$. Set $p_0=X(x_0)$ and $p_1=X(x_1)$. For $t\in[0,1]$, define $p_t=(1-t)p_0+tp_1$. The points $p_0$ and $p_1$ belong to $\partial K$, and therefore the convexity of $K$ implies that $p_t\in K$ for every $t\in[0,1]$. 

Since $x_0,x_1\in\Omega_{\varepsilon}$, for $j=0,1$ we have
\eq{
\sum_{i=1}^{m}e^{-\ell_{w_i}(x_j)/\varepsilon}<e^{-c(\varepsilon)/\varepsilon}.
}
Hence $\ell_{w_i}(x_j)>c(\varepsilon)>0$ for every $i=1,\ldots,m$ and $j=0,1$. It follows that
\eq{
\langle p_t,w_i\rangle=(1-t)\langle p_0,w_i\rangle+t\langle p_1,w_i\rangle=(1-t)\ell_{w_i}(x_0)+t\ell_{w_i}(x_1)>c(\varepsilon)>0\quad \forall t\in[0,1].
}
In particular, $p_t\neq 0$ for every $t\in[0,1]$. We write 
\eq{
p_t=r_t\theta_t, \text{ where } r_t=|p_t|>0 \text{ and } \theta_t=p_t/|p_t|\in\Sn. 
}
Let $\rho_K$ denote the radial function of $K$. Then $r_t\le \rho_K(\theta_t)$. Hence, if we set
\eq{
\lambda_t=\frac{\rho_K(\theta_t)}{r_t}\ge 1,
}
then $\widehat{p}_t=\lambda_t p_t\in \partial K$. Now we define $x_t=X^{-1}(\widehat{p}_t)$. Since $X$ is a diffeomorphism, the curve $t\mapsto x_t$ is continuous. We prove that this curve stays in $\Omega_{\varepsilon}$. 

For each $i$,
\eq{
e^{-\langle p_t,w_i\rangle/\varepsilon}=e^{-((1-t)\ell_{w_i}(x_0)+t\ell_{w_i}(x_1))/\varepsilon}\le (1-t)e^{-\ell_{w_i}(x_0)/\varepsilon}+te^{-\ell_{w_i}(x_1)/\varepsilon}.
}
Summing this inequality over $i=1,\ldots,m$, we obtain
\eq{
\sum_{i=1}^{m}e^{-\langle p_t,w_i\rangle/\varepsilon}\le (1-t)\sum_{i=1}^{m}e^{-\ell_{w_i}(x_0)/\varepsilon}+t\sum_{i=1}^{m}e^{-\ell_{w_i}(x_1)/\varepsilon}<e^{-c(\varepsilon)/\varepsilon}.
}
Next, since $\langle p_t,w_i\rangle>c(\varepsilon)>0$, and $\lambda_t\ge 1$, we have
\eq{
e^{-\langle \widehat{p}_t,w_i\rangle/\varepsilon}=e^{-\lambda_t\langle p_t,w_i\rangle/\varepsilon}\le e^{-\langle p_t,w_i\rangle/\varepsilon}.
}
Consequently,
\eq{
\sum_{i=1}^{m}e^{-\ell_{w_i}(x_t)/\varepsilon}=\sum_{i=1}^{m}e^{-\langle X(x_t),w_i\rangle/\varepsilon}=\sum_{i=1}^{m}e^{-\langle \widehat{p}_t,w_i\rangle/\varepsilon}\le \sum_{i=1}^{m}e^{-\langle p_t,w_i\rangle/\varepsilon}<e^{-c(\varepsilon)/\varepsilon}.
}
That is, $x_t\in\Omega_{\varepsilon}$ for every $t\in[0,1]$.
\end{proof}

By \eqref{eq:lower-upper-bounds-Q}, we have $\Omega_{\varepsilon}\subset\Omega$. Conversely, let $x\in\Omega$. Then $q(x)>0$. Moreover, for all sufficiently small $\varepsilon$, we have $q_{\varepsilon}(x)\ge q(x)-\varepsilon\log m>\varepsilon>c(\varepsilon)$, and hence $x\in\Omega_{\varepsilon}$. We have shown
\eq{
\label{eq:set-convergence} 
\mathbf{1}_{\Omega_{\varepsilon}}(x)\to \mathbf{1}_\Omega(x)\quad \text{as }\varepsilon\to 0.
}

We next verify that $\nabla^2q_{\varepsilon}\le 0$. For convenience, we set
\eq{
\theta_i=\frac{e^{-\ell_{w_i}/\varepsilon}}{\sum_{j=1}^m e^{-\ell_{w_j}/\varepsilon}},\quad i=1,\ldots,m.
}
Using $\nabla^2 \ell_{w_i}=-\ell_{w_i} g$, we calculate
\eq{
\nabla q_{\varepsilon}&=\sum_{i=1}^m \theta_i \nabla \ell_{w_i},\\ 
\nabla^2 q_{\varepsilon}&=-\left(\sum_{i=1}^m \theta_i \ell_{w_i}\right)g -\frac{1}{\varepsilon}\left( \sum_{i=1}^m \theta_i\,d\ell_{w_i}\otimes d\ell_{w_i} -\sum_{i,j=1}^m \theta_i\theta_j\,d\ell_{w_i}\otimes d\ell_{w_j} \right).
}
Due to $\sum_{i=1}^m \theta_i=1$ and the Cauchy--Schwarz inequality, the second term is non-positive definite. Therefore
\eq{
\label{eq:hess-Q-epsilon} 
\nabla^2 q_{\varepsilon}\le 0 \quad \text{on } \Omega.
}

Let $\mathbf{N}_{\varepsilon}=-\frac{\nabla q_{\varepsilon}}{|\nabla q_{\varepsilon}|_g}$ be the outward $g$-unit normal to $\partial\Omega_{\varepsilon}$. For  vector fields $Z$ that are tangential along $\partial\Omega_{\varepsilon}$, due to \eqref{eq:hess-Q-epsilon}, we have
\eq{
\mathrm{II}^*_{\varepsilon}(Z,Z) =g(Z,\nabla_Z^*\mathbf{N}_{\varepsilon}) =-\frac{1}{|\nabla q_{\varepsilon}|_g}\nabla^2 q_{\varepsilon}(Z,Z)\ge 0.
}

Now applying the weighted Bochner formula on the smooth domain $\Omega_{\varepsilon}$, exactly as in the proof of \autoref{thm:R-type-Neumann-eps}, we obtain
\eq{
\label{eq:spectral-epsilon} 
\int_{\Omega_{\varepsilon}}(\LPhi u)^2\,d\mu_{\Phi}\ge (n+a-1)\int_{\Omega_{\varepsilon}} |\nabla u|_g^2\,d\mu_{\Phi}
}
for every $u\in C^3(\overline{\Omega_{\varepsilon}})$ with $u_{\mathbf{N}_{\varepsilon}}=0 $ on $ \partial\Omega_{\varepsilon}$.

Define $\bar{F}_{\Phi,\varepsilon}=
\frac{\int_{\Omega_{\varepsilon}} F\,d\mu_{\Phi}}{\int_{\Omega_{\varepsilon}} d\mu_{\Phi}}$ and $f_{\varepsilon}=F-\bar{F}_{\Phi,\varepsilon}$. Then there exists a solution $u_{\varepsilon}\in C^\infty(\overline{\Omega_{\varepsilon}})$ to
\eq{
\LPhi u_{\varepsilon}=-f_{\varepsilon} \quad \text{in } \Omega_{\varepsilon}, \quad (u_{\varepsilon})_{\mathbf{N}_{\varepsilon}}=0 \quad \text{on } \partial\Omega_{\varepsilon}.
}
Applying \eqref{eq:spectral-epsilon} to $u_{\varepsilon}$, integrating by parts, and using the Cauchy--Schwarz inequality:
\eq{
(n+a-1)\int_{\Omega_{\varepsilon}} (F-\bar{F}_{\Phi,\varepsilon})^2\,d\mu_{\Phi} \le \int_{\Omega_{\varepsilon}}|\nabla F|_g^2\,d\mu_{\Phi}.
}
Due to \eqref{eq:set-convergence} and the dominated convergence theorem, we may let $\varepsilon\to 0$ to obtain:
\eq{
(n+a-1)\int_{\Omega} (F-\bar{F}_{\Phi})^2\,d\mu_{\Phi} \le \int_{\Omega} |\nabla F|_g^2\,d\mu_{\Phi}.
}
\end{proof}

In \autoref{thm:intersection-caps-poincare-ineq}, consider the special case that $w_i=E_i$, $i=1,\ldots,n$, where $\{E_i\}$ is the standard coordinate basis of $\R^n$. Then the Poincar\'e constant can be improved to $n+a$. Indeed, for an unconditional convex body $Q$, define $|Q|_{\bar{\alpha},+}=\int_{Q\cap (0,\infty)^n} x_1^{\alpha_1}\cdots x_n^{\alpha_n}\, dx$ and
\eq{
f_{Q,\bar{\alpha}}(z)=\mathbf{1}_{Q\cap (0,\infty)^n}(e^{z_1},\ldots,e^{z_n})e^{(\alpha_1+1)z_1+\cdots+(\alpha_n+1)z_n}, \quad z\in\R^n,
}
where $\bar{\alpha}:=(\alpha_1,\ldots,\alpha_n)$ and we additionally allow $\alpha_i>-1$. Then, by the change of variables $x_i=e^{z_i}$, we have $|Q|_{\bar{\alpha},+}=\int_{\R^n} f_{Q,\bar{\alpha}}(z)\,dz$.

Since the density in $f_{Q,\bar{\alpha}}$ is log-affine, the argument of \cite{Sar15} based on the Pr\'ekopa--Leindler inequality gives
\eq{
\label{eq:Sar15-improved-constant} 
\big|(1-\lambda)\cdot K+_0 \lambda\cdot L\big|_{\bar{\alpha},+} \ge |K|_{\bar{\alpha},+}^{1-\lambda}|L|_{\bar{\alpha},+}^{\lambda}.
}

For $|s|$ sufficiently small, $h_s=he^{sF}$ is the support function of a $C^2_+$ unconditional convex body $K_s$, and by the definition of the $L_0$-sum,
\eq{
(1-\lambda)\cdot K_s+_0 \lambda\cdot K_t=K_{(1-\lambda)s+\lambda t}.
}
Hence \eqref{eq:Sar15-improved-constant} implies that
\eq{
\big|K_{(1-\lambda)s+\lambda t}\big|_{\bar{\alpha},+} \ge |K_s|_{\bar{\alpha},+}^{1-\lambda}|K_t|_{\bar{\alpha},+}^{\lambda}.
}
That is, $s\mapsto \log |K_s|_{\bar{\alpha},+}$ is concave for $|s|$ sufficiently small. We may now differentiate at $s=0$ to obtain the corresponding weighted Poincar\'e inequality with the improved constant $n+a$. See also \autoref{thm:power-product-unconditional-improved} for an alternative proof that does not rely on the Pr\'ekopa--Leindler inequality.

\begin{remark}
The inequality of \autoref{thm:intersection-caps-poincare-ineq} may be interpreted as an infinitesimal version of a Brunn--Minkowski-type inequality. For $a=\sum_{i=1}^m\alpha_i$, consider the functional
\eq{
\mathcal{F}(L)=\int_L\prod_{i=1}^m \ip{x}{w_i}_{+}^{\alpha_i}\, dx.
}
The Borell--Brascamp--Lieb inequality then gives the concavity of $\mathcal{F}^{1/(n+a)}$. Therefore, formally differentiating this statement along a linear variation $h_s=h+s\psi$, with $F=\psi/h$, leads to the same weighted Poincar\'e constant $n+a-1$ as in \autoref{thm:intersection-caps-poincare-ineq}.

There is, however, a technical nuisance in this variational interpretation. The Gauss image is the moving intersection
\eq{
\Omega_s=\bigcap_{i=1}^m\{\ell_{i,s}>0\}, \quad \ell_{i,s}=\ip{X_s}{w_i}.
}
Its boundary is generally only piecewise smooth, and differentiating the weight $\prod_{i=1}^m(\ell_{i,s})_{+}^{\alpha_i}$ requires keeping track of weak derivatives near the faces and corners. Although the singular factors $(\ell_{i,s})^{\alpha_i-1}$ are locally integrable for $\alpha_i>0$, and one expects this approach to yield the same result after a suitable approximation or weak differentiation argument, the full details are not clear to us.

The above cap-based proof overcomes these issues conveniently and provides an intrinsic derivation of the preceding Poincar\'e inequality.

In the capillary setting, this variational viewpoint becomes manageable, because the Gauss image is fixed and the admissible variations are well posed via a Robin boundary condition. We will not pursue this direction here.
\end{remark}

\begin{proof}[Proof of \autoref{thm:intersection-caps-general-weight}]
Set $\ell=(\ell_{w_1},\dots,\ell_{w_m})$. We calculate
\eq{
\nabla \Phi=-\sum_{i=1}^{m}\partial_i(\log\omega)(\ell)\nabla \ell_{w_i}
}
and
\eq{
\begin{multlined} 
\nabla^{2}\Phi-\frac{1}{\alpha}d\Phi\otimes d\Phi =\left(\sum_{i=1}^{m}\ell_{w_i}\partial_i(\log\omega)(\ell)\right)g \\ 
-\sum_{i,j=1}^{m}\left(\partial_{ij}^{2}(\log\omega)+\frac{1}{\alpha}\partial_i(\log\omega)\partial_j(\log\omega)\right)(\ell)\,d\ell_{w_i}\otimes d\ell_{w_j}. 
\end{multlined}
}
Note that \eqref{eq:matrix-condition-general-weight} says
\eq{
D^{2}\log\omega+\frac{1}{\alpha}D\log\omega\otimes D\log\omega =\alpha\omega^{-\frac{1}{\alpha}}D^{2}\omega^{\frac{1}{\alpha}}\le 0.
}
Therefore, the second term is non-negative definite, and by \eqref{eq:radial-condition-general-weight},
\eq{
\nabla^{2}\Phi-\frac{1}{\alpha}d\Phi\otimes d\Phi \ge \left(\sum_{i=1}^{m}\ell_{w_i}\partial_i(\log\omega)(\ell)\right)g\ge \beta g.
}
Since $\Ric=(n-2)g$, it follows that $\Ric+\nabla^{2}\Phi-\frac{1}{\alpha}d\Phi\otimes d\Phi\ge (n-2+\beta)g$. Using this and the finite-$N$ estimate (see the proof of \autoref{lem:finite-N-BE}), we obtain
\eq{
|(\nabla^{*})^{2}u|_{g}^{2}+(\Ric+\nabla^{2}\Phi)(\nabla u,\nabla u) \ge \frac{1}{n-1+\alpha}(\mathcal{L}_{\Phi}u)^{2}+(n-2+\beta)|\nabla u|_{g}^{2}.
}
The rest of the proof is identical to that of \autoref{thm:intersection-caps-poincare-ineq}.
\end{proof}

\begin{corollary}\label{cor:unconditional-sum-weight}
Let $0<\alpha\le 1$. Assume that $K$ is a $C^{2}_{+}$ unconditional convex body, and that $F\in C^1(\Sn)$ is unconditional. Then
\eq{
(n+\alpha-1)\int_{\Sn}(F-\bar{F})^2\sum_{i=1}^{n}|X_i|^\alpha\, dV_K \le \int_{\Sn} |\nabla F|_g^2\sum_{i=1}^{n}|X_i|^\alpha\, dV_K,
}
where
\eq{
\bar{F}=\frac{\int_{\Sn}F\sum_{i=1}^{n}|X_i|^\alpha\, dV_K}{\int_{\Sn}\sum_{i=1}^{n}|X_i|^\alpha\, dV_K}.
}
\end{corollary}

\begin{proof}
We apply \autoref{thm:intersection-caps-general-weight} on $\Omega_+=\bigcap_{i=1}^{n}\{X_i>0\}$, with $w_i=E_i$ and the weight $\omega(t_1,\dots,t_n)=\sum_{i=1}^{n} t_i^\alpha$. Since $\omega^{1/\alpha}$ is concave on $(0,\infty)^n$ for $0<\alpha\le 1$, the matrix condition holds. Moreover, since $\omega$ is homogeneous of degree $\alpha$, $\sum_{i=1}^{n} t_i\partial_i\log\omega(t)=\alpha$. Hence
\eq{
( n+\alpha-1 )\int_{\Omega_+}(F-\bar{F}_{\Omega_+})^2\sum_{i=1}^{n}X_i^\alpha\, dV_K \le \int_{\Omega_+} |\nabla F|_g^2\sum_{i=1}^{n}X_i^\alpha\, dV_K,
}
where $\bar{F}_{\Omega_+}=\frac{\int_{\Omega_+}F\sum_{i=1}^{n}X_i^\alpha\, dV_K}
{\int_{\Omega_+}\sum_{i=1}^{n}X_i^\alpha\, dV_K}$. Since $K$ is unconditional and $F$ is unconditional, all orthants contribute equally.
\end{proof}

\begin{remark}\label{rem:unconditional-B1-ball}
Assume that $\tilde{f}\in C^2\left((0,\infty)^n\right)$ is positive, concave, and $k$-homogeneous on  $(0,\infty)^n$. Define $f(x)=\tilde{f}\left(X(x)\right),\,
\Phi(x)=-\alpha \log f(x)
$ on $ \Omega_+$, where $\alpha>0$. By \autoref{lem:homogeneous-extension},
\eq{
\label{eq:simple-poincare-derivation} 
\nabla^2\Phi-\frac{1}{\alpha}d\Phi\otimes d\Phi =-\frac{\alpha}{f}\nabla^2 f \ge \alpha kg \quad \text{on } \Omega_+.
}

For example, we may take $\tilde{f}(x)=x_1+\cdots+x_n$ on $(0,\infty)^n$. Then $\tilde{f}$ is concave and one-homogeneous. Thus, by \eqref{eq:simple-poincare-derivation},
\eq{
\label{eq:unconditional-l1-type-weight} 
(n+\alpha-1) \int_{\Sn}(F-\bar{F})^2\left(\sum_{i=1}^n |X_i|\right)^\alpha\, dV_K \le \int_{\Sn}|\nabla F|_g^2\left(\sum_{i=1}^n |X_i|\right)^\alpha\, dV_K,
}
where $\bar{F}=\frac{\int_{\Sn}F\left(\sum_{i=1}^n |X_i|\right)^\alpha\, dV_K}
{\int_{\Sn}\left(\sum_{i=1}^n |X_i|\right)^\alpha\, dV_K}$.
\end{remark}

\begin{remark}
There is a complementary formulation of the cap argument obtained by duality, in which the roles of $\nabla$ and $\nabla^*$ are interchanged and the measure $dV_K$ is replaced by $h^{-n}\, dx$. In this setting the exact analogue of the function $\ell_w(x)=\ip{X(x)}{w}$ is $l_w(x)=\ip{x}{w}/h(x)$. Using the identity $(\nabla^*)^2 l_w+l_wg=0$, which is dual to $\nabla^2\ell_w+\ell_wg=0$, one can obtain analogous cap inequalities in the dual setting.
\end{remark}


\section{Second variation in centro-affine geometry}\label{sec:boundary-variation-to-centro-affine-formulation}

Let $W:\R^n\to \R$ be a smooth function, and consider the associated weighted measure and its total mass on a smooth strictly convex body $K$ by
\eq{
d\mu(x)=e^{-W(x)}\, dx,\quad \mu(K)=\int_K e^{-W(x)}\, dx,\quad d\mu_{\partial K}=e^{-W}\,d\mathcal{H}^{n-1}\big|_{\partial K}.
}
Let $\psi\in C^2(\Sn)$, and for $|s|$ sufficiently small let $K_s$ be the convex body with support function $h_{K_s}=h_K+s\psi$. We define
\eq{
\phi=\psi\circ\nu\quad \text{i.e.,}\quad \psi=\phi\circ X.
}
Here $\nu=X^{-1}$ denotes the outward unit normal vector field on $\partial K$.

We recall the following second-variation formula (cf. \cite[Thm. 6.6]{KM18}, \cite[Eq. (10)]{SZ25}):
\eq{
\label{eq:standard-second-var} 
\left.\frac{d^2}{ds^2} \mu(K_s)\right|_{s=0} =\int_{\partial K}\left(H_W\phi^2-\ip{\mathrm{II}^{-1}\nabla_{\partial K}\phi}{\nabla_{\partial K}\phi}\right)\,d\mu_{\partial K},
}
where $H_W=H-\ip{DW}{\nu}$, $\mathrm{II}$ denotes the second fundamental form of the boundary $\partial K$, $\nabla_{\partial K}$ is the tangential gradient on $\partial K$, and $H$ denotes the mean curvature of $\partial K$.

In this section we rewrite the second variation formula \eqref{eq:standard-second-var} on the sphere $\Sn$ and express it in the language of centro-affine geometry. Recall first that the pullback of the Euclidean surface area measure under the inverse Gauss map $X$ is $X^*\left(d\mathcal{H}^{n-1}\big|_{\partial K}\right)=dS_K$. Consequently
\eq{
X^*(d\mu_{\partial K})=X^*\left(e^{-W}\,d\mathcal{H}^{n-1}\big|_{\partial K}\right) =e^{-W(X)}\,dS_K=e^{-w}\,dS_K,
}
where $w$ on the sphere is defined by $w=W\circ X$.

\begin{lemma}\label{lem:pullback-business} We have
\eq{
\int_{\partial K}\phi\,d\mu_{\partial K}=\int_{\Sn}\psi\frac{e^{-w}}{h}\, dV_K,
}
\eq{
\label{eq:pullback-HW} 
(H_W\circ X)(x)=\tr(\tau^{-1})(x)-\ip{DW(X(x))}{x},
}
\eq{
\label{eq:pullback-gradient} 
\left(\ip{\mathrm{II}^{-1}\nabla_{\partial K}\phi}{\nabla_{\partial K}\phi}\right)\circ X=\frac{1}{h}\abs{\nabla\psi}_g^2.
}
\end{lemma}

\begin{proof}
The first identity follows from $X^*(d\mu_{\partial K})=\frac{e^{-w}}{h}\, dV_K$, together with $\psi=\phi\circ X$.

Next we prove \eqref{eq:pullback-HW}. At the point $X(x)\in \partial K$, the outer unit normal is $x$. Let $\{e_i\}_{i=1}^{n-1}$ be a local $\bar{g}$-orthonormal frame diagonalizing $\tau$, say $\tau e_i=\lambda_i e_i$. Then the principal curvatures are $\lambda_i^{-1}$, and $H\circ X=\tr(\tau^{-1})$. Since $\nu\circ X(x)=x$, it follows that
\eq{
H_W\circ X=H\circ X-\ip{DW(X)}{x}=\tr(\tau^{-1})-\ip{DW(X)}{x}.
}

We prove \eqref{eq:pullback-gradient}. Let $\hat{e}_i=\frac{1}{\lambda_i}dX(e_i)$. Then $\{\hat{e}_i\}$ is an orthonormal basis for $T_{X(x)}\partial K$ and
\eq{
\hat{e}_i\phi=d\phi(\hat{e}_i)=\frac{1}{\lambda_i}d(\phi\circ X)(e_i)=\frac{e_i\psi}{\lambda_i}.
}
Since the eigenvalues of $\mathrm{II}^{-1}$ are $\lambda_i$, from \eqref{eq:g-frame} we obtain
\eq{
\left(\ip{\mathrm{II}^{-1}\nabla_{\partial K}\phi}{\nabla_{\partial K}\phi}\right)\circ X =\sum_{i=1}^{n-1}\lambda_i(\hat{e}_i\phi)^2=\sum_{i=1}^{n-1}\frac{(e_i\psi)^2}{\lambda_i}=\frac{1}{h}\abs{\nabla\psi}_g^2.
}
\end{proof}

\begin{proposition}\label{prop:standard-second-var}
The formula \eqref{eq:standard-second-var} is equivalent to
\eq{
\label{eq:sphere-2nd-var} 
\left.\frac{d^2}{ds^2}\mu(K_s)\right|_{s=0} &=\int_{\Sn}\left(\left(\tr(\tau^{-1})-\ip{DW(X)}{x}\right)\psi^2\right)\frac{e^{-w}}{h}\, dV_K \\ 
&\quad -\int_{\Sn}\frac{e^{-w}}{h^2}\abs{\nabla\psi}_g^2\, dV_K.
}
\end{proposition}

\begin{proof}
Using \autoref{lem:pullback-business}, we compute
\eq{
\int_{\partial K}H_W\phi^2\,d\mu_{\partial K} &=\int_{\Sn}\left(\tr(\tau^{-1})-\ip{DW(X)}{x}\right)\psi^2\frac{e^{-w}}{h}\, dV_K,\\ 
\int_{\partial K}\ip{\mathrm{II}^{-1}\nabla_{\partial K}\phi}{\nabla_{\partial K}\phi}\,d\mu_{\partial K} &=\int_{\Sn}\frac{1}{h}\abs{\nabla\psi}_g^2\frac{e^{-w}}{h}\, dV_K=\int_{\Sn}\frac{e^{-w}}{h^2}\abs{\nabla\psi}_g^2\, dV_K.
}
Substituting these identities into \eqref{eq:standard-second-var} yields \eqref{eq:sphere-2nd-var}.
\end{proof}

We now specialize to the context of the dual quermassintegrals where
\eq{
\label{eq:W-special} 
W(x):=(n-q)\log\abs{x},\quad w(x)=W\circ X(x)=(n-q)\log\abs{X(x)}.
}
Although $W$ is singular at the origin, the singularity is irrelevant for the first and second variation calculations. Indeed, for $|s|$ sufficiently small, the support functions $h_s=h+s\psi$ remain uniformly positive. Hence there exists $r>0$ such that
\eq{
B_r\subset K_s
\quad
\text{for all sufficiently small } |s|.
}
We may now choose $\widetilde{W}\in C^{\infty}(\R^n)$ such that
\eq{
\widetilde{W}=W
\quad
\text{on } \R^n\setminus B_r.
}
Then
\eq{
\int_{K_s}e^{-W}\, dx-\int_{K_s}e^{-\widetilde{W}}\, dx
=\int_{B_r}\left(e^{-W}-e^{-\widetilde{W}}\right)\, dx.
}
Since the right-hand side is independent of $s$, \autoref{prop:standard-second-var} applies to $W$.

Define
\eq{
f=\frac{\psi}{h} e^{-w},\quad F=\frac{\psi}{h}=e^w f=\abs{X}^{n-q}f.
}
The $L^q$ cone-volume measure is defined by
\eq{
dV_{K,q}=\abs{X}^{q-n}\, dV_K=e^{-w}\, dV_K.
}

With the notation above,
\eq{
\int_{\Sn}\frac{\psi}{h}e^{-w}\, dV_K=\int_{\Sn}f\, dV_K=\int_{\Sn}F\,dV_{K,q}.
}
Moreover,
\eq{
\label{eq:gradF-identity} 
\nabla F=e^w\left(\nabla f+f\nabla w\right).
}

\begin{lemma}\label{lem:simple-cross}
We have $g\left(\nabla\log h,\nabla \log\abs{X}\right)=1-\frac{h^2}{\abs{X}^2}$.
\end{lemma}

\begin{proof}
Choose a local $\bar{g}$-orthonormal frame $\{e_i\}$ diagonalizing $\tau$, with $\tau e_i=\lambda_i e_i$. By \eqref{eq:g-frame},
\eq{
g(\nabla a,\nabla b)=\sum_{i=1}^{n-1}\frac{h}{\lambda_i}(e_i a)(e_i b).
}
Due to \eqref{eq:dX-tau-relation}, $e_i(\abs{X}^2)=2\ip{X}{dX(e_i)}=2\lambda_i\ip{X}{e_i}$. Therefore
\eq{
e_i(\log\abs{X})=\frac{1}{2\abs{X}^2}e_i(\abs{X}^2)=\frac{\lambda_i}{\abs{X}^2}e_i h.
}
Combining the last two identities, we obtain
\eq{
g\left(\nabla\log h,\nabla\log\abs{X}\right) =\sum_{i=1}^{n-1}\frac{(e_i h)^2}{\abs{X}^2} =\frac{\abs{\bar{\nabla} h}^2}{\abs{X}^2}=1-\frac{h^2}{\abs{X}^2}.
}
\end{proof}

\begin{proposition} \label{prop:q-identity}
The formula \eqref{eq:sphere-2nd-var} can be rewritten as
\eq{
\label{eq:final-q-1} 
\left.\frac{d^2}{ds^2}\mu(K_s)\right|_{s=0}&=\int_{\Sn}(q-1)F^2-\abs{\nabla F}_{g}^{2}\,dV_{K,q}.
}
\end{proposition}

\begin{proof}
First we substitute $\psi=he^w f$ into \eqref{eq:sphere-2nd-var}:
\eq{
\left(\tr(\tau^{-1})-\ip{DW(X)}{x}\right)\psi^2\frac{e^{-w}}{h}\, dV_K =\left(\tr(\tau^{-1})-\ip{DW(X)}{x}\right)he^w f^2\, dV_K.
}
Next we rewrite the gradient term:
\eq{
\nabla\psi=\nabla(he^w f)=he^w\left(\nabla f+f\nabla\log h+f\nabla w\right),
}
so
\eq{
\label{eq:grad-step-1} 
\frac{e^{-w}}{h^2}\abs{\nabla\psi}_g^2=e^w\abs{\nabla f+f\nabla\log h+f\nabla w}_g^2.
}

Note that (see \eqref{eq:centro-affine-laplacian})
\eq{
\label{eq:trace-identity} 
\tr(\tau^{-1})=\Delta\frac{1}{h}+(n-1)\frac{1}{h}.
}
Using this and integrating by parts, we obtain
\eq{
\label{eq:ibp-expanded} 
&\int_{\Sn}he^w f^2\Delta\frac{1}{h}\, dV_K\\ 
&=-\int_{\Sn}g\left(\nabla\frac{1}{h},\nabla(he^w f^2)\right)\, dV_K \\ 
&=\int_{\Sn}e^w\left(f^2 g(\nabla\log h,\nabla w)+2f g(\nabla\log h,\nabla f)+f^2\abs{\nabla\log h}_g^2\right)\, dV_K.
}
We also rewrite \eqref{eq:grad-step-1} as follows:
\eq{
\label{eq:square-expanded} 
\abs{\nabla f+f\nabla\log h+f\nabla w}_g^2 
&=\abs{\nabla f+f\nabla w}_g^2+f^2\abs{\nabla\log h}_g^2 \\ 
&\quad+2f g(\nabla f,\nabla\log h)+2f^2 g(\nabla w,\nabla\log h).
}
Now substituting \eqref{eq:trace-identity}, \eqref{eq:ibp-expanded}, and \eqref{eq:square-expanded} into \eqref{eq:sphere-2nd-var}, we obtain
\eq{
\label{eq:after-cancel} 
\left.\frac{d^2}{ds^2}\mu(K_s)\right|_{s=0} =\int_{\Sn}e^w\left(-\abs{\nabla f+f\nabla w}_g^2+\mathcal{R}_q f^2\right)\, dV_K,
}
where $\mathcal{R}_q=(n-1)-h\ip{DW(X)}{x}-g(\nabla\log h,\nabla w)$. 

We have
\eq{
DW(x)=(n-q)\frac{x}{\abs{x}^2},\quad \nabla w=(n-q)\nabla\log\abs{X}
}
and thus
\eq{
h\ip{DW(X)}{x}=(n-q)\frac{h^2}{\abs{X}^2},\quad g(\nabla\log h,\nabla w)=(n-q)\left(1-\frac{h^2}{\abs{X}^2}\right),
}
where we used \autoref{lem:simple-cross}. Therefore $\mathcal{R}_q=q-1$ and \eqref{eq:after-cancel} becomes
\eq{
\label{eq:before-F} 
\left.\frac{d^2}{ds^2}\mu(K_s)\right|_{s=0}=\int_{\Sn}e^w\left(-\abs{\nabla f+f\nabla w}_g^2+(q-1)f^2\right)\, dV_K.
}

Finally, by \eqref{eq:gradF-identity}, $\abs{\nabla f+f\nabla w}_g^2=e^{-2w}\abs{\nabla F}_g^2$. Now recall $F=e^w f$ and $dV_{K,q}=e^{-w}dV_K$, and hence
\eq{
e^w\abs{\nabla f+f\nabla w}_g^2\, dV_K=\abs{\nabla F}_g^2\,dV_{K,q},\quad e^w f^2\, dV_K=F^2\,dV_{K,q}.
}
\end{proof}

\begin{proof}[First proof of \autoref{thm:main-dualquermass-q} (in the $C^2_+$ class and without equality characterization)]
Let
\eq{
\mu(K)=\int_K |x|^2\, dx =\frac{n}{n+2}\widetilde{V}_{n+2}(K) =\frac{1}{n+2}\int_{\Sn}|X|^2\, dV_K.
}
Due to \autoref{prop:q-identity} with $q=n+2$ and \autoref{cor:main-unconditional-poincare},
\eq{
\left.\frac{d^2}{ds^2}\mu(K_s)\right|_{s=0} \le \frac{n+1}{n+2}\frac{\left(\left.\frac{d}{ds}\mu(K_s)\right|_{s=0}\right)^2}{\mu(K)},
}
where we used the identity (cf. \cite[Thm. 6.6]{KM18})
\eq{
\left.\frac{d}{ds}\mu(K_s)\right|_{s=0}=\int_{\Sn}F|X|^{2}\, dV_K.
}
Thus $\mu^{1/(n+2)}$ is (locally) infinitesimally concave on the class of unconditional convex bodies. By the proof of \cite[Lem. 4]{SZ25} or \cite[Lem. 3.1]{CLM17}, this implies the (global) concavity of $\mu^{1/(n+2)}$ with respect to $L_1$-sum, and hence
\eq{
\mu(K+L)^{\frac{1}{n+2}} \ge \mu(K)^{\frac{1}{n+2}}+\mu(L)^{\frac{1}{n+2}}.
}
\end{proof}

\section{Applications of BBL and PL inequalities}

\begin{lemma}
\label{lem:linear-weight-alpha-vector}
Let $\alpha>0$ and $w\in\R^n\setminus\{0\}$, and define $\mathcal{H}_w=\{x\in\R^n:\ip{x}{w}\ge 0\}$. Suppose $K,L$ are convex bodies contained in $\mathcal{H}_w$. Then
\eq{
\left(\int_{K+L}\ip{x}{w}^{\alpha}\, dx\right)^{\frac1{n+\alpha}} \ge \left(\int_K\ip{x}{w}^{\alpha}\, dx\right)^{\frac1{n+\alpha}}+ \left(\int_L\ip{x}{w}^{\alpha}\, dx\right)^{\frac1{n+\alpha}}.
}
\end{lemma}

\begin{proof}
Let $\lambda\in(0,1)$ and $A,B\subset \mathcal{H}_w$ be convex bodies:
\eq{
\label{eq:lambda-concavity} 
\left(\int_{(1-\lambda)A+\lambda B}\ip{x}{w}^{\alpha}\, dx\right)^{\frac1{n+\alpha}} \ge (1-\lambda)\left(\int_A\ip{x}{w}^{\alpha}\, dx\right)^{\frac1{n+\alpha}}+\lambda\left(\int_B\ip{x}{w}^{\alpha}\, dx\right)^{\frac1{n+\alpha}}.
}

Define
\eq{
\mathrm{f}(x)=\mathbf{1}_A(x)\ip{x}{w}^{\alpha}, \quad \mathrm{g}(y)=\mathbf{1}_B(y)\ip{y}{w}^{\alpha}, \quad \mathrm{h}(z)=\mathbf{1}_{(1-\lambda)A+\lambda B}(z)\ip{z}{w}^{\alpha}.
}
Then for all $x,y\in\R^n$,
\eq{
\mathrm{h}((1-\lambda)x+\lambda y)\ge \mathrm{M}_{\frac{1}{\alpha}}^\lambda(\mathrm{f}(x),\mathrm{g}(y)),
}
where, for $s\neq 0$,
\eq{
\mathrm{M}_s^\lambda(a,b)= \begin{cases} \left((1-\lambda)a^s+\lambda b^s\right)^{\frac{1}{s}}, & ab>0,\\ 0, & ab=0. 
\end{cases}
}
Therefore, by the Borell--Brascamp--Lieb inequality with $s=\alpha^{-1}$ (see, e.g. \cite{ILS25}),
\eq{
\left(\int_{\R^n} \mathrm{h}\right)^{\frac1{n+\alpha}} \ge (1-\lambda)\left(\int_{\R^n}\mathrm{f}\right)^{\frac1{n+\alpha}}+ \lambda\left(\int_{\R^n}\mathrm{g}\right)^{\frac1{n+\alpha}}.
}
Now we apply \eqref{eq:lambda-concavity} to the dilated sets $A=\frac{1}{1-\lambda}K$ and $B=\frac{1}{\lambda}L$:
\eq{
\left(\int_{K+L}\ip{x}{w}^{\alpha}\, dx\right)^{\frac1{n+\alpha}} \ge (1-\lambda)\left(\int_A\ip{x}{w}^{\alpha}\, dx\right)^{\frac1{n+\alpha}}+ \lambda\left(\int_B\ip{x}{w}^{\alpha}\, dx\right)^{\frac1{n+\alpha}}.
}
Using the homogeneity of the weighted integral under dilations, the claim follows.
\end{proof}

\begin{proof}[Second proof of \autoref{thm:main-dualquermass-q} (with equality characterization)]
Let $w=(1,\dots,1)\in\R^n$. For every unconditional measurable set $A\subset\R^n$,
\eq{
\label{eq:second-proof-halfspace-identity} 
\int_{A\cap \mathcal{H}_w}\ip{x}{w}^2\, dx =\frac{1}{2}\int_A |x|^2\, dx.
}
Note that
\eq{
(K\cap \mathcal{H}_w)+(L\cap \mathcal{H}_w)\subset (K+L)\cap \mathcal{H}_w.
}
Applying \autoref{lem:linear-weight-alpha-vector} with $\alpha=2$ to the convex bodies $K\cap \mathcal{H}_w$ and $L\cap \mathcal{H}_w$, and then using \eqref{eq:second-proof-halfspace-identity} for $K$, $L$, and $K+L$, we immediately obtain the inequality.

Assume that equality holds. Then
\eq{
\label{eq:equality-plus} 
\left(\int_{(K\cap\mathcal{H}_w)+(L\cap \mathcal{H}_w)}\ip{x}{w}^2\, dx\right)^{\frac{1}{n+2}}
=\left(\int_{K\cap\mathcal{H}_w}\ip{x}{w}^2\, dx\right)^{\frac{1}{n+2}}
+\left(\int_{L\cap\mathcal{H}_w}\ip{x}{w}^2\, dx\right)^{\frac{1}{n+2}}.
}
Due to Dubuc's equality characterization \cite{Dub77}, there exist
\eq{
a>0,\quad b>0,\quad y\in\R^n
}
such that $\mathrm{g}(x)=a\mathrm{f}(b^{-1}(x-y))$ for almost every $x\in\R^n$. Hence
\eq{
L\cap\mathcal{H}_w=b(K\cap\mathcal{H}_w)+y.
}
Since $K$ and $L$ are origin-symmetric, we also have
\eq{
L\cap\mathcal{H}_{-w}=b(K\cap\mathcal{H}_{-w})-y.
}

We next show that $\ip{y}{w}=0$. The equality relation for the densities implies that, for almost every $z\in K\cap\mathcal{H}_w$, $\ip{bz+y}{w}^{2}=c_0\ip{z}{w}^{2}$ for some constant $c_0>0$. Since $\ip{z}{w}$ ranges over a non-trivial interval of positive values, the identity
\eq{
\left(bt+\ip{y}{w}\right)^2=c_0t^2
}
holds for all $t$ in a non-trivial interval. Hence $\ip{y}{w}=0$. Now, intersecting with $w^\perp$, we find
\eq{
L\cap w^\perp=b(K\cap w^\perp)+y.
}
Both sides are centrally symmetric convex bodies in $w^\perp$; their centers are respectively $0$ and $y$. Hence $y=0$. Therefore $L\cap\mathcal{H}_w=b(K\cap\mathcal{H}_w)$. Together with the reflected identity on $\mathcal{H}_{-w}$, this implies that $L=bK$. Thus $K$ and $L$ are dilates.
\end{proof}

\begin{proof}[Proof of \autoref{thm:unconditional-logBM-Vq}]
Set
\eq{
K_+=K\cap(0,\infty)^n,\quad L_+=L\cap(0,\infty)^n,\quad Q=(1-\lambda)\cdot K+_0\lambda\cdot L.
}
By \cite[Lem. 4.1]{Sar15}, we have
\eq{
\label{eq:Saroglou-inclusion-theorem-Vq} 
\left\{\left(x_1^{1-\lambda}y_1^\lambda,\dots,x_n^{1-\lambda}y_n^\lambda\right): x\in K_+,\ y\in L_+ \right\}\subset Q\cap(0,\infty)^n.
}

We write
\eq{
A=\{x\in\R^n:e^x\in K_+\},\quad B=\{y\in\R^n:e^y\in L_+\}.
}
Note that $A$ and $B$ are closed convex sets, and
\eq{
\{e^{(1-\lambda)x+\lambda y}: x\in A,\ y\in B\}=\left\{\left(x_1^{1-\lambda}y_1^\lambda,\dots,x_n^{1-\lambda}y_n^\lambda\right): x\in K_+,\ y\in L_+ \right\}.
}

We also define
\eq{
\phi_q(z)=\left(\sum_{i=1}^n e^{2z_i}\right)^{\frac{q-n}{2}}e^{z_1+\cdots+z_n}, \quad z\in\R^n.
}
Since $z\mapsto \log(\sum_i e^{2z_i})$ is convex, $\phi_q$ is log-concave for $0<q<n$. Hence, for all $x,y\in\R^n$:
\eq{
\phi_q((1-\lambda)x+\lambda y)\ge \phi_q(x)^{1-\lambda}\phi_q(y)^\lambda.
}

Applying the Pr\'ekopa--Leindler inequality to the functions $\mathbf{1}_A\phi_q$ and $\mathbf{1}_B\phi_q$, we obtain
\eq{
\int_{(1-\lambda)A+\lambda B}\phi_q(z)\,dz \ge \left(\int_A\phi_q(z)\,dz\right)^{1-\lambda} \left(\int_B\phi_q(z)\,dz\right)^\lambda.
}
Using the change of variables $x_i=e^{z_i}$ and $dx=e^{z_1+\cdots+z_n}\,dz$, this becomes
\eq{
\int_{(1-\lambda)A+\lambda B}\phi_q(z)\,dz \ge \left(\int_{K_+}|x|^{q-n}\, dx\right)^{1-\lambda} \left(\int_{L_+}|x|^{q-n}\, dx\right)^\lambda.
}
On the other hand, in view of \eqref{eq:Saroglou-inclusion-theorem-Vq}, we have
\eq{
\int_{(1-\lambda)A+\lambda B}\phi_q(z)\,dz =\int_{\{e^{(1-\lambda)r+\lambda s}:\ r\in A,\ s\in B\}}|x|^{q-n}\, dx \le \int_{Q\cap(0,\infty)^n}|x|^{q-n}\, dx.
}
Therefore
\eq{
\int_{Q\cap(0,\infty)^n}|x|^{q-n}\, dx \ge \left(\int_{K_+}|x|^{q-n}\, dx\right)^{1-\lambda} \left(\int_{L_+}|x|^{q-n}\, dx\right)^\lambda.
}
Since all orthants contribute equally, this is equivalent to
\eq{
\widetilde{V}_q\left((1-\lambda)\cdot K+_0\lambda\cdot L\right) \ge \widetilde{V}_q(K)^{1-\lambda}\widetilde{V}_q(L)^\lambda.
}

Assume now that equality holds in \eqref{eq:unconditional-logBM-Vq}. By Dubuc's equality characterization \cite[Thm. 12]{Dub77}, there exist $a>0$ and $b\in\R^n$ such that
\eq{
\mathbf{1}_A(z)\phi_q(z)=a^\lambda \mathbf{1}_C(z-\lambda b)\phi_q(z-\lambda b) \quad\text{for a.e. }z\in\R^n,
}
and
\eq{
\mathbf{1}_B(z)\phi_q(z)=a^{-(1-\lambda)}\mathbf{1}_C(z+(1-\lambda)b)\phi_q(z+(1-\lambda)b) \quad\text{for a.e. }z\in\R^n,
}
where $C=(1-\lambda)A+\lambda B$. Therefore, up to null sets,
\eq{
A=C+\lambda b,\quad B=C-(1-\lambda)b.
}
In particular,
\eq{
\label{eq:psi-equality-dubuc-short} 
\phi_q(z)=a^\lambda\phi_q(z-\lambda b)\quad\text{for a.e. }z\in A.
}
Since $A$ is a closed convex set with non-empty interior and both sides of \eqref{eq:psi-equality-dubuc-short} are real analytic, \eqref{eq:psi-equality-dubuc-short} holds on $\operatorname{int}(A)$ and we may differentiate in $\operatorname{int}(A)$.

Next, we compute
\eq{
\partial_i\log\phi_q(z)=1+(q-n)p_i(z),\quad p_i(z):=\frac{e^{2z_i}}{\sum_{j=1}^n e^{2z_j}}.
}
Since $q\neq n$, differentiating \eqref{eq:psi-equality-dubuc-short}, we find
\eq{
p_i(z)=p_i(z-\lambda b) \quad\text{for all }i=1,\dots,n,\ \text{and all }z\in\operatorname{int}(A).
}
Hence, for all distinct $i,j$,
\eq{
\frac{p_i(z-\lambda b)}{p_j(z-\lambda b)} =\frac{p_i(z)}{p_j(z)} \implies 1=e^{2\lambda(b_j-b_i)}.
}
Thus $b_i=b_j$ for all distinct $i,j$ as $\lambda\in (0,1)$. That is, $b=t(1,\dots,1)$ for some $t\in\R$.

Since $A$ and $B$ are closed convex sets and $B=A-b$ up to null sets, we have
\eq{
B=A-t(1,\dots,1).
}
If $x\in K_+$, then $\log x\in A$, and hence
\eq{
\log x-t(1,\dots,1)\in B.
}
Thus  $e^{-t}K_+\subset L_+$. Similarly, $L_+\subset e^{-t}K_+$. Hence $L_+=e^{-t}K_+$, and since $K$ and $L$ are unconditional, this implies $L=e^{-t}K$.
\end{proof}

\section{Brunn--Minkowski inequalities for \texorpdfstring{$q>n$}{q>n}}\label{section:q-greater-than-n}
Let us recall in the context of the dual quermassintegral:
\eq{
W(x)&=(n-q)\log\abs{x},\quad && d\mu(x)=e^{-W(x)}\, dx,\\
\mu(K)&=\int_K e^{-W(x)}\, dx,\quad && d\mu_{\partial K}=e^{-W}\,d\mathcal{H}^{n-1}\big|_{\partial K}.
}

We begin this section with the following discussion. To handle the singularity of $W$ at the origin, one may work throughout with a suitable regularization:
\eq{
W_{\varepsilon}(x)=\frac{1}{2}(n-q)\log(|x|^2+\varepsilon^2),\quad \varepsilon>0.
}

Let $\alpha=q-n$ and consider the elliptic operator
\eq{
L(\cdot)=\Delta_{\mathrm{euc}}(\cdot)-\ip{DW}{D(\cdot)} =\Delta_{\mathrm{euc}}(\cdot)+\alpha\langle \frac{x}{|x|^2}, D(\cdot) \rangle.
}

According to \cite{KM18,KM22} and \cite{KL21}, in order to establish the log-concavity of the functional $\widetilde{V}_q(\cdot)$ along the $L_1$-sum, it is sufficient to prove the non-negativity of
\eq{
I_q(u):=\int_K\left(\|D^2u\|^2+D^2W(Du,Du)\right)\,d\mu
}
for every solution $u$ of the Neumann problem
\eq{
Lu=C_0 \quad\text{in }K,
}
where $C_0$ is a constant satisfying the compatibility condition
\eq{
C_0=\frac{\int_{\partial K}D_{\nu}u\,d\mu_{\partial K}}{\mu(K)}.
}

We begin by decomposing the Euclidean gradient $Du$ on $K\setminus\{0\}$ into its radial and tangential components (with respect to the origin):
\eq{
\partial_ru=\langle Du,\frac{x}{|x|}\rangle, \quad D_Tu=Du-\frac{x}{|x|}\partial_ru.
}
We calculate
\eq{
D^2W=-\frac{\alpha}{|x|^2}\left(I-2\frac{x\otimes x}{|x|^2}\right),
}
and therefore
\eq{
D^2W(Du,Du)=-\alpha\frac{|D_Tu|^2}{|x|^2}+\alpha\frac{|\partial_r u|^2}{|x|^2}.
}
Substituting this identity into the definition of $I_q(u)$ yields
\eq{
\label{eq:Iq} 
I_q(u)=\int_K\left(\|D^2u\|^2 -\alpha\frac{|D_Tu|^2}{|x|^2}+\alpha\frac{|\partial_ru|^2}{|x|^2}\right)\,d\mu.
}
This shows that the cases $\alpha>0$ and $\alpha<0$ are structurally different. When $\alpha>0$ the tangential term provides the unfavorable contribution, while when $\alpha<0$ the tangential contribution becomes favorable and the radial term is the one that may decrease the integrand.

Now suppose that $q<n$, the convex body $K$ is origin-symmetric and that the test function $u$ is even. In this setting, the inequality from \cite[Thm. 4]{CER23} applies and asserts that
\eq{
\int_K |D\partial_iu|^2\,d\mu\ge (n-q)\int_K \frac{(\partial_iu)^2}{|x|^2}\,d\mu \quad\text{for each $i=1,\dots,n$}.
}
Summing these $n$ inequalities yields a lower bound for the norm of the Hessian term:
\eq{
\int_K \|D^2u\|^2\,d\mu\ge (n-q)\int_K \frac{|Du|^2}{|x|^2}\,d\mu.
}
Inserting this estimate into the expression for $I_q(u)$, we find
\eq{
I_q(u) &\ge (n-q)\int_K \frac{|Du|^2}{|x|^2}\,d\mu+(n-q)\int_K \frac{|D_Tu|^2}{|x|^2}\,d\mu+(q-n)\int_K \frac{|\partial_ru|^2}{|x|^2}\,d\mu \\ 
&=2(n-q)\int_K \frac{|D_Tu|^2}{|x|^2}\,d\mu \ge 0.
}

Next we briefly outline a second proof of \autoref{thm:unconditional-logBM-Vq} in the class of unconditional bodies for $q\in(0,n)$, which additionally provides a perspective on the difference between the cases $q<n$ and $q>n$.

If $K$ is unconditional and $u$ is unconditional, then, similarly to the proof of \cite[Thm. 8.3]{KM22}, one can show that
\eq{
\label{eq:diagonal-negative-weight-boundary} 
\int_K\left(\sum_{i=1}^n u_{ii}^2+\alpha\frac{|Du|^2}{|x|^2} \right)|x|^{\alpha}\, dx \ge \int_{\partial K}\frac{u_\nu^2}{h}|x|^{\alpha}\,d\mathcal{H}^{n-1}.
}
Now, since $\|D^2u\|^2=\sum_{i=1}^n u_{ii}^2+2\sum_{i<j}u_{ij}^2$, we have
\eq{
&\|D^2u\|^2-\alpha\frac{|D_Tu|^2}{|x|^2}+\alpha\frac{(\partial_r u)^2}{|x|^2}\\ 
&\quad =2\sum_{i<j}u_{ij}^2-2\alpha\frac{|D_Tu|^2}{|x|^2}+\left(\sum_{i=1}^n u_{ii}^2+\alpha\frac{|Du|^2}{|x|^2}\right).
}
In particular, when $\alpha\le 0$, we obtain
\eq{
\label{eq:weighted-unconditional-q-less-n} 
\int_K\left(\|D^2u\|^2-\alpha\frac{|D_Tu|^2}{|x|^2}+\alpha\frac{(\partial_r u)^2}{|x|^2}\right)|x|^{q-n}\, dx \ge \int_{\partial K}\frac{u_\nu^2}{h}|x|^{q-n}\,d\mathcal{H}^{n-1}.
}
Hence, $\widetilde{V}_q$ is log-concave with respect to $L_0$-sum for $q\in(0,n)$. In fact, by \autoref{prop:q-identity}, \eqref{eq:weighted-unconditional-q-less-n} is equivalent to
\eq{
q\int_{\Sn}(F-\bar{F})^2|X|^{q-n}\, dV_K \le \int_{\Sn} |\nabla F|_g^2|X|^{q-n}\, dV_K,\quad \bar{F}=\frac{\int_{\Sn}F|X|^{q-n}\, dV_K}{\int_{\Sn}|X|^{q-n}\, dV_K}.
}

However, for $\alpha=q-n>0$, in order to prove the non-negativity of the left-hand side of \eqref{eq:weighted-unconditional-q-less-n}, one needs to additionally control the tangential term appearing in the following expression with the off-diagonal terms of $D^2u$:
\eq{
\sum_{i<j}u_{ij}^2-\alpha\frac{|D_Tu|^2}{|x|^2}.
}

Now we prove the log-concavity of $\widetilde{V}_q$ with respect to $L_p$-sum for $q>n$ very close to $n$.

\begin{definition}
For $0<p\le 1$, define
\eq{
T_p:(0,\infty)^n\to (0,\infty)^n,\quad T_p(x)=\left(\frac{x_1^p}{p},\dots,\frac{x_n^p}{p}\right).
}
If $K$ is unconditional, set
\eq{
K_+=K\cap(0,\infty)^n,\quad A_p(K)=T_p(K_+).
}
\end{definition}

\begin{lemma}
\label{lem:power-image-convex}
Let $K$ be an unconditional convex body. Then $A_p(K)$ is convex for every $0< p\le 1$.
\end{lemma}

\begin{proof}
The map $z\mapsto z^p$, $z\in (0,\infty)$, is concave. The claim then follows from \autoref{lem:unconditional-convexity}.
\end{proof}

\begin{lemma}
\label{lem:Lp-change of-coordinates-inclusion}
Let $0<p\le 1$ and $\lambda\in[0,1]$. Suppose $K,L\subset\R^n$ are unconditional convex bodies. Then
\eq{
T_p^{-1}\left((1-\lambda)A_p(K)+\lambda A_p(L)\right)\subset\left((1-\lambda)\cdot K+_p\lambda\cdot L\right)\cap(0,\infty)^n.
}
\end{lemma}

\begin{proof}
The proof is similar to the proof of the $p=0$ case in \cite{Sar15}; see also \cite{BL95}. Let
\eq{
z\in T_p^{-1}\left((1-\lambda)A_p(K)+\lambda A_p(L)\right).
}
Then there exist $x\in K_+$ and $y\in L_+$ such that
\eq{
z_i^p=(1-\lambda)x_i^p+\lambda y_i^p.
}
Note that $z_i=\mathrm{M}_p(x_i,y_i)$, where
\eq{
\mathrm{M}_p(a,b):=\left((1-\lambda)a^p+\lambda b^p\right)^{\frac{1}{p}}.
}
Since $\mathrm{M}_p$ is concave and one-homogeneous on $(0,\infty)^2$, for $v\in [0,\infty)^n$:
\eq{
\sum_{i=1}^n v_i\mathrm{M}_p(x_i,y_i)\le \mathrm{M}_p\left(\sum_{i=1}^n v_ix_i,\sum_{i=1}^n v_iy_i\right).
}
Therefore
\eq{
\langle z,v\rangle\le\left((1-\lambda)\langle x,v\rangle^p+\lambda\langle y,v\rangle^p\right)^{\frac{1}{p}}\le\left((1-\lambda)h_K(v)^p+\lambda h_L(v)^p\right)^{\frac{1}{p}}.
}

For general $v\in\R^n$, if we set $v_+=(|v_1|,\ldots,|v_n|)$, then
\eq{
\langle z,v\rangle\le \langle z,v_+\rangle,\quad h_K(v)=h_K(v_+),\quad h_L(v)=h_L(v_+).
}
Thus $z$ belongs to the $L_p$-sum of $K$ and $L$.
\end{proof}

\begin{lemma}\label{lem:log-sum-power-hessian}
Let $r\ge2$ and $\Phi(y)=\log\left(\sum_{i=1}^n y_i^r\right)$ on $(0,\infty)^n$. Then, for every $\xi\in\R^n$,
\eq{
D^2\Phi(y)[\xi,\xi]\le \frac{r(r-1)}{2}\sum_{i=1}^n\frac{\xi_i^2}{y_i^2}.
}
\end{lemma}

\begin{proof}
Let us put $S=\sum_{i=1}^n y_i^r$, $\theta_i=y_i^r/S$, and $a_i=\xi_i/y_i$. We calculate
\eq{
\partial_i\Phi=\frac{r y_i^{r-1}}{S},\quad
\partial_{ij}^2\Phi=\frac{r(r-1)y_i^{r-2}}{S}\delta_{ij}-\frac{r^2y_i^{r-1}y_j^{r-1}}{S^2}.
}
Therefore
\eq{
D^2\Phi(y)[\xi,\xi]&=r(r-1)\sum_{i=1}^n\theta_i a_i^2-r^2\left(\sum_{i=1}^n\theta_i a_i\right)^2\\
&\le r(r-1)\left[\sum_{i=1}^n\theta_i a_i^2-\left(\sum_{i=1}^n\theta_i a_i\right)^2\right].
}
Now, the claim follows from
\eq{
\sum_{i=1}^n\theta_i a_i^2-\left(\sum_{i=1}^n\theta_i a_i\right)^2&=\sum_{1\le i<j\le n}\theta_i\theta_j(a_i-a_j)^2\\
&\le 2\sum_{1\le i<j\le n}\theta_i\theta_j(a_i^2+a_j^2)\\
&=2\sum_{i=1}^n\theta_i(1-\theta_i)a_i^2\le \frac{1}{2}\sum_{i=1}^n a_i^2.
}
\end{proof}

\begin{lemma}
\label{lem:transformed-density-log-concave}
Let $\alpha>0$ and $0< p\le 1$. Under $y=T_p(x)$, the measure $|x|^\alpha dx$ becomes
\eq{
\rho_{p,\alpha}(y)dy,\quad \rho_{p,\alpha}(y):=p^{\frac{n+\alpha}{p}-n}\left(\sum_{i=1}^n y_i^{\frac{2}{p}}\right)^{\frac{\alpha}{2}}\prod_{i=1}^n y_i^{\frac{1}{p}-1}.
}
The density $\rho_{p,\alpha}$ is log-concave whenever $
\alpha\le \frac{2p(1-p)}{2-p}$.
\end{lemma}

\begin{proof}
Since $x_i=(py_i)^{\frac{1}{p}}$ and $dx_i=p^{\frac{1}{p}-1}y_i^{\frac{1}{p}-1}dy_i$,
\eq{
|x|^\alpha dx=p^{\frac{n+\alpha}{p}-n}\left(\sum_{i=1}^n y_i^{\frac{2}{p}}\right)^{\frac{\alpha}{2}}\prod_{i=1}^n y_i^{\frac{1}{p}-1}dy.
}

Due to \autoref{lem:log-sum-power-hessian},
\eq{
D^2\log\left(\sum_i y_i^{2/p}\right)[\xi,\xi]\le \frac{2-p}{p^2}\sum_i\frac{\xi_i^2}{y_i^2}.
}
Hence
\eq{
D^2\log\rho_{p,\alpha}(y)[\xi,\xi]\le \left(\frac{\alpha(2-p)}{2p^2}-\frac{1-p}{p}\right)\sum_i\frac{\xi_i^2}{y_i^2}.
}
From this the claim follows.
\end{proof}

Let $\alpha_*=6-4\sqrt{2}$, which is the maximum value of $\frac{2p(1-p)}{2-p}$ on $[0,1]$. For $0<\alpha\le \alpha_*$ define
\eq{
p_-(\alpha)=\frac{2+\alpha-\sqrt{\alpha^2-12\alpha+4}}{4},\quad p_+(\alpha)=\frac{2+\alpha+\sqrt{\alpha^2-12\alpha+4}}{4}.
}
Then
\eq{
\alpha\le \frac{2p(1-p)}{2-p}\quad \text{for } p\in[p_-(\alpha),p_+(\alpha)].
}
Moreover, $p_-(\alpha)=\alpha+\frac{1}{2}\alpha^2+O(\alpha^3)$ as $\alpha\to 0$ and $p_-(\alpha_*)=p_+(\alpha_*)=2-\sqrt{2}$.

\begin{theorem}
\label{thm:power-coordinate-PL}
Let $\mu_\alpha(K)=\int_K |x|^\alpha dx$. Let $K,L$ be unconditional convex bodies, and let $\lambda\in(0,1)$. If $0<\alpha\le\alpha_*$ and $p\in[p_-(\alpha),p_+(\alpha)]$, then
\eq{
\mu_\alpha\left((1-\lambda)\cdot K+_p\lambda\cdot L\right)\ge\mu_\alpha(K)^{1-\lambda}\mu_\alpha(L)^\lambda.
}
Moreover, equality holds if and only if $K=L$.
\end{theorem}

\begin{proof}
Let $K_+=K\cap(0,\infty)^n$ and $L_+=L\cap(0,\infty)^n$. Set $A=T_p(K_+)$ and $B=T_p(L_+)$. By \autoref{lem:power-image-convex}, $A$ and $B$ are convex. By \autoref{lem:transformed-density-log-concave}, the density $\rho_{p,\alpha}$ is log-concave. Hence, by the Pr\'ekopa--Leindler inequality,
\eq{
\int_{(1-\lambda)A+\lambda B}\rho_{p,\alpha}(y)\, dy\ge \left(\int_A\rho_{p,\alpha}(y)\, dy\right)^{1-\lambda}\left(\int_B\rho_{p,\alpha}(y)\, dy\right)^\lambda.
}
Pulling back by $T_p$,
\eq{
\int_{T_p^{-1}((1-\lambda)A+\lambda B)} |x|^\alpha dx\ge \left(\int_{K_+}|x|^\alpha dx\right)^{1-\lambda}\left(\int_{L_+}|x|^\alpha dx\right)^\lambda.
}

Due to \autoref{lem:Lp-change of-coordinates-inclusion}, we have
\eq{
T_p^{-1}((1-\lambda)A+\lambda B)\subset \left((1-\lambda)\cdot K+_p\lambda\cdot L\right)\cap(0,\infty)^n.
}
Therefore
\eq{
\int_{\left((1-\lambda)\cdot K+_p\lambda\cdot L\right)\cap(0,\infty)^n}|x|^\alpha dx\ge \left(\int_{K_+}|x|^\alpha dx\right)^{1-\lambda}\left(\int_{L_+}|x|^\alpha dx\right)^\lambda.
}
Since $K$, $L$, and $(1-\lambda)\cdot K+_p\lambda\cdot L$ are unconditional, and since $|x|^\alpha$ is unconditional, all orthants contribute equally. Multiplying by $2^n$ proves the inequality.

By Dubuc's equality characterization, up to null sets, 
\eq{
A=C+\lambda b,\quad B=C-(1-\lambda)b
}
for some $b\in\R^n$. Therefore, $B=A-b$ up to null sets. Since $A$ and $B$ are convex sets, we have $\bar{B}=\bar{A}-b$. Since $A$ and $B$ are contained in $(0,\infty)^n$ and their closures both contain the origin, we must have $b=0$. Moreover, $T_p$ extends continuously to a homeomorphism of $[0,\infty)^n$ onto itself, and thus
\eq{
\bar{A}=T_p\left(K\cap[0,\infty)^n\right)=T_p\left(L\cap[0,\infty)^n\right)=\bar{B}\implies K=L.
} 
\end{proof}

\begin{proof}[Proof of \autoref{thm:Lp-homogeneous-BM}]
Set $a=\widetilde{V}_q(K)^{\frac{p}{q}}$, $b=\widetilde{V}_q(L)^{\frac{p}{q}}$, and define
\eq{
K_0=a^{-\frac{1}{p}}K, \quad L_0=b^{-\frac{1}{p}}L.
}
Since $\widetilde{V}_q$ is $q$-homogeneous, we have
\eq{
\widetilde{V}_q(K_0)=a^{-\frac{q}{p}}\widetilde{V}_q(K)=1,\quad \widetilde{V}_q(L_0)=b^{-\frac{q}{p}}\widetilde{V}_q(L)=1.
}

Let
\eq{
1-\lambda=\frac{a}{a+b},
\quad
\lambda=\frac{b}{a+b}.
}
By the definition of the $L_p$-sum,
\eq{
(1-\lambda)\cdot K_0+_p\lambda\cdot L_0=(a+b)^{-\frac{1}{p}}(K+_pL).
}
Indeed, we have
\eq{
\left((1-\lambda)h_{K_0}^p+\lambda h_{L_0}^p\right)^{\frac{1}{p}}=\left(\frac{1}{a+b}h_K^p+\frac{1}{a+b}h_L^p\right)^{\frac{1}{p}}=(a+b)^{-\frac{1}{p}}(h_K^p+h_L^p)^{\frac{1}{p}}.
}

Applying \autoref{thm:power-coordinate-PL} to $K_0$ and $L_0$,
\eq{
\widetilde{V}_q\left((1-\lambda)\cdot K_0+_p\lambda\cdot L_0\right)\ge \widetilde{V}_q(K_0)^{1-\lambda}\widetilde{V}_q(L_0)^\lambda=1.
}
Using the identity above and $q$-homogeneity again, we obtain
\eq{
(a+b)^{-\frac{q}{p}}\widetilde{V}_q(K+_pL)\ge1.
}
Therefore
\eq{
\widetilde{V}_q(K+_pL)^{\frac{p}{q}}\ge a+b=\widetilde{V}_q(K)^{\frac{p}{q}}+\widetilde{V}_q(L)^{\frac{p}{q}}.
}

Now assume equality holds. Then
\eq{
\widetilde{V}_q\left((1-\lambda)\cdot K_0+_p\lambda\cdot L_0\right)=1.
}
By the equality case in \autoref{thm:power-coordinate-PL}, this implies
$K_0=L_0$. Hence $K$ and $L$ are dilates.
\end{proof}

For the range $q\in(n,n+3-2\sqrt{2}]$, we give an alternative proof, based on the Reilly approach, of the log-concavity of $\widetilde{V}_q$ with respect to the $L_1$-sum.

\begin{lemma}
\label{lem:1d-inequality}
Let $\alpha\in(-1,1)$ and $c>0$. If $f\in C^1([0,c])$ satisfies $f(0)=0$, then
\eq{
\int_0^c |f'(s)|^2 s^\alpha\,ds \ge \frac{(1-\alpha)^2}{4}\int_0^c f(s)^2 s^{\alpha-2}\,ds.
}
\end{lemma}

\begin{proof} The proof is similar to that of \cite[Lem. 8]{CER23}, with $v=0$. Since
\eq{
\int_0^c \left|f'(s)-\frac{1-\alpha}{2}\frac{f(s)}{s}\right|^2 s^\alpha\,ds \ge 0,
}
integrating by parts gives
\eq{
\int_0^c |f'(s)|^2 s^\alpha\,ds \ge \frac{1-\alpha}{2}c^{\alpha-1}f(c)^2+\frac{(1-\alpha)^2}{4}\int_0^c f(s)^2 s^{\alpha-2}\,ds.
}
\end{proof}

\begin{corollary}\label{corollary:odd-weighted-slice}
Let $K$ be an unconditional convex body, $\alpha\in\left(0,1\right)$, and let $u\in C^1\left(K\right)$ be odd in the $i$-th coordinate. Then for every $\delta\ge 0$:
\eq{
\label{eq:odd-weighted-slice} 
\int_K \left|\partial_i u\right|^2 \left(\left|x\right|^2+\delta^2\right)^{\frac{\alpha}{2}}\, dx \ge \frac{\left(1-\alpha\right)^2}{4}\int_K u^2 \left(\left|x\right|^2+\delta^2\right)^{\frac{\alpha-2}{2}}\, dx.
}
\end{corollary}

\begin{proof}
We integrate over coordinate slices of $K$, which are intervals; see also \cite[Lem. 8.2]{KM22}. Without loss of generality, we may assume that $i=1$. Since $K$ is unconditional and $u$ is odd in the first coordinate, $u\left(0,x'\right)=0$, where $x'=\left(x_2,\dots,x_n\right)$. Let $K^+=K\cap\left\{x_1>0\right\}$.

For $x\in K^+$ we have
\eq{
\left(\left|x\right|^2+\delta^2\right)^{\frac{\alpha}{2}}\ge x_1^\alpha \quad\text{and}\quad \left(\left|x\right|^2+\delta^2\right)^{\frac{\alpha-2}{2}}\le x_1^{\alpha-2}.
}
Thus it suffices to prove
\eq{
\int_{K^+}\left|\partial_1 u\right|^2 x_1^\alpha\, dx \ge \frac{\left(1-\alpha\right)^2}{4}\int_{K^+}u^2 x_1^{\alpha-2}\, dx.
}
The left-hand side of \eqref{eq:odd-weighted-slice} equals twice
\eq{
\int_{\R^{n-1}}\left(\int_{\left\{s\ge0:\ \left(s,x'\right)\in K^+\right\}}\left|\partial_1 u\left(s,x'\right)\right|^2 \left(s^2+\left|x'\right|^2+\delta^2\right)^{\frac{\alpha}{2}}\,ds\right)\, dx'.
}
The right-hand side can be rewritten similarly. Now \autoref{lem:1d-inequality} yields
\eq{
\int_{\left\{s\ge0:\ \left(s,x'\right)\in K^+\right\}}\left|\partial_1 u\left(s,x'\right)\right|^2 s^\alpha\,ds \ge \frac{\left(1-\alpha\right)^2}{4} \int_{\left\{s\ge 0:\ \left(s,x'\right)\in K^+\right\}}u\left(s,x'\right)^2 s^{\alpha-2}\,ds.
}
Now integrating over the projection of $K^+$ onto the $x'$-variables proves the claim.
\end{proof}

\begin{proposition}\label{thm:log-concavity-q-close-to-n}
Let $q\in(n,n+3-2\sqrt{2}]$. Then, in the class of unconditional bodies:
\eq{
\widetilde{V}_q(K+L)^{1/q}\ge \widetilde{V}_q(K)^{1/q}+\widetilde{V}_q(L)^{1/q}.
}
\end{proposition}

\begin{proof}
Let $\varepsilon>0$ and $\alpha=q-n$. We introduce the regularized weight
\eq{
W_{\varepsilon}(x)=-\frac{\alpha}{2}\log\left(|x|^2+\varepsilon^2\right),
}
together with the associated weighted measure
\eq{
d\mu_{\varepsilon}(x)=e^{-W_{\varepsilon}(x)}\, dx=\left(|x|^2+\varepsilon^2\right)^{\frac{\alpha}{2}}\, dx.
}
We also define the differential operator
\eq{
L_{\varepsilon} u=\Delta_{\mathrm{euc}}u-\ip{DW_{\varepsilon}}{Du} =\Delta_{\mathrm{euc}}u+\alpha\langle \frac{x}{|x|^2+\varepsilon^2},Du\rangle.
}

We now show that, for every unconditional $C^2$ function $u$ on the unconditional convex body $K$ satisfying $L_{\varepsilon} u=C$ for some constant $C$, one has
\eq{
I_{q,\varepsilon}(u):= \int_K\left(\|D^2u\|^2+D^2W_{\varepsilon}(Du,Du)\right)\,d\mu_{\varepsilon} \ge 0.
}
This non-negativity then implies the log-concavity of $\mu_{\varepsilon}$ with respect to the $L_1$-sum on unconditional convex bodies. Passing to the limit as $\varepsilon\to 0$, we deduce the log-concavity of $\widetilde{V}_q$ with respect to the $L_1$-sum on unconditional convex bodies.

By \autoref{corollary:odd-weighted-slice}, if $v\in C^1\left(K\right)$ is odd in the $i$-th coordinate, then
\eq{
\int_K |\partial_i v|^2\,d\mu_{\varepsilon} \ge \frac{(1-\alpha)^2}{4} \int_K v^2\left(|x|^2+\varepsilon^2\right)^{\frac{\alpha-2}{2}}\, dx.
}
Applying this to $v=\partial_iu$ and summing over all coordinate directions yields
\eq{
\int_K \|D^2u\|^2\,d\mu_{\varepsilon} \ge \frac{(1-\alpha)^2}{4} \int_K |Du|^2\left(|x|^2+\varepsilon^2\right)^{\frac{\alpha-2}{2}}\, dx.
}

On the other hand, a direct computation shows that
\eq{
D^2W_{\varepsilon} =-\alpha\left(\frac{I}{|x|^2+\varepsilon^2}-2\frac{x\otimes x}{(|x|^2+\varepsilon^2)^2}\right),
}
and hence on $K\setminus\{0\}$:
\eq{
D^2W_{\varepsilon}(Du,Du) &=-\alpha\frac{|D_Tu|^2}{|x|^2+\varepsilon^2}+\alpha\frac{|x|^2-\varepsilon^2}{(|x|^2+\varepsilon^2)^2}|\partial_r u|^2.
}

Combining these estimates, we obtain
\eq{
I_{q,\varepsilon}(u) &\ge \left(\frac{(1-\alpha)^2}{4}-\alpha\right) \int_K |D_Tu|^2\left(|x|^2+\varepsilon^2\right)^{\frac{\alpha-2}{2}}\, dx \\ 
&\quad+\int_K\left(\frac{(1-\alpha)^2}{4}+\alpha\frac{|x|^2-\varepsilon^2}{|x|^2+\varepsilon^2}\right)|\partial_r u|^2\left(|x|^2+\varepsilon^2\right)^{\frac{\alpha-2}{2}}\, dx.
}
Since $\frac{|x|^2-\varepsilon^2}{|x|^2+\varepsilon^2}\ge -1$, the coefficient of the radial term is at least $\frac{(1-\alpha)^2}{4}-\alpha$ as well. Therefore
\eq{
I_{q,\varepsilon}(u) \ge \left(\frac{(1-\alpha)^2}{4}-\alpha\right) \int_K |Du|^2\left(|x|^2+\varepsilon^2\right)^{\frac{\alpha-2}{2}}\, dx\ge 0.
}
\end{proof}

\begin{proof}[Proof of \autoref{thm:SZ-poincare}] By approximation, we may assume that $F\in C^{\infty}(\Sn)$.

\emph{Case 1:} If $0<q\le1$, the first asserted inequality is trivial. Suppose $K$ is origin-symmetric and $F$ is even. In this case, we may assume $q\in (1,n)$. Recall our setup at the beginning of \autoref{sec:boundary-variation-to-centro-affine-formulation}. Due to \cite[Thm. 2]{SZ25}, the map $s\mapsto \widetilde{V}_q\left(sK+(1-s)L\right)^{\frac{1}{q}}$ is concave on $[0,1]$ in the class of origin-symmetric convex bodies. Therefore, for every even linear deformation $h_s=h+s\psi$ with $\psi=hF$, the measure $\mu(K)=\frac{n}{q}\widetilde{V}_q(K)$ satisfies
\eq{
\left.\frac{d^2}{ds^2}\mu(K_s)\right|_{s=0} \le \frac{q-1}{q}\frac{\left(\left.\frac{d}{ds}\mu(K_s)\right|_{s=0}\right)^2}{\mu(K)}.
}
The claim now follows from \autoref{prop:q-identity} and the identity (cf. \cite[Thm. 6.6]{KM18})
\eq{
\left.\frac{d}{ds}\mu(K_s)\right|_{s=0}=\int_{\Sn}F|X|^{q-n}\, dV_K.
}

\emph{Case 2:} Suppose $K$ is unconditional and $F$ is unconditional. For $|s|$ sufficiently small, let
\eq{
h_s=he^{sF},
}
and $K_s$ be the $C^2_+$ unconditional convex body whose support function is $h_s$. We have
\eq{
h_{(1-\lambda)s+\lambda t}=h_s^{1-\lambda}h_t^{\lambda},
}
and hence $K_{(1-\lambda)s+\lambda t}=(1-\lambda)\cdot K_s+_0 \lambda\cdot K_t$, for $|s|,|t|$ sufficiently small. Therefore, due to \autoref{thm:unconditional-logBM-Vq}, the map
\eq{
s\mapsto \log \widetilde{V}_q(K_s)
}
is concave in a neighborhood of $0$. In particular, we have
\eq{
\label{eq:log-concavity-Vq-unconditional} 
\widetilde{V}_q(K)\left.\frac{d^2}{ds^2}\widetilde{V}_q(K_s)\right|_{s=0} \le \left(\left.\frac{d}{ds}\widetilde{V}_q(K_s)\right|_{s=0}\right)^2.
}

We now compute the first and second derivatives at $s=0$. Before invoking \autoref{prop:q-identity}, we make a few preliminary observations that simplify the calculation.

Consider the map
\eq{
f\mapsto \Phi(f):=\widetilde{V}_q(K_{h+f}),
}
defined for $f$ in a sufficiently small $C^2$-neighborhood of $0$. Here $K_{h+f}$ denotes the $C^2_+$ convex body whose support function is $h+f$. We have
\eq{
\widetilde{V}_q(K_s)=\Phi(f_s), \quad f_s:=h_s-h=h(e^{sF}-1).
}

Note that
\eq{
f_s&=shF+\frac{s^2}{2}hF^2+o(s^2),\\ 
\Phi(f_s)&=\Phi(0)+D\Phi_0[f_s]+\frac{1}{2} D^2\Phi_0[f_s,f_s]+o(s^2).
}
Here $D\Phi_0=D\Phi|_{s=0}$ and $D^2\Phi_0=D^2\Phi|_{s=0}$. Therefore
\eq{
\Phi(f_s)=\Phi(0)+sD\Phi_0[hF]+\frac{s^2}{2}D\Phi_0[hF^2]+\frac{s^2}{2}D^2\Phi_0[hF,hF]+o(s^2)
}
and
\eq{
\label{eq:first-second-by-taylor-expansion} 
\left.\frac{d}{ds}\right|_{s=0}\widetilde{V}_q(K_s)=D\Phi_0[hF], \quad \left.\frac{d^2}{ds^2}\right|_{s=0}\widetilde{V}_q(K_s) =D^2\Phi_0[hF,hF]+D\Phi_0[hF^2].
}

By the first variation formula for $\widetilde{V}_q$,
\eq{
D\Phi_0[hG]=\frac{q}{n}\int_{\Sn}G|X|^{q-n}\, dV_K \quad\text{for every }G\in C^2(\Sn).
}
Now, applying this with $G=F$ and $G=F^2$, we obtain
\eq{
\label{eq:first-variation-Vq-unconditional} 
\left.\frac{d}{ds}\right|_{s=0}\widetilde{V}_q(K_s) =\frac{q}{n}\int_{\Sn}F|X|^{q-n}\, dV_K,
}
and
\eq{
\label{eq:first-var-hF2} 
D\Phi_0[hF^2]=\frac{q}{n}\int_{\Sn}F^2|X|^{q-n}\, dV_K.
}

Next, by applying \autoref{prop:q-identity} to the linear perturbation $s\mapsto h+s(hF)$, we have
\eq{
\label{eq:additive-second-var-q} 
D^2\Phi_0[hF,hF]=\frac{q}{n}\int_{\Sn} \left((q-1)F^2-|\nabla F|_g^2\right)|X|^{q-n}\, dV_K.
}
Putting \eqref{eq:first-second-by-taylor-expansion}, \eqref{eq:first-var-hF2}, and \eqref{eq:additive-second-var-q} together, we obtain
\eq{
\label{eq:second-variation-Vq-unconditional} 
\left.\frac{d^2}{ds^2}\right|_{s=0}\widetilde{V}_q(K_s) =\frac{q}{n}\int_{\Sn}\left(qF^2-|\nabla F|_g^2\right)|X|^{q-n}\, dV_K.
}
Finally, substituting \eqref{eq:first-variation-Vq-unconditional} and
\eqref{eq:second-variation-Vq-unconditional} into
\eqref{eq:log-concavity-Vq-unconditional} proves the claim.
\end{proof}

\section{Uniqueness}
We now derive uniqueness consequences for the $L_{p,q}$-Minkowski problem in the class of $C^2_+$ unconditional convex bodies.

\begin{theorem}
\label{thm:unconditional-p-q-inequalities}
Let $K,L$ be $C^2_+$ unconditional convex bodies.

\begin{enumerate}
\item[\rm(i)] For $q=n+2$,
\eq{
\label{eq:minkowski-type-uniqueness} 
\frac{1}{n}\int_{\Sn} h_L|Dh_K|^2\,dS_K \ge \widetilde{V}_{n+2}(K)^{\frac{n+1}{n+2}} \widetilde{V}_{n+2}(L)^{\frac{1}{n+2}}.
}
Equality holds if and only if $K$ and $L$ are dilates.

\item[\rm(ii)] Let $q\in(0,n)$. Then
\eq{
\label{eq:unconditional-log-Vq} 
\frac{1}{n}\int_{\Sn} \log\frac{h_L}{h_K}|Dh_K|^{q-n}\, dV_K \ge \frac{\widetilde{V}_q(K)}{q}\log\frac{\widetilde{V}_q(L)}{\widetilde{V}_q(K)}.
}
Moreover, for every $p\in(0,q]$,
\eq{
\label{eq:unconditional-mixed-Vq-p-unified} 
\frac{1}{n}\int_{\Sn} h_L^ph_K^{1-p}|Dh_K|^{q-n}\,dS_K \ge \widetilde{V}_q(K)^{1-\frac{p}{q}}\widetilde{V}_q(L)^{\frac{p}{q}}.
}
Equality in either \eqref{eq:unconditional-log-Vq} or \eqref{eq:unconditional-mixed-Vq-p-unified} holds if and only if $K$ and $L$ are dilates.
\end{enumerate}
\end{theorem}

\begin{proof} The proof is the same as that of \cite[Prop. 4.1]{XZ22}. We first prove \rm(i). Set
\eq{
\phi(t)=\widetilde{V}_{n+2}\left((1-t)K+tL\right)^{\frac{1}{n+2}}, \quad t\in[0,1].
}
In view of \autoref{cor:main-unconditional-poincare}, the function $\phi$ is concave. Hence
\eq{
\label{eq:concavity-derivative-step-unified} 
\phi'(0)\ge \phi(1)-\phi(0) =\widetilde{V}_{n+2}(L)^{\frac{1}{n+2}} -\widetilde{V}_{n+2}(K)^{\frac{1}{n+2}}.
}
Let $K_t=(1-t)K+tL$. We have
\eq{
\left.\frac{d}{dt}\widetilde{V}_{n+2}(K_t)\right|_{t=0} =\frac{n+2}{n}\int_{\Sn}h_L|Dh_K|^2\,dS_K-(n+2)\widetilde{V}_{n+2}(K).
}
Therefore
\eq{
\phi'(0)=\widetilde{V}_{n+2}(K)^{-\frac{n+1}{n+2}} \left(\frac{1}{n}\int_{\Sn}h_L|Dh_K|^2\,dS_K-\widetilde{V}_{n+2}(K) \right).
}
Substituting this into \eqref{eq:concavity-derivative-step-unified} and multiplying by
$\widetilde{V}_{n+2}(K)^{\frac{n+1}{n+2}}$, we get
\eq{
\frac{1}{n}\int_{\Sn}h_L|Dh_K|^2\,dS_K-\widetilde{V}_{n+2}(K) \ge \widetilde{V}_{n+2}(K)^{\frac{n+1}{n+2}} \widetilde{V}_{n+2}(L)^{\frac{1}{n+2}} -\widetilde{V}_{n+2}(K).
}

If $K$ and $L$ are not dilates, then, by the characterization of equality cases in \autoref{thm:main-dualquermass-q}, $\phi'(0)>\phi(1)-\phi(0)$, and the inequality is strict.

We now prove \rm(ii). For $\lambda\in[0,1]$, set
\eq{
h_{\lambda}=h_K^{1-\lambda}h_L^\lambda, \quad K_{\lambda}=(1-\lambda)\cdot K+_0 \lambda\cdot L.
}
By definition, $K_{\lambda}$ is the Wulff shape generated by $h_{\lambda}$. In general, $h_{\lambda}$ need not be the support function of $K_{\lambda}$ pointwise. For the first variation formula we invoke \cite[Thm. 6.4]{LYZ18}; see also \cite[Lem. 2.1]{XZ22}.

By \autoref{thm:unconditional-logBM-Vq}, the function $\phi(\lambda):=\log \widetilde{V}_q(K_{\lambda})$ is concave on $[0,1]$. Hence
\eq{
\phi'(0)\ge \phi(1)-\phi(0) =\log\frac{\widetilde{V}_q(L)}{\widetilde{V}_q(K)}.
}
Since $\log h_{\lambda}=\log h_K+\lambda \log\frac{h_L}{h_K}$, the first variation formula for $\widetilde{V}_q$ implies that
\eq{
\left.\frac{d}{d\lambda}\right|_{\lambda=0}\widetilde{V}_q(K_{\lambda})=\frac{q}{n}\int_{\Sn} h_K \log\frac{h_L}{h_K}|Dh_K|^{q-n}\,dS_K.
}
Therefore
\eq{
\frac{1}{\widetilde{V}_q(K)} \frac{q}{n}\int_{\Sn} h_K \log\frac{h_L}{h_K}|Dh_K|^{q-n}\,dS_K \ge \log\frac{\widetilde{V}_q(L)}{\widetilde{V}_q(K)},
}
which is \eqref{eq:unconditional-log-Vq}. Finally \eqref{eq:unconditional-mixed-Vq-p-unified} follows from \eqref{eq:unconditional-log-Vq} and Jensen's inequality.

Equality in \eqref{eq:unconditional-log-Vq} implies that $\phi$ is affine on $[0,1]$, and hence equality holds in \autoref{thm:unconditional-logBM-Vq} and thus $K$ and $L$ are dilates. Equality in \eqref{eq:unconditional-mixed-Vq-p-unified} implies equality in Jensen's inequality, so $(h_L/h_K)^p$ is constant on $\Sn$, and thus again $K$ and $L$ are dilates.
\end{proof}

\begin{proof}[Proof of \autoref{thm:uniqueness-Lpq-unconditional}]
Assume that $K$ and $L$ are $C^2_+$ unconditional convex bodies solving
\eq{
\label{eq:Lpq-eq-proof-unified} 
|Dh|^{q-n}h^{1-p}\det(\bar{\nabla}^2 h+hI)=f.
}
We prove that $K=L$.

Consider the case $q=n+2$ and $1\le p<q$. By \eqref{eq:Lpq-eq-proof-unified} for $K$,
\eq{
\frac{1}{n}\int_{\Sn} h_L|Dh_K|^2\,dS_K =\frac{1}{n}\int_{\Sn} h_L h_K^{p-1}f\, dx.
}
By \eqref{eq:minkowski-type-uniqueness} and H\"older's inequality,
\eq{
\widetilde{V}_{n+2}(K)^{\frac{n+1}{n+2}} \widetilde{V}_{n+2}(L)^{\frac{1}{n+2}} &\le \frac{1}{n}\int_{\Sn} h_L h_K^{p-1}f\, dx\\ 
&\le \left(\frac{1}{n}\int_{\Sn} h_K^p f\, dx\right)^{\frac{p-1}{p}} \left(\frac{1}{n}\int_{\Sn} h_L^p f\, dx\right)^{\frac{1}{p}}.
}
Now using
\eq{
\widetilde{V}_{n+2}(K)=\frac{1}{n}\int_{\Sn} h_K^p f\, dx, \quad \widetilde{V}_{n+2}(L)=\frac{1}{n}\int_{\Sn} h_L^p f\, dx,
}
we find
\eq{
\widetilde{V}_{n+2}(K)^{\frac{n+1}{n+2}} \widetilde{V}_{n+2}(L)^{\frac{1}{n+2}} \le \widetilde{V}_{n+2}(K)^{\frac{p-1}{p}} \widetilde{V}_{n+2}(L)^{\frac{1}{p}}.
}
Interchanging the roles of $K$ and $L$ shows that $\widetilde{V}_{n+2}(K)=\widetilde{V}_{n+2}(L)$. Hence equality holds in H\"older's inequality and in
\eqref{eq:minkowski-type-uniqueness}. Thus $K$ and $L$ are dilates, and since $p<q$, the dilation factor must be $1$.

The remaining cases, corresponding to $q\in(0,n)$ and $0\le p<q$, are treated in the same way. For $p=0$ we use \eqref{eq:unconditional-log-Vq}, while for $0<p<q$ we use \eqref{eq:unconditional-mixed-Vq-p-unified}, together with the corresponding equality cases.
\end{proof}

\section{Logarithmic centro-affine geometry}\label{sec:log-centro-affine-geometry}

Our main reference for affine geometry in this section is \cite{NS94}. Throughout this section, $K\subset\R^n$ is an unconditional, smooth, strictly  convex body. We write
\eq{
\Omega_+=\{x\in\Sn:\ X_i(x)>0,\ i=1,\dots,n\}.
}
On $\Omega_+$, we define a smooth immersion
\eq{
Y:\Omega_+\to \R^n,\quad Y=\log X:=(\log X_1,\dots,\log X_n),
}
and we set $\bm{\xi}=(1,\dots,1)$, $P=\operatorname{diag}(X_1,\dots,X_n)$ and $Q=\operatorname{diag}(x_1,\ldots,x_n)$.

Since $dY=P^{-1}dX$, for every $v\in T_x\Sn$ we have
\eq{
\langle dY(v),Px\rangle=\langle P^{-1}dX(v),Px\rangle=\langle dX(v),x\rangle=0.
}
Thus a Euclidean unit normal to $Y(\Omega_+)$ is
\eq{
\mathbf{n}=\frac{Px}{\sigma},\quad \text{where }\sigma=\left(\sum_{i=1}^nX_i^2x_i^2\right)^{\frac{1}{2}}.
}
This implies that
\eq{
\langle\bm{\xi},\mathbf{n}\rangle=\frac{\langle\bm{\xi},Px\rangle}{\sigma}=\frac{h}{\sigma}.
}
Hence $\bm{\xi}$ is transverse to $Y(\Omega_+)$.

The vector $\bm{\xi}$ induces a torsion-free connection $\nabla^{\log}$ and a symmetric bilinear form $g_{\log}$:
\eq{\label{eq:log-gauss-formula}
D_vdY(w)=dY(\nabla^{\log}_vw)-g_{\log}(v,w)\bm{\xi}.
}

The choice of the transversal field $\bm{\xi}=(1,\dots,1)$ is natural. In centro-affine geometry, the transversal vector field is the position vector $X$, which may be interpreted as the infinitesimal direction of dilations $X\mapsto e^tX$. After passing to logarithmic coordinates $Y=\log X$, these dilations become translations $Y\mapsto Y+t\bm{\xi}$. Thus $\bm{\xi}$ is precisely the logarithmic representative of the centro-affine transversal vector field.

\begin{lemma}\label{lem:log-affine-metric}
We have
\eq{
\label{eq:glog-formula}
g_{\log}(v,w)=g(v,w)+\sum_{i=1}^n\frac{x_iX_i}{h}d\log X_i(v)d\log X_i(w).
}
In particular, $g_{\log}> g$ on $\Omega_+$.
\end{lemma}

\begin{proof}
Let $\mathrm{II}(v,w)=\langle D_v\mathbf{n},dY(w)\rangle$ be the Euclidean second fundamental form of $Y(\Omega_+)$. Since $\langle Px,dY(w)\rangle=0$, we have
\eq{
\mathrm{II}(v,w)=\frac{1}{\sigma}\left\langle D_v(Px),P^{-1}dX(w)\right\rangle.
}

On the other hand,
\eq{
D_v(Px)=Pv+Q dX(v).
}
Moreover, we have
\eq{
\left\langle Pv,P^{-1}dX(w)\right\rangle&=\langle v,dX(w)\rangle=\bar{g}(v,\tau w),\\
\left\langle QdX(v),P^{-1}dX(w)\right\rangle&=\sum_{i=1}^n\frac{x_i}{X_i}dX_i(v)dX_i(w).
}
Hence
\eq{
\mathrm{II}(v,w)=\frac{1}{\sigma}\left(\bar{g}(v,\tau w)+\sum_{i=1}^n\frac{x_i}{X_i}dX_i(v)dX_i(w)\right).
}

Taking the Euclidean scalar product of \eqref{eq:log-gauss-formula} with $\mathbf{n}$, we find
\eq{
-g_{\log}(v,w)\langle\bm{\xi},\mathbf{n}\rangle=\langle D_vdY(w),\mathbf{n}\rangle=-\mathrm{II}(v,w).
}
Thus
\eq{
g_{\log}(v,w)=\frac{\mathrm{II}(v,w)}{\langle\bm{\xi},\mathbf{n}\rangle}=\frac{\sigma}{h}\mathrm{II}(v,w).
}
Due to $\bar{g}(v,\tau w)=hg(v,w)$, \eqref{eq:glog-formula} follows.

Since $dX=\tau$ is injective, by \autoref{lem:sign-xi-Xi}, the second term in \eqref{eq:glog-formula} is positive definite. Consequently, $g_{\log}>g$ on $\Omega_+$.
\end{proof}

\begin{proposition}\label{prop:g-log-g-polar-log}
Let $K\subset\mathbb R^n$ be a $C^2_+$ unconditional convex body. Put
\eq{
y=\frac{x}{h(x)},\quad p_i=X_i y_i,\quad u_i=\log X_i,\quad v_i=\log y_i.
}
Then on $\Omega_+$:
\eq{
\sum_i p_idu_i=0,\quad \sum_i p_idv_i=0,\quad g=\sum_i p_idu_idv_i.
}
Define $g_{\log}=g+\sum_i p_idu_i^2$ and $ g_{\log}^{\circ}=g+\sum_i p_idv_i^2$. Then
\eq{
g_{\log}=\sum_i du_idp_i,\quad g_{\log}^{\circ}=\sum_i dv_idp_i.
}
Moreover, if $g_{\mathrm{av}}:=\frac{1}{2}\left(g_{\log}+g_{\log}^{\circ}\right)$, then
\eq{
 g_{\mathrm{av}}=\frac{1}{2}\sum_i\frac{1}{p_i}dp_i^2.
}
In particular, $ g_{\mathrm{av}}\ge 2g$. For a coordinate ellipsoid,
\eq{
 g_{\mathrm{av}}=g_{\log}=g_{\log}^{\circ}=2g.
}
\end{proposition}

\begin{proof}
The first set of identities was proved in \cite[p. 28]{Mil23}. From $p_i=X_i y_i$ we have
\eq{
dp_i=p_i(du_i+dv_i).
}
Therefore
\eq{
\sum_i du_idp_i=\sum_i p_idu_i^2+\sum_i p_idu_idv_i=g+\sum_i p_idu_i^2=g_{\log}.
}
Interchanging $u_i$ and $v_i$, we obtain
\eq{
\sum_i dv_idp_i=\sum_i p_idv_idu_i+\sum_i p_idv_i^2=g+\sum_i p_idv_i^2=g_{\log}^{\circ}.
}
Now using again $dp_i=p_i(du_i+dv_i)$, we obtain
\eq{
 g_{\mathrm{av}}=\frac{1}{2}\sum_i (du_i+dv_i)dp_i=\frac{1}{2}\sum_i\frac{1}{p_i}dp_i^2.
}
Finally,
\eq{
 g_{\mathrm{av}}-2g=\frac{1}{2}\sum_i p_i(du_i+dv_i)^2-2\sum_i p_i du_idv_i=\frac{1}{2}\sum_i p_i(du_i-dv_i)^2\ge 0.
}

For a coordinate ellipsoid $\{z\in\mathbb R^n:\sum_i z_i^2/a_i^2 \le 1\}$ we have $y_i=X_i/a_i^2$ and hence $dv_i=du_i$. That is, $g_{\log}=g_{\log}^{\circ}= g_{\mathrm{av}}=2g$.
\end{proof}

\begin{proposition}\label{prop:g-log-plus-g-polar-log-is-round}
Let $\Delta^{n-1}_+=\{p=(p_1,\ldots,p_n): p_i>0,\ \sum_{i=1}^n p_i=1\}$. On $\Delta^{n-1}_+$ consider the Riemannian metric $ g_{\mathrm{av}}=\frac{1}{2}\sum_i\frac{1}{p_i}dp_i^2$. Let $\tilde{\nabla}$ be its Levi-Civita connection. Then $\left(\Delta^{n-1}_+, g_{\mathrm{av}}\right)$ is locally isometric to the round sphere of radius $\sqrt{2}$.
\end{proposition}

\begin{proof}
Define
\eq{
f: \Delta^{n-1}_+\to S^{n-1}_{\sqrt{2},+}, \quad f(p)=\left(\sqrt{2p_1},\ldots,\sqrt{2p_n}\right),
}
where
\eq{
S^{n-1}_{\sqrt{2},+}=\left\{z\in\mathbb R^n:z_i>0,\ \sum_i z_i^2=2\right\}.
}
Note that $f$ is a diffeomorphism from $\Delta^{n-1}_+$ onto its image. Moreover, we have
\eq{
f^{\ast}\left(\sum_i dz_i^2\right)=\sum_i df_i^2=\frac{1}{2}\sum_i\frac{1}{p_i}dp_i^2= g_{\mathrm{av}}.
}
Thus $\left(\Delta^{n-1}_+, g_{\mathrm{av}}\right)$ is locally isometric to the round sphere of radius $\sqrt{2}$.
\end{proof}

By the affine Weingarten formula,
\eq{
D_v\bm{\xi}=dY(\mathcal{S}v)-\theta(v)\bm{\xi},
}
where $\mathcal{S}$ is called the affine shape operator and $\theta$ is called the transversal connection form. Therefore we have $\mathcal{S}=0$ and $\theta=0$. Now the Gauss equation gives $\operatorname{Rm}^{\nabla^{\log}}=0$.

The structure $(g_{\log},\nabla^{\log})$ is an equiaffine structure, i.e. $\nabla^{\log}g_{\log}$ is Codazzi and the induced volume form $dV_{\log}$ satisfies $\nabla^{\log}dV_{\log}=0$; see \cite[pp. 30--32]{NS94}. 

Let $\nabla^{\log,*}$ denote the connection conjugate to $\nabla^{\log}$ with respect to $g_{\log}$, that is,
\eq{\label{eq:log-conjugacy}
U(g_{\log}(V,W))=g_{\log}(\nabla^{\log}_UV,W)+g_{\log}(V,\nabla_U^{\log,*}W).
}
The connection $\nabla^{\log,*}$ is torsion-free; see \cite[p. 21]{NS94}.

For $F\in C^\infty(\Omega_+)$ define
\eq{
\Delta_{\log}F=\operatorname{tr}_{g_{\log}}\left((\nabla^{\log,*})^2F\right).
}

\begin{lemma}\label{lem:log-affine-volume}
We have $dV_{\log}=(X_1\cdots X_n)^{-1}\, dV_K$. Equivalently, if we set $\eta:=\sum_{i=1}^nY_i$, then
\eq{
e^{\eta}\, dV_{\log}=dV_K.
 }
\end{lemma}

\begin{proof}
Let $e_1,\ldots,e_{n-1}$ be a local positively oriented $\bar{g}$-orthonormal frame. Then, using $dY=P^{-1}dX$ and $P\bm{\xi}=X$, we have
\eq{
\frac{dV_{\log}}{dx}=\det(dY(e_1),\dots,dY(e_{n-1}),\bm{\xi})=\det P^{-1}\cdot\det(dX(e_1),\dots,dX(e_{n-1}),X).
}
The last determinant is the density of the cone-volume measure of $K$.
\end{proof}

\begin{lemma}\label{lem:log-hessian-flat}
Let $\Psi\in C^2(\R^n)$ and set $\psi=\Psi\circ Y$. Then
\eq{
\label{eq:log-hessian}
(\nabla^{\log})^2\psi(v,w)=D^2\Psi(Y)[dY(v),dY(w)]-D\Psi(Y)[\bm{\xi}]\,g_{\log}(v,w).
}
In particular, we have
\eq{
\label{eq:log-hessian-eta}
(\nabla^{\log})^2\eta+n g_{\log}=0.
}
Moreover, $r(y):=\frac{1}{2}\log\left(\sum_{i=1}^n e^{2y_i}\right)$ satisfies
\eq{
\label{eq:log-hessian-r}
(\nabla^{\log})^2(r\circ Y)=D^2r(Y)[dY,dY]-g_{\log}.
}
\end{lemma}
\begin{proof}
By definition,
\eq{
(\nabla^{\log})^2\psi(v,w)=v(w\psi)-(\nabla^{\log}_vw)\psi.
}
Since $\psi=\Psi\circ Y$, we have
\eq{
v(w\psi)=D^2\Psi(Y)[dY(v),dY(w)]+D\Psi(Y)[D_vdY(w)].
}
Moreover, we have $(\nabla^{\log}_vw)\psi=D\Psi(Y)[dY(\nabla^{\log}_vw)]$. Hence subtracting these two equations and using \eqref{eq:log-gauss-formula} yields \eqref{eq:log-hessian}.

The identity \eqref{eq:log-hessian-eta} follows from \eqref{eq:log-hessian} with $\Psi(y)=\sum_{i=1}^ny_i$: $D^2\Psi=0$ and $D\Psi(Y)[\bm{\xi}]=n$. 

Finally, since $Dr(Y)[\bm{\xi}]=1$, applying \eqref{eq:log-hessian} to $\Psi=r$ yields \eqref{eq:log-hessian-r}.
\end{proof}

For a weight $\Phi\in C^\infty(\Omega_+)$, we set
\eq{
d\mu_\Phi=e^{-\Phi}dV_{\log},\quad \mathcal{L}_{\Phi}^{\log}F=\Delta_{\log}F-g_{\log}(\nabla^{\log}\Phi,\nabla^{\log}F).
}
Moreover, for a vector field $Z$, we define
\eq{
\operatorname{div}^{\log}_{\Phi}Z=\operatorname{div}_{\nabla^{\log}}Z-g_{\log}(\nabla^{\log}\Phi,Z).
}

\begin{lemma}\label{lem:log-weighted-integration-by-parts}
For every $u\in C_c^\infty(\Omega_+)$ and every smooth vector field $Z$ there holds
\eq{
\int_{\Omega_+}u\,\operatorname{div}^{\log}_{\Phi}Z\,d\mu_\Phi=-\int_{\Omega_+}g_{\log}(\nabla^{\log}u,Z)\,d\mu_\Phi.
}
In particular, for smooth compactly supported functions $u,v$, we have
\eq{
\int_{\Omega_+}u\,\mathcal{L}_{\Phi}^{\log}v\,d\mu_\Phi=-\int_{\Omega_+}g_{\log}(\nabla^{\log}u,\nabla^{\log}v)\,d\mu_\Phi.
}
\end{lemma}

\begin{proof}
By \cite[Chap. 2, Prop. 1.4]{NS94}, $\nabla^{\log}dV_{\log}=0$. Since $\nabla^{\log}$ is torsion-free, by the divergence theorem, whenever $W$ has compact support in $\Omega_+$, we have
\eq{
\int_{\Omega_+}\operatorname{div}_{\nabla^{\log}}W\,dV_{\log}=0.
}
Therefore, since $\operatorname{div}_{\nabla^{\log}}(e^{-\Phi}Z)=e^{-\Phi}\operatorname{div}_{\Phi}^{\log}Z$, we have $\int_{\Omega_+}\operatorname{div}_{\Phi}^{\log}Z\,d\mu_\Phi=0$ for every compactly supported vector field $Z$. Applying this to $uZ$, and using
\eq{
\operatorname{div}_{\Phi}^{\log}(uZ)=u\operatorname{div}_{\Phi}^{\log}Z+g_{\log}(\nabla^{\log}u,Z),
}
we obtain the first identity. Taking $Z=\nabla^{\log}v$, the second identity follows.
\end{proof}

\begin{lemma}\label{lem:weighted-log-bochner}
Let $\Phi\in C^\infty(\Omega_+)$. Then, for every $F\in C_c^\infty(\Omega_+)$,
\eq{
\int_{\Omega_+}\left(\mathcal{L}_{\Phi}^{\log}F\right)^2\,d\mu_\Phi=\int_{\Omega_+}\left|(\nabla^{\log,*})^2F\right|_{g_{\log}}^2\,d\mu_\Phi+\int_{\Omega_+}(\nabla^{\log})^2\Phi(\nabla^{\log}F,\nabla^{\log}F)\,d\mu_\Phi.
}
\end{lemma}

\begin{proof}
For readability, write $\nabla=\nabla^{\log}$, $\nabla^*=\nabla^{\log,*}$, $g=g_{\log}$, $dV=dV_{\log}$, $\mathcal{L}=\mathcal{L}_{\Phi}^{\log}$, and $Z=\nabla F$. We first prove the asymmetric affine Bochner identity
\eq{
\operatorname{div}_{\Phi}^{\log}(\nabla_ZZ)=Z(\mathcal{L}F)+\nabla^2\Phi(Z,Z)+\left|(\nabla^*)^2F\right|_g^2.
}

Since the logarithmic affine connection $\nabla^{\log}$ is flat, \cite[Eq. (8)]{Opo15} gives
\eq{
\operatorname{div}_{\nabla}(\nabla_ZZ)=Z(\operatorname{div}_{\nabla}Z)+\operatorname{tr}_g\left((\nabla Z)\circ(\nabla Z)\right).
}
Moreover, $\operatorname{div}_{\nabla}Z=\Delta F$. The conjugacy identity \eqref{eq:log-conjugacy} yields, as in the proof of \eqref{eq:trace-hess},
\eq{
\operatorname{tr}_g\left((\nabla Z)\circ(\nabla Z)\right)=\left|(\nabla^*)^2F\right|_g^2.
}
Consequently,
\eq{
\operatorname{div}_{\nabla}(\nabla_ZZ)=Z(\Delta F)+\left|(\nabla^*)^2F\right|_g^2.
}

Now using the definition of the weighted divergence, we obtain
\eq{
\operatorname{div}_{\Phi}^{\log}(\nabla_ZZ)=Z(\Delta F)-g(\nabla\Phi,\nabla_ZZ)+\left|(\nabla^*)^2F\right|_g^2.
}
Since $\mathcal{L}F=\Delta F-g(\nabla\Phi,Z)$, we have
\eq{
Z(\mathcal{L}F)=Z(\Delta F)-Z\left(g(\nabla\Phi,Z)\right).
}
Moreover, we have
\eq{
Z\left(g(\nabla\Phi,Z)\right)=g(\nabla_ZZ,\nabla\Phi)+g(Z,\nabla_Z^*\nabla\Phi),\quad g(Z,\nabla_Z^*\nabla\Phi)=\nabla^2\Phi(Z,Z).
}
Therefore,
\eq{
Z(\Delta F)-g(\nabla\Phi,\nabla_ZZ)=Z(\mathcal{L}F)+\nabla^2\Phi(Z,Z).
}
Substituting this into the expression for $\operatorname{div}_{\Phi}^{\log}(\nabla_ZZ)$ proves the identity.

The vector field $\nabla_ZZ$ has compact support. Hence
\eq{
\int_{\Omega_+}\operatorname{div}_{\Phi}^{\log}(\nabla_ZZ)\,d\mu_\Phi=0
}
and
\eq{
0=\int_{\Omega_+}Z(\mathcal{L}F)\,d\mu_\Phi+\int_{\Omega_+}\nabla^2\Phi(Z,Z)\,d\mu_\Phi+\int_{\Omega_+}\left|(\nabla^*)^2F\right|_g^2\,d\mu_\Phi.
}
Moreover, we have
\eq{
\int_{\Omega_+}Z(\mathcal{L}F)\,d\mu_\Phi=-\int_{\Omega_+}(\mathcal{L}F)^2\,d\mu_\Phi.
}
Therefore
\eq{
0=-\int_{\Omega_+}(\mathcal{L}F)^2\,d\mu_\Phi+\int_{\Omega_+}\nabla^2\Phi(\nabla F,\nabla F)\,d\mu_\Phi+\int_{\Omega_+}\left|(\nabla^*)^2F\right|_g^2\,d\mu_\Phi.
}
Now we return to logarithmic notation.
\end{proof}

For $\alpha\in\R$, we define $\Phi_\alpha=-\eta-\alpha r\circ Y$. Then, by \autoref{lem:log-affine-volume}, $e^{-\Phi_\alpha}dV_{\log}=|X|^\alpha dV_K$. We abbreviate $\mathcal{L}_{\alpha}^{\log}:=\mathcal{L}_{\Phi_\alpha}^{\log}$. In what follows, $\Phi_0=-\eta$, so that $dV_K=e^{-\Phi_0}dV_{\log}$.

\begin{lemma}
\label{lem:div-log-weighted-equals-centro-affine}
For every smooth vector field $Z$ on $\Omega_+$, one has
\eq{
\operatorname{div}_{\Phi_0}^{\log}Z=\operatorname{div}_{\nabla}Z.
}
\end{lemma}

\begin{proof}
We work in local coordinates $x^1,\dots,x^{n-1}$. We write
\eq{
dV_{\log}=J_{\log}\, dx^1\cdots dx^{n-1},\quad dV_K=J_K\, dx^1\cdots dx^{n-1}.
}
Since $dV_K=e^{-\Phi_0}dV_{\log}$, we have $J_K=e^{-\Phi_0}J_{\log}$. Moreover, $\nabla^{\log}dV_{\log}=0=\nabla dV_K$ and the connections are torsion-free, therefore
\eq{
(\Gamma^{\log})^i_{ik}=\partial_k\log J_{\log},\quad \Gamma^i_{ik}=\partial_k\log J_K,
}
where $\nabla^{\log}_{\partial_j}\partial_k=\left(\Gamma^{\log}\right)^i_{jk}\partial_i$ and $\nabla_{\partial_j}\partial_k=\Gamma^i_{jk}\partial_i$.

Now let $Z=Z^i\partial_i$. By definition,
\eq{
\operatorname{div}_{\Phi_0}^{\log}Z&=\partial_iZ^i+(\Gamma^{\log})^i_{ik}Z^k-Z^k\partial_k\Phi_0\\
&=\partial_iZ^i+Z^k\partial_k\log J_{\log}-Z^k\partial_k\Phi_0\\
&=\partial_iZ^i+Z^k\partial_k\log J_K\\
&=\partial_iZ^i+\Gamma^i_{ik}Z^k=\operatorname{div}_{\nabla}Z.
}
\end{proof}

\begin{lemma}
Let $\Omega\Subset\Omega_+$ be a smooth domain. Let $\mathbf{N}_{\Omega}$ be the outward $g_{\log}$-unit normal to $\partial\Omega$ and $\mathbf{n}_{\Omega}$ be the outward $\bar{g}$-unit normal to $\partial\Omega$. Then
\eq{
\int_{\Omega}\operatorname{div}_{\Phi_0}^{\log}Z\, dV_K
=\int_{\partial\Omega}
g_{\log}(Z,\mathbf{N}_{\Omega})\,d\sigma_{\log},
}
where $d\sigma_{\log}=h\det(\tau)\,\bar{g}(\mathbf{N}_{\Omega},\mathbf{n}_{\Omega})\,d\bar\sigma$.
\end{lemma}

\begin{proof}
By \autoref{lem:div-log-weighted-equals-centro-affine} and \autoref{lem:divnabla-divbarnabla}, we have
\eq{
\operatorname{div}_{\nabla}Z\, dV_K
=\bar{\nabla}_i\left(h\det\left(\tau\right)Z^i\right)\, dx.
}
Applying the standard divergence theorem on $\left(\Sn,\bar{g}\right)$,
\eq{
\int_{\Omega}\operatorname{div}^{\log}_{\Phi_{0}}Z\, dV_K
=\int_{\partial\Omega}
h\det\left(\tau\right)\bar{g}\left(Z,\mathbf{n}_{\Omega}\right)\,d\bar{\sigma}.
}
Now the claim follows from the identity
\eq{
\bar{g}\left(Z,\mathbf{n}_{\Omega}\right)
=g_{\log}\left(Z,\mathbf{N}_{\Omega}\right)
\bar{g}\left(\mathbf{N}_{\Omega},\mathbf{n}_{\Omega}\right).
}
\end{proof}

We now derive the corresponding Bochner formula on smooth subdomains of $\Omega_+$. 

\begin{theorem}\label{thm:log-bochner-domain}
Let $\Omega\Subset\Omega_+$ be a smooth domain. For $u\in C^3(\overline\Omega)$, set $Z=\nabla^{\log}u$. Then
\eq{
&\int_{\Omega}\left((\mathcal{L}_0^{\log}u)^2-\left|(\nabla^{\log,*})^2u\right|_{g_{\log}}^2-n|\nabla^{\log}u|_{g_{\log}}^2\right)\, dV_K\\
&=\int_{\partial\Omega}\left((\mathcal{L}_0^{\log}u)u_{\mathbf{N}_{\Omega}}-g_{\log}(\nabla_Z^{\log}Z,\mathbf{N}_{\Omega})\right)\,d\sigma_{\log}.
}
Here $u_{\mathbf{N}_{\Omega}}=g_{\log}(\nabla^{\log}u,\mathbf{N}_{\Omega})$. In particular, if $u_{\mathbf{N}_{\Omega}}=0$ on $\partial\Omega$, then
\eq{
\int_{\Omega}\left((\mathcal{L}_0^{\log}u)^2-\left|(\nabla^{\log,*})^2u\right|_{g_{\log}}^2-n|\nabla^{\log}u|_{g_{\log}}^2\right)\, dV_K=\int_{\partial\Omega}\mathrm{II}_{\log}^{*}(\nabla^{\log}u,\nabla^{\log}u)\,d\sigma_{\log},
}
where $\mathrm{II}_{\log}^{*}(U,V):=g_{\log}(U,\nabla_V^{\log,*}\mathbf{N}_{\Omega})$ for tangent vector fields $U,V$ along $\partial\Omega$.
\end{theorem}

\begin{proof}
The pointwise weighted Bochner identity from the proof of \autoref{lem:weighted-log-bochner}, applied with $\Phi=\Phi_0$, yields
\eq{
\operatorname{div}_{\Phi_0}^{\log}(\nabla_Z^{\log}Z)=Z(\mathcal{L}_0^{\log}u)+n|\nabla^{\log}u|_{g_{\log}}^2+\left|(\nabla^{\log,*})^2u\right|_{g_{\log}}^2.
}
Here we used $(\nabla^{\log})^2\Phi_0=ng_{\log}$. Integrating over $\Omega$, we obtain
\eq{
\int_{\partial\Omega}g_{\log}(\nabla_Z^{\log}Z,\mathbf{N}_{\Omega})\,d\sigma_{\log}=\int_{\Omega}Z(\mathcal{L}_0^{\log}u)\, dV_K+n\int_{\Omega}|\nabla^{\log}u|_{g_{\log}}^2\, dV_K+\int_{\Omega}\left|(\nabla^{\log,*})^2u\right|_{g_{\log}}^2\, dV_K.
}
Moreover, we have
\eq{
\int_{\Omega}Z(\mathcal{L}_0^{\log}u)\, dV_K=-\int_{\Omega}(\mathcal{L}_0^{\log}u)^2\, dV_K+\int_{\partial\Omega}(\mathcal{L}_0^{\log}u)u_{\mathbf{N}_{\Omega}}\,d\sigma_{\log}.
}
Hence the first formula follows.

If $u_{\mathbf{N}_{\Omega}}=0$ on $\partial\Omega$, then $Z=\nabla^{\log}u$ is tangent to $\partial\Omega$. Since $g_{\log}(Z,\mathbf{N}_{\Omega})=0$ along $\partial\Omega$, differentiating tangentially in the direction $Z$ and using \eqref{eq:log-conjugacy}, we obtain
\eq{
0=Z(g_{\log}(Z,\mathbf{N}_{\Omega}))=g_{\log}(\nabla_Z^{\log}Z,\mathbf{N}_{\Omega})+g_{\log}(Z,\nabla_Z^{\log,*}\mathbf{N}_{\Omega}).
}
Thus
\eq{
-g_{\log}(\nabla_Z^{\log}Z,\mathbf{N}_{\Omega})=g_{\log}(Z,\nabla_Z^{\log,*}\mathbf{N}_{\Omega})=\mathrm{II}_{\log}^{*}(Z,Z),
}
and the second identity follows.
\end{proof}

\begin{lemma}\label{lem:log-boundary-2nd-fund-form}
For $\varepsilon>0$, suppose $c$ is a regular value of
\eq{
q_{\varepsilon}=-\varepsilon\log\left(\sum_{i=1}^n e^{-Y_i/\varepsilon}\right).
}
Set $\Omega_{\varepsilon}=\{q_{\varepsilon}>c\}\Subset\Omega_+$ and let
\eq{
\mathbf{N}_{\varepsilon}=-\frac{\nabla^{\log}q_{\varepsilon}}{\abs{\nabla^{\log}q_{\varepsilon}}_{g_{\log}}}
}
be the outward $g_{\log}$-unit normal to $\partial\Omega_{\varepsilon}$.
 Then, for every tangent vector $Z\in T\partial\Omega_{\varepsilon}$,
\eq{
\mathrm{II}_{\log,\varepsilon}^{*}(Z,Z)=-\frac{(\nabla^{\log})^2q_{\varepsilon}(Z,Z)}{|\nabla^{\log}q_{\varepsilon}|_{g_{\log}}}\ge 0.
}
\end{lemma}

\begin{proof}
Since $(\nabla^{\log})^2Y_i=-g_{\log}$ (see \eqref{eq:log-hessian}), direct differentiation gives
\eq{\label{eq:approximation-log-centro-affine}
dq_{\varepsilon}&=\sum_{i=1}^n\theta_i dY_i,\\
(\nabla^{\log})^2q_{\varepsilon}&=-g_{\log}-\frac1\varepsilon\left(\sum_{i=1}^n\theta_i dY_i\otimes dY_i-\sum_{i,j=1}^n\theta_i\theta_j dY_i\otimes dY_j\right),
}
where $\theta_i=\frac{e^{-Y_i/\varepsilon}}{\sum_{j=1}^n e^{-Y_j/\varepsilon}}$. The tensor in parentheses is non-negative definite, and therefore
\eq{
(\nabla^{\log})^2q_{\varepsilon}\le -g_{\log}\le 0.
}

Note that $dq_{\varepsilon}(Z)=0$, that is, $g_{\log}(Z,\nabla^{\log}q_{\varepsilon})=0$. Now using \eqref{eq:log-conjugacy},
\eq{
\mathrm{II}_{\log,\varepsilon}^{*}(Z,Z)&=g_{\log}(Z,\nabla_Z^{\log,*}\mathbf{N}_{\varepsilon})=-\frac{g_{\log}(Z,\nabla_Z^{\log,*}\nabla^{\log}q_{\varepsilon})}{|\nabla^{\log}q_{\varepsilon}|_{g_{\log}}}=-\frac{(\nabla^{\log})^2q_{\varepsilon}(Z,Z)}{|\nabla^{\log}q_{\varepsilon}|_{g_{\log}}}\ge 0.
}
\end{proof}

The proof of \autoref{thm:log-proof-unconditional-poincare} extends verbatim to the weighted measures $|X|^{q-n}dV_K$ for all $q\in(0,n)$. For simplicity, we present the details only in the case $q=n$. We mention that the inequality \eqref{ineq:unconditional-centro-affine-n}, without the equality characterization, was proved by Kolesnikov--Milman in \cite{KM22}. 

\begin{theorem}\label{thm:log-proof-unconditional-poincare}
Let $K\subset\R^n$ be a $C^2_+$ unconditional convex body, and let $F\in C^1(\Sn)$ be unconditional. Then
\eq{
n\int_{\Omega_+}(F-\bar{F}_+)^2\, dV_K\le \int_{\Omega_+}|\nabla^{\log}F|_{g_{\log}}^2\, dV_K\le \int_{\Omega_+}|\nabla F|_{g}^2\, dV_K,\quad \bar{F}_+=\frac{\int_{\Omega_+}F\, dV_K}{\int_{\Omega_+}dV_K}.
}
In particular
\eq{\label{ineq:unconditional-centro-affine-n}
n\int_{\Sn}(F-\bar{F})^2\, dV_K\le \int_{\Sn}|\nabla F|_g^2\, dV_K,\quad \bar{F}=\frac{\int_{\Sn}F\, dV_K}{\int_{\Sn}dV_K}.
}
Moreover, equality in \eqref{ineq:unconditional-centro-affine-n} holds only for constant functions.
\end{theorem}

\begin{proof}
By approximation, we may assume that $F\in C^\infty(\Sn)$ and $K$ is $C^{\infty}_+$. On $\Omega_+$,
\eq{
q(x)-\varepsilon\log n\le q_{\varepsilon}(x)\le q(x),\quad \text{where } q=\min_iY_i.
}

Let $\varepsilon_j\to 0$ and fix $x_0\in\Omega_+$. Since $q_{\varepsilon_j}(x_0)\to q(x_0)>-\infty$, there exists $j_0$ such that $q_{\varepsilon_j}(x_0)>-j$ for every $j\ge j_0$. Hence, for every $j\ge j_0$, by Sard's theorem we can choose a regular value $c_j\in(-j-1,-j)$ of $q_{\varepsilon_j}$. Now set $\Omega_j=\{x\in\Omega_+:\ q_{\varepsilon_j}(x)>c_j\}$, which is a non-empty smooth domain. Since $q_{\varepsilon_j}\le q$, the condition $q_{\varepsilon_j}>c_j$ implies $Y_i>c_j$ for every $i$, equivalently $X_i>e^{c_j}$ for every $i$; hence $\Omega_j\Subset\Omega_+$. Moreover, $\Omega_j$ is connected; the proof is similar to that of \autoref{claim:Omega-eps-connected}. Finally, it is easy to see that $\mathbf{1}_{\Omega_j}\to\mathbf{1}_{\Omega_+}$ pointwise.

Let $\mathbf{N}_j$ be the outward $g_{\log}$-unit normal to $\partial\Omega_j$. For each $j$, we define
\eq{
\bar{F}_j=\frac{\int_{\Omega_j}F\, dV_K}{\int_{\Omega_j}dV_K},\quad f_j=F-\bar{F}_j.
}
Since $\mathcal{L}_0^{\log}$ is uniformly elliptic on $\Omega_j$, and since the compatibility condition holds, there exists a smooth solution $u_j$, unique up to an additive constant, of
\eq{
\mathcal{L}_0^{\log}u_j=-f_j\quad\text{in }\Omega_j,\quad (u_j)_{\mathbf{N}_j}=0 \quad\text{on }\partial\Omega_j.
}
See the proof of \autoref{lem:neumann-eps} for more details.

Due to the Bochner formula of \autoref{thm:log-bochner-domain}, together with \autoref{lem:log-boundary-2nd-fund-form},
\eq{
\int_{\Omega_j}f_j^2\, dV_K=\int_{\Omega_j}(\mathcal{L}_0^{\log}u_j)^2\, dV_K\ge n\int_{\Omega_j}|\nabla^{\log}u_j|_{g_{\log}}^2\, dV_K.
}
On the other hand, integration-by-parts, the Neumann condition, and the Cauchy--Schwarz inequality imply
\eq{
n\int_{\Omega_j}(F-\bar{F}_j)^2\, dV_K \le \int_{\Omega_j}|\nabla^{\log}F|_{g_{\log}}^2\, dV_K;
}
see the proof of \autoref{thm:cap-poincare-from-epsilon} for more details.

By \autoref{lem:log-affine-metric}, $g_{\log}^{-1}\le g^{-1}$. Therefore
\eq{
n\int_{\Omega_j}(F-\bar{F}_j)^2\, dV_K\le\int_{\Omega_j}|\nabla^{\log}F|_{g_{\log}}^2\, dV_K
\le \int_{\Omega_j}|\nabla F|_g^2\, dV_K\le \int_{\Omega_+}|\nabla F|_g^2\, dV_K.
}
Letting $j\to \infty$ and using the dominated convergence theorem, we conclude that
\eq{
n\int_{\Omega_+}(F-\bar{F}_+)^2\, dV_K\le\int_{\Omega_+}|\nabla^{\log}F|_{g_{\log}}^2\, dV_K\le  \int_{\Omega_+}|\nabla F|_{g}^2\, dV_K.
}
Since $K$ and $F$ are unconditional, summing over all orthants implies the non-strict version of the inequality on $\Sn$. 

Now suppose $F\in C^1(\Sn)$ and $K$ is $C^2_+$. On $\Omega_+$, due to \autoref{lem:log-affine-metric}, we have $g_{\log}>g$, and hence
\eq{
|\nabla^{\log}F|_{g_{\log}}^2<|\nabla F|_g^2
}
at every point where $dF\neq 0$. Therefore, if $F$ is non-constant on $\Omega_+$, then
\eq{
\int_{\Omega_+}|\nabla^{\log}F|_{g_{\log}}^2\, dV_K< \int_{\Omega_+}|\nabla F|_{g}^2\, dV_K.
}
Thus equality in the theorem can occur only if $F$ is constant on $\Omega_+$. By unconditionality and continuity, this implies that $F$ is constant on $\Sn$.
\end{proof}

\begin{theorem}\label{thm:log-poincare-equality-case}
Let $K\subset\R^n$ be a $C^{\infty}_+$ unconditional convex body.
The equality in
\eq{
n\int_{\Omega_+}(F-\bar{F}_+)^2\, dV_K\le\int_{\Omega_+}|\nabla^{\log}F|_{g_{\log}}^2\, dV_K
}
is attained by the functions
\eq{
F_i=\frac{x_iX_i}{h}-\frac{1}{n},\quad i=1,\ldots,n.
}
\end{theorem}

\begin{proof}
Let $\bm{\xi}^{\ast}$ be the conormal field defined by
\eq{
\bm{\xi}^{\ast}=\frac{Px}{h}=\left(\frac{x_1X_1}{h},\ldots,\frac{x_nX_n}{h}\right).
}
Then
\eq{
\langle \bm{\xi}^{\ast},dY(V)\rangle=0 \quad \text{for all } V\in T\Omega_+,\quad \langle \bm{\xi}^{\ast},\bm{\xi}\rangle=1.
}

We first prove
\eq{
\mathcal{L}^{\log}_0 \bm{\xi}^{\ast}+n\bm{\xi}^{\ast}=\bm{\xi}.
}
 
Differentiating the relations $\langle \bm{\xi}^{\ast},\bm{\xi}\rangle=1$ and $\langle \bm{\xi}^{\ast},dY(V)\rangle=0$ in the direction $U$ and using the affine Gauss formula
\eq{
D_UdY(V)=dY(\nabla^{\log}_UV)-g_{\log}(U,V)\bm{\xi}
}
yields
\eq{
\langle d\bm{\xi}^{\ast}(U),\bm{\xi}\rangle=0, \quad \langle d\bm{\xi}^{\ast}(U),dY(V)\rangle=g_{\log}(U,V).
}

For $a\in\R^n$, set $f_a=\langle \bm{\xi}^{\ast},a\rangle$.
Note that the vector $a$ decomposes uniquely into a tangential part and a transversal part:
\eq{
a=dY(W)+f_a\bm{\xi}.
}
For every $V\in T\Omega_+$, we have
\eq{
df_a(V)=\langle d\bm{\xi}^{\ast}(V),a\rangle.
}
Hence, using $\langle d\bm{\xi}^{\ast}(V),\bm{\xi}\rangle=0$ and $\langle d\bm{\xi}^{\ast}(V),dY(W)\rangle=g_{\log}(V,W)$,
\eq{
df_a(V)=\langle d\bm{\xi}^{\ast}(V),dY(W)\rangle=g_{\log}(V,W).
}
Thus $W=\nabla^{\log}f_a$ and
\eq{
a=dY(\nabla^{\log}f_a)+f_a\bm{\xi}.
}

Now we differentiate this identity in the direction $U$:
\eq{
0=D_UdY(\nabla^{\log}f_a)+df_a(U)\bm{\xi}.
}
Using the affine Gauss formula,
\eq{
D_UdY(Z)=dY(\nabla^{\log}_UZ)-g_{\log}(U,Z)\bm{\xi},
}
we obtain
\eq{
dY(\nabla^{\log}_U\nabla^{\log}f_a)=0 \implies \nabla^{\log}_U\nabla^{\log}f_a=0.
}
Moreover, by the conjugacy relation between $\nabla^{\log}$ and $\nabla^{\log,*}$,
\eq{
g_{\log}(\nabla^{\log}_U\nabla^{\log}f_a,V)=(\nabla^{\log,*})^2f_a(U,V).
}
Therefore
\eq{
(\nabla^{\log,*})^2f_a=0\implies \Delta_{\log}f_a=0 
\implies \mathcal{L}^{\log}_0 f_a=g_{\log}(\nabla^{\log}\eta,\nabla^{\log}f_a).
}

Taking the Euclidean inner product of $\bm{\xi}$ with $a$ yields
\eq{
\langle dY(\nabla^{\log}f_a),\bm{\xi}\rangle=\langle a,\bm{\xi}\rangle-nf_a.
}
On the other hand, we have
\eq{
\langle dY(\nabla^{\log}f_a),\bm{\xi}\rangle=d\eta(\nabla^{\log}f_a)=g_{\log}(\nabla^{\log}\eta,\nabla^{\log}f_a).
}
Hence
\eq{
 \mathcal{L}^{\log}_0 f_a+nf_a=\langle a,\bm{\xi}\rangle.
}

Next, we justify the integrations by parts on $\Omega_+$. We recall from the proof of \autoref{thm:log-proof-unconditional-poincare} that $q=\min_kY_k$ and $q_{\varepsilon}=-\varepsilon\log\left(\sum_{k=1}^ne^{-Y_k/\varepsilon}\right)$. We choose a sequence $\varepsilon_j\to 0$, and for simplicity we set $q_j=q_{\varepsilon_j}$. Let $\chi_j=\rho_j\circ q_j$, where $\rho_j\in C^\infty(\mathbb R)$ satisfies $0\le\rho_j\le 1$, $\rho_j=0$ on $(-\infty,-j-1]$, $\rho_j=1$ on $[-j,\infty)$, and $|\rho_j'|\le 2$. Then $\chi_j\in C_c^\infty(\Omega_+)$ and $\chi_j\to 1$ pointwise on $\Omega_+$.

Recall $p_i=x_iX_i/h$. From $E_i=dY(\nabla^{\log}p_i)+p_i\bm{\xi}$, we obtain $dY_k(\nabla^{\log}p_i)=\delta_{ik}-p_i$. Moreover, if we set $\theta_{j,k}=e^{-Y_k/\varepsilon_j}/\sum_{\ell=1}^ne^{-Y_\ell/\varepsilon_j}$, then $dq_j=\sum_k\theta_{j,k}dY_k$. Therefore
\eq{
dq_j(\nabla^{\log}p_i)
=\sum_{k=1}^n\theta_{j,k}dY_k(\nabla^{\log}p_i)
=\sum_{k=1}^n\theta_{j,k}(\delta_{ik}-p_i)=\theta_{j,i}-p_i,
}
where we used $\sum_k\theta_{j,k}=1$. Since $ p_i, \theta_{j,i}\in (0,1)$ on $\Omega_+$, we have 
\eq{
|dq_j(\nabla^{\log}p_i)|\le 1.
}

Due to $\mathcal{L}_0^{\log}F_i=-nF_i$, $\nabla^{\log}F_i=\nabla^{\log}p_i$, and \autoref{lem:log-weighted-integration-by-parts},
\eq{
n\int_{\Omega_+}\chi_jF_i^2\, dV_K
=\int_{\Omega_+}\chi_j|\nabla^{\log}F_i|_{g_{\log}}^2\, dV_K+\int_{\Omega_+}F_i\rho_j'(q_j)dq_j(\nabla^{\log}p_i)\, dV_K.
}
Note that $\rho_j'(q_j)$ is supported in $A_j=\{-j-1<q_j<-j\}$, and $\mathbf{1}_{A_j}\to 0$ pointwise. Hence
\eq{
\left|\int_{\Omega_+}F_i\rho_j'(q_j)dq_j(\nabla^{\log}p_i)\, dV_K\right|\le 2\|F_i\|_{L^\infty(\Omega_+)}dV_K(A_j)\to 0.
}
Since $0\le \chi_j\le 1$ and $\chi_j\to 1$ pointwise, Fatou's lemma gives
\eq{
\int_{\Omega_+}|\nabla^{\log}F_i|_{g_{\log}}^2\, dV_K\le n\int_{\Omega_+}F_i^2\, dV_K.
}

Finally, by \cite[Lem. 4.4]{HI25}, we have
\eq{
\frac{1}{V(K)}\int_{\Sn}\frac{x}{h}\otimes X\, dV_K=I \implies \int_{\Omega_+}F_i\, dV_K=0.
}
Consequently, due to \autoref{thm:log-proof-unconditional-poincare},
\eq{
n\int_{\Omega_+}F_i^2\, dV_K=\int_{\Omega_+}|\nabla^{\log}F_i|_{g_{\log}}^2\, dV_K.
}
\end{proof}

The next theorem should be compared with \autoref{thm:intersection-caps-poincare-ineq}, which gives only the constant $(n+\sum\alpha_i)-1$ in the unconditional case; see also the discussion thereafter.

\begin{theorem}
\label{thm:power-product-unconditional-improved} Let $\alpha_i\ge 0$. Let $K$ be a $C^2_+$ unconditional convex body and $F\in C^1(\Sn)$ be unconditional. Then
\eq{
(n+\sum\alpha_i)\int_{\Sn}(F-\bar{F})^2 |X_1|^{\alpha_1}\cdots |X_n|^{\alpha_n}\, dV_K \le 
\int_{\Sn}|\nabla F|_g^2 |X_1|^{\alpha_1}\cdots |X_n|^{\alpha_n}\, dV_K,
}
where
\eq{
\bar{F}=\frac{\int_{\Sn}F|X_1|^{\alpha_1}\cdots |X_n|^{\alpha_n}\, dV_K}
{\int_{\Sn}|X_1|^{\alpha_1}\cdots |X_n|^{\alpha_n}\, dV_K}.
}
Moreover, equality holds only for constant functions.
\end{theorem}

\begin{proof}
On $\Omega_+$, consider the weight $\Phi=-\eta-\sum \alpha_i Y_i$. By \autoref{lem:log-affine-volume},
\eq{
d\mu_\Phi=e^{-\Phi}\,dV_{\log}=e^{\eta+\sum \alpha_iY_i}\, dV_{\log}=|X_1|^{\alpha_1}\cdots |X_n|^{\alpha_n}\, dV_K.
}
By \autoref{lem:log-hessian-flat}, $(\nabla^{\log})^2\Phi=(n+\sum\alpha_i) g_{\log}$. Now one repeats the proof of \autoref{thm:log-proof-unconditional-poincare}, with $dV_K$ and $\mathcal{L}_0^{\log}$ replaced by $d\mu_\Phi$ and $\mathcal{L}_\Phi^{\log}$.
\end{proof}

\begin{remark}
For $p\in (0,1]$, the $L_p$ version of centro-affine geometry for $C^{\infty}_+$ unconditional convex bodies is obtained by replacing the logarithmic map by the power map
\eq{
T_p:(0,\infty)^n\to (0,\infty)^n, \quad T_p(z)=\frac{1}{p}\left(z_1^p,\dots,z_n^p\right).
}
On $\Omega_+$, we define
\eq{
Y_p=T_p\circ X,\quad (Y_p)_i=\frac{X_i^p}{p},\quad \bm{\xi}_p=pY_p.
}
The induced $L_p$-centro-affine structure is given by
\eq{
D_vdY_p(w)&=dY_p(\nabla^{(p)}_vw)-g_p(v,w)\bm{\xi}_p, \quad dV_p/dx=\det(dY_p(e_1),\dots,dY_p(e_{n-1}),\bm{\xi}_p).
}

Since $D_v\bm{\xi}_p=p\, dY_p(v)$,
\eq{
D_v\bm{\xi}_p=dY_p(\mathcal{S}_pv)-\theta_p(v)\bm{\xi}_p \quad \text{with} \quad \mathcal{S}_p=pI, \quad \theta_p=0.
}
Thus the structure is equiaffine, and by the Gauss equation:
\eq{
\operatorname{Rm}^{\nabla^{(p)}}(u,v)w
=p\left(g_p(v,w)u-g_p(u,w)v\right) \implies \operatorname{Ric}^{\nabla^{(p)}}=p(n-2)g_p.
}

The case $p=1$ corresponds to the classical centro-affine geometry, which is globally defined on the whole sphere. For $0\le p<1$, however, the corresponding $L_p$-type geometry is confined to the positive orthant $\Sn\cap(0,\infty)^n$. Nevertheless, for $p=0$, this restriction is compensated by the fact that the resulting geometry is flat.

By deriving the corresponding Bochner formulas and carrying out the same approximation scheme as in the case $p=0$, together with an appropriate Bakry--\'Emery argument, one should obtain another proof of \autoref{thm:power-coordinate-PL}. We leave the details to the interested reader.

It is tempting to view the logarithmic centro-affine geometry as the shadow of a geometry for origin-symmetric $C^{\infty}_+$ convex bodies, one that would ideally produce the additional $+1$ contribution without needing the restriction to $(0,\infty)^n$. At present, however, such a geometry remains elusive, and with it the source of the additional $+1$ contribution. 

We end this remark with a dictionary relating the objects in logarithmic centro-affine geometry in the unconditional class to those in centro-affine geometry. Let us define the $g$-self-adjoint endomorphism $\mathcal{A}:T\Omega_+\to T\Omega_+$ by
\eq{
\mathcal{A}U&=
U+\sum_{i=1}^n\frac{x_iX_i}{h}d\log X_i(U)\nabla \log X_i,
}
where $\nabla \log X_i$ denotes the $g$-gradient of $\log X_i$. Recall $\eta=\sum_{i=1}^n\log X_i$. We have:

\[
\renewcommand{\arraystretch}{1.1}
\begin{array}{|c|c|}
\hline
\text{Logarithmic centro-affine} & \text{Centro-affine}\\
\hline
Y=\log X,\, \text{transversal }(1,\ldots,1) &
X,\, \text{transversal }X \\
\hline
g_{\log}(U,V) &
g(\mathcal{A}U,V) \\
\hline
\nabla^{\log}u &
\mathcal{A}^{-1}\nabla u \\
\hline
e^{\eta}dV_{\log} &
dV_K \\
\hline
\operatorname{div}^{\log}_{-\eta}W &
\operatorname{div}_{\nabla}W \\
\hline
\Rm^{\nabla^{\log}}=0 &
\Ric^{\nabla}=(n-2)g \\
\hline
(\nabla^{\log})^2\eta+ng_{\log}=0 &
\nabla^2X_i+X_i g=0 \\
\hline
\mathcal{L}^{\log}_{-\eta}\left(\frac{x_iX_i}{h}\right)
=-n\left(\frac{x_iX_i}{h}-\frac{1}{n}\right) &
\Delta\frac{x_i}{h}=-(n-1)\frac{x_i}{h} \\
\hline
\end{array}
\]
\end{remark}

\section*{Acknowledgment}
Hu was supported by the National Key Research and Development Program of China (Grant No. 2021YFA1001800). Ivaki was supported by the Austrian Science Fund (FWF) under Project P36545.

\vspace{5mm}

\noindent
\begin{minipage}[t]{0.40\textwidth}
\textsc{School of Mathematical\\ Sciences, Beihang University,\\ Beijing 100191, China}\\
\email{\href{mailto:huyingxiang@buaa.edu.cn}{huyingxiang@buaa.edu.cn}}
\end{minipage}
\hfill
\begin{minipage}[t]{0.40\textwidth}
\textsc{Institut f\"{u}r Diskrete\\ Mathematik und Geometrie,\\ Technische Universit\"{a}t Wien, Wiedner Hauptstra{\ss}e 8--10,\\ 1040 Wien, Austria}\\
\email{\href{mailto:mohammad.ivaki@tuwien.ac.at}{mohammad.ivaki@tuwien.ac.at}}
\end{minipage}


\begin{thebibliography}{99}

\bibitem[And99]{And99} B. Andrews, 
\textit{Gauss curvature flow: the fate of the rolling stones}, Invent. Math. 138 (1999): 151--161.

\bibitem[BL95]{BL95} B. Bollob\'as, I. Leader, 
\textit{Products of unconditional bodies}, Geometric aspects of functional analysis (Israel, 1992--1994), Oper. Theory Adv. Appl., 77, Birkh\"auser, Basel, (1995), 13--24.

\bibitem[BLYZ12]{BLYZ12} K. J. B\"or\"oczky, E. Lutwak, D. Yang, G. Zhang, 
\textit{The log-Brunn-Minkowski inequality}, Adv. Math. 231 (2012): 1974--1997.

\bibitem[BCD17]{BCD17} S. Brendle, K. Choi, P. Daskalopoulos, 
\textit{Asymptotic behavior of flows by powers of the Gaussian curvature}, Acta Math. 219 (2017): 1--16.

\bibitem[CSX24]{CSX24} S. Chen, Q.-R. Li, L. Xu, 
\textit{Symmetry of solutions to a class of geometric equations for hypersurfaces in $\R^{n+1}$}, J. Differential Equations 401 (2024): 671--682.

\bibitem[CHG17]{CHG17} A. Colesanti, D. Hug, E. S. Gomez, 
\textit{Monotonicity and concavity of integral functionals involving area measures of convex bodies}, Commun. Contemp. Math. 19 (2017): 1650033.

\bibitem[CLM17]{CLM17} A. Colesanti, G. Livshyts, A. Marsiglietti, 
\textit{On the stability of Brunn-Minkowski type inequalities}, J. Funct. Anal. 273 (3) (2017): 1120--1139.

\bibitem[CER23]{CER23} D. Cordero-Erausquin, L. Rotem, 
\textit{Improved log-concavity for rotationally invariant measures of symmetric convex sets}, Ann. Probab. 51 (3) (2023): 987--1003.

\bibitem[Dub77]{Dub77} S. Dubuc, 
\textit{Crit\`eres de convexit\'e et in\'egalit\'es int\'egrales}, Ann. Inst. Fourier (Grenoble) 27 (1977): 135--165.

\bibitem[Had56]{Had56} H. Hadwiger,
\textit{Konkave Eik\"orperfunktionale und h\"ohere Tr\"agheitsmomente}, Comment. Math. Helv. 30 (1956): 285--296

\bibitem[HI24]{HI24} Y. Hu, M. N. Ivaki, 
\textit{On the uniqueness of solutions to the isotropic $L_p$ dual Minkowski problem}, Nonlinear Anal. 241 (2024): 113493.

\bibitem[HI25]{HI25} Y. Hu, M. N. Ivaki, 
\textit{Stability of the cone-volume measure with near constant density}, Int. Math. Res. Not. IMRN 2025, no. 6 (2025).

\bibitem[HLYZ16]{HLYZ16} Y. Huang, E. Lutwak, D. Yang, G. Zhang, 
\textit{Geometric measures in the dual Brunn-Minkowski theory and their associated Minkowski problems}, Acta Math. 216 (2) (2016): 325--388.

\bibitem[HYZ25]{HYZ25} Y. Huang, D. Yang, G. Zhang, 
\textit{Minkowski problems for geometric measures}, Bull. Amer. Math. Soc. 62, no. 3 (2025): 359--425.

\bibitem[ILS25]{ILS25} K. Ishige, Q. Liu, P. Salani, 
\textit{A parabolic PDE-based approach to Borell-Brascamp-Lieb inequality}, Math. Ann. 392 (4) (2025): 4891--4937.

\bibitem[Iva23]{Iva23} M. N. Ivaki, 
\textit{Uniqueness of solutions to a class of non-homogeneous curvature problems}, arXiv preprint arXiv:2307.06252 (2023).

\bibitem[IM23]{IM23} M. N. Ivaki, E. Milman, 
\textit{Uniqueness of solutions to a class of isotropic curvature problems}, Adv. Math. 435 (2023): No. 109350.

\bibitem[KL21]{KL21} A. V. Kolesnikov, G. V. Livshyts, 
\textit{On the Gardner-Zvavitch conjecture: symmetry in inequalities of Brunn-Minkowski type}, Adv. Math. 384 (2021): 107689.

\bibitem[KM16]{KM16} A. V. Kolesnikov, E. Milman, 
\textit{Riemannian metrics on convex sets with applications to Poincar\'e and log-Sobolev inequalities},  Calc. Var. Partial Differential Equations 55, no. 4 (2016): 77.

\bibitem[KM18]{KM18} A. V. Kolesnikov, E. Milman, 
\textit{Poincar\'e and Brunn-Minkowski inequalities on the boundary of weighted Riemannian manifolds}, Amer. J. Math. 140 (5) (2018): 1147--1185.

\bibitem[KM22]{KM22} A. V. Kolesnikov, E. Milman, 
\textit{Local $L^p$-Brunn-Minkowski inequalities for $p<1$}, Mem. Amer. Math. Soc. 277 (2022): No. 1360. 

\bibitem[LW24]{LW24} H. Li, Y. Wan, 
\textit{Uniqueness of solutions to some classes of anisotropic and isotropic curvature problems}, J. Funct. Anal. 287 (3) (2024): 110471.

\bibitem[LW26]{LW26} H. Li, Y. Wan, 
\textit{Classification of solutions for the planar isotropic $L_p$ dual Minkowski problem}, J. Reine Angew. Math. 834 (2026): 27--56.

\bibitem[Lie13]{Lie13} G. M. Lieberman, 
\textit{Oblique derivative problems for elliptic equations}, World Scientific, 2013.

\bibitem[LYZ18]{LYZ18} E. Lutwak, D. Yang, G. Zhang,
\textit{$L_p$ dual curvature measures}, Adv. Math. 329 (2018): 85--132.

\bibitem[Mil23]{Mil23}
E. Milman, \textit{Centro-affine differential geometry and the log-Minkowski problem}, J. Eur. Math. Soc. DOI: 10.4171/JEMS/1386 (2023).

\bibitem[NS94]{NS94} K. Nomizu, T. Sasaki, 
\textit{Affine Differential Geometry}, Cambridge Univ. Press, 1994.

\bibitem[Opo15]{Opo15} B. Opozda, 
\textit{Bochner's technique for statistical structures}, Ann. Global Anal. Geom. 48, no. 4 (2015): 357--395.

\bibitem[SZ25]{SZ25} S. Sadovsky, G. Zhang,
\textit{Brunn--Minkowski and reverse isoperimetric inequalities for dual quermassintegrals}, Adv. Math. 480 (2025): 110456.

\bibitem[Sar15]{Sar15} C. Saroglou, 
\textit{Remarks on the conjectured log-Brunn-Minkowski inequality}, Geom. Dedicata 177, no. 1 (2015): 353--365.

\bibitem[Sch14]{Sch14} R. Schneider, 
\textit{Convex bodies: the Brunn-Minkowski theory}, volume 151 of Encyclopedia of Mathematics and its Applications. Cambridge University Press, Cambridge, second expanded edition, 2014.

\bibitem[XZ22]{XZ22} D. Xi, Z. Zhang, 
\textit{The $L_p$ Brunn-Minkowski inequalities for dual quermassintegrals}, Proc. Amer. Math. Soc. 150 (7) (2022): 3075--3086.

\end{thebibliography}
\end{document}